\documentclass[manuscript,screen]{acmart}
\AtBeginDocument{%
  \providecommand\BibTeX{{%
    \normalfont B\kern-0.5em{\scshape i\kern-0.25em b}\kern-0.8em\TeX}}}
\setcopyright{none}
\copyrightyear{2022}
\acmYear{2022}
\acmDOI{XXXXXXX.XXXXXXX}
\usepackage{setspace}
\usepackage{array}
\usepackage{endnotes}
\let\footnote=\endnote

\usepackage{cases}
\usepackage{soul}
\usepackage{color}
\usepackage{amsmath}
\usepackage{graphicx}
\usepackage{enumitem}

\usepackage{enumitem}
\usepackage{mathrsfs}
\usepackage{epsfig}
\usepackage{tikz}
\usepackage{mathtools}
\usepackage{geometry}
\usepackage{longtable}
\usepackage{supertabular}
\usepackage{url}
\usepackage{blkarray}
\usepackage{enumitem}
\usepackage{multirow}
\usepackage{setspace}
\usepackage{float}
\usepackage[ruled]{algorithm2e}
\usepackage[title]{appendix}
\usetikzlibrary{decorations.pathreplacing}
\usetikzlibrary{shapes.misc, positioning}
\usepackage{bbm}
\usepackage{booktabs}
\usepackage{environ}
\usepackage{cancel}
\makeatletter
\newsavebox{\measure@tikzpicture}
\NewEnviron{scaletikzpicturetowidth}[1]{%
  \def\tikz@width{#1}%
  \begin{lrbox}{\measure@tikzpicture}%
  \BODY
  \end{lrbox}%
  \pgfmathparse{#1/\wd\measure@tikzpicture}%
  \BODY
}
\makeatother

\newtheorem*{remark}{Remark}
\newtheorem{assumption}{Assumption}

\newcommand{\exclude}[1]{}

\newcommand{\inP}{\mbox{$\,\stackrel{\scriptsize{\mbox{p}}}{\rightarrow}\,$}}
\newcommand{\as}{\mbox{$\,\stackrel{\scriptsize{\mbox{wp1}}}{\rightarrow}\,$}}
\newcommand{\inD}{\mbox{$\,\stackrel{\scriptsize{\mbox{\rm d}}}{\rightarrow}\,$}}

\newcommand*\diff{\mathop{}\!\mathrm{d}}

\newcommand{\tabincell}[2]{\begin{tabular}{@{}#1@{}}#2\end{tabular}} 
\renewcommand{\mathfrak}{\mathcal}
\newcommand{\rp}[1]{{\color{black}{{#1}}\color{black}}}

\begin{document}
\title[Overlapping Batch Confidence Intervals on Statistical Functionals]{Overlapping Batch Confidence Intervals on Statistical Functionals Constructed from Time Series:  Application to Quantiles, Optimization, and Estimation}

\author{Ziwei Su}
\email{su230@purdue.edu}
\affiliation{%
  \institution{Department of Statistics, Purdue University, West Lafayette, IN; Industrial Engineering and Management Sciences, Northwestern University, Evanston, IL}
  \streetaddress{150 North University Street}
  \city{West Lafayette}
  \state{Indiana}
  \country{USA}
  \postcode{47907}
  }
\author{Raghu Pasupathy}
\email{pasupath@purdue.edu}
\affiliation{%
  \institution{Department of Statistics, Purdue University, West Lafayette, IN, USA; Department of Computer Science \& Engineering, IIT Madras, Chennai}
  \streetaddress{150 North University Street}
  \city{West Lafayette}
  \state{Indiana}
  \country{India}
  \postcode{47907}
}

\author{Yingchieh Yeh}
\affiliation{
  \institution{Institute of Industrial Management, National Central University}
  \city{Taoyuan}
  \country{Taiwan}}
\email{yeh@mgt.ncu.edu.tw}

\author{Peter W. Glynn}
\affiliation{%
  \institution{Department of Management Science and Engineering, Stanford University}
  \city{Palo Alto}
  \state{CA}
  \country{USA}
}
\email{glynn@stanford.edu}
\renewcommand{\shortauthors}{Su et al.}

\begin{abstract} \rp{
{We propose a general purpose confidence interval procedure (CIP) for statistical functionals constructed using data from a stationary time series. The procedures we propose are based on derived distribution-free analogues of the $\chi^2$ and Student's $t$ random variables for the statistical functional context, and hence apply in a wide variety of settings including quantile estimation, gradient estimation, M-estimation, CVAR-estimation, and arrival process rate estimation, apart from more traditional statistical settings. Like the method of subsampling, we use overlapping batches of time series data to estimate the underlying variance parameter; unlike subsampling and the bootstrap, however, we assume that the implied point estimator of the statistical functional obeys a central limit theorem (CLT) to help identify the weak asymptotics (called OB-x limits, x=I,II,III) of batched Studentized statistics. The OB-x limits, certain functionals of the Wiener process parameterized by the size of the batches and the extent of their overlap, form the essential machinery for characterizing dependence, and consequently the correctness of the proposed CIPs. The message from extensive numerical experimentation is that in settings where a functional CLT on the point estimator is in effect, using \emph{large overlapping batches} alongside OB-x critical values yields confidence intervals that are often of significantly higher quality than those obtained from more generic methods like subsampling or the bootstrap.  We illustrate using examples from CVaR estimation, ARMA parameter estimation, and NHPP rate estimation; R and MATLAB code for OB-x critical values is available at~\texttt{web.ics.purdue.edu/$\sim$pasupath}. } }
\end{abstract}


\maketitle


\section{Introduction} \rp{Let $\{X_j, j \geq 1\}$ be an $S$-valued discrete-time, stationary, observable stochastic process having distribution $P$, and let $\theta: \mathcal{W} \to \Theta \subseteq \mathbb{R}$ denote a known \emph{statistical functional} (see~\cite[Chapter 6]{1980ser} or~\cite[Section 6.2]{1999leh}) defined on the space $\mathcal{W}$ of probability measures. In this paper, we propose an overlapping batch (OB) confidence interval procedure (CIP) to construct an interval $I_n \subset \mathbb{R}, I_n \in \sigma(X_1,X_2,\ldots,X_n)$ such that \begin{equation}\label{confint}\lim_{n \to \infty} \mathbb{P}(\omega: \theta(P) \in I_n(\omega)) = 1-\alpha,\end{equation} for any specified constant $\alpha \in (0,1)$. Importantly, since $\{X_n, n \geq 1\}$ is a time series, the dependence between random variables $X_n, n \geq 1$ is a key feature requiring careful treatment.}

\begin{remark} The initial segment $X_j, 1 \leq j \leq n$ of the observable process $\{X_j, j \geq 1\}$ is assumed to be a ``collected dataset'' or the output of a simulation that is exogenous to problem at hand. We assume no facility for variance reduction, e.g., by changing the measure governing the process $\{X_j, j \geq 1\}$, as is sometimes possible in simulation settings. See~\cite{2012chunak,2014nak,2018donnak,2011nak,2016graetal,2014donnak} for variance reduced confidence interval problems in the quantile context.
\end{remark}

\subsection{Motivation}\label{sec:motitaion}
\rp{Statistical functionals subsume a variety of interesting quantities arising in modern data settings, and are thus useful mathematical objects on which to construct confidence intervals. Consider, for instance, the following examples of statistical functionals. As a matter of notation, whenever relevant, $X: \Omega \to \mathcal{S}$ is an $\mathcal{S}$-valued random variable distributed according to $\mu$ and ``obtainable'' from the measure $P$ governing the observed stationary time series $\{X_j, j \geq 1\}$), and $s \in \mathcal{S}$ denotes an ``outcome" in $\mathcal{S}$. \setlist{nolistsep}\begin{itemize}[noitemsep] \item[(a)] Expectation. For $g: \mathcal{S} \to \mathbb{R}$, define the expectation $$\theta(P) :=  \mathbb{E}[g(X)] = \int_\mathcal{S} g(s)\,d\mu.$$   
\item[(b)] Quantile. For $g: \mathcal{S} \to \mathbb{R}$ and $\gamma \in (0,1)$, define the $(1-\gamma)$-quantile $$\theta(P) := \inf_{y \in \mathbb{R}} \{ \mu(g(X) \leq y) \geq 1-\gamma\}.$$ \item[(c)] Finite Difference Approximation. For $g: \mathcal{X} \times \mathcal{S} \to \mathbb{R}$ where $\mathcal{X} := \mbox{dom}(g) \subseteq \mathbb{R}^d$, define the finite-difference approximation of the directional derivative (assumed to exist) at $x \in \mbox{int}(\mathcal{X})$ along $u \in \mathbb{R}^d$: $$\theta(P) = \theta_{x,u,\epsilon}(P) :=  \frac{1}{\epsilon} \left(\mathbb{E}[g(x+\epsilon u,X) - g(x,X)]\right) = \frac{1}{\epsilon} \left(\int_{\mathcal{S}} \left(g(x+\epsilon u,s) - g(x,s)\right)\, d \mu\right).$$ \item[(d)] General Optimization. For $g: \mathcal{X} \times \mathcal{S} \to \mathbb{R}$, where $(\mathcal{X},d)$ is a metric space, $$\theta(P) := \inf_{x \in \mathcal{X}}  \mathbb{E}[g(x,X)] = \inf_{x \in \mathcal{X}} \int_\mathcal{S} g(x,s)\, d\mu.$$ \item[(e)]  Root Finding. For $g: \mathcal{X} \times \mathcal{S} \to \mathbb{R}$, $\theta(P) = \theta \in \mathcal{X}$ is such that  $$\int_\mathcal{S} g(\theta,s) \, d\mu = 0.$$ \item[(f)] Conditional Value at Risk (CVaR). For $0 < \gamma < 1$, and $g: \mathcal{S} \to \mathbb{R}$,  \begin{align*}\theta(P) = \theta_{\gamma}(P) &:= \mathbb{E}\left[ g(S) \, \vert \, g(S) > q_{\gamma} \right] = \frac{1}{P(g(S)>q_{\gamma})}\left(\int_{\mathcal{S}} g(s)\mathbb{I}\{g(s) > q_{\gamma}\}\,d\mu\right),\end{align*} where the $\gamma$-quantile $q_{\gamma} := \inf\left\{t \in \mathbb{R}: \int_{\mathcal{S}} \mathbb{I}\{g(s) \leq t\}\,d\mu \geq \gamma\right\}.$  
\item[(g)] ARMA($p$,$q$). The ARMA$(p,q)$ process is a discrete-time real-valued process $\{Y_t, t \geq 1\}$ having $p$ ``autoregressive'' parameters, $\phi_j, j=1,2,\ldots,p$, and $q$ ``moving average'' parameters, $\theta_j, j=1,2,\ldots,q$, and is expressed as $$Y_t = c + \sum_{j=1}^p \phi_j Y_{t-j} + \sum_{j=1}^q \theta_j \epsilon_{t-j} + \epsilon_t,$$ where $\{\epsilon_j, j \geq 1\}$ are independent and identically distributed (iid) random variables having mean zero and unit variance. Given observations $y_i, i=1,2,\ldots,n$ of the process $\{Y_t, t \geq 1\}$, the estimators $\hat{c}, \hat{\phi}_j, j=1,2,\ldots,p$ and $\hat{\theta}_j, j=1,2,\ldots,q$ of the parameters $c, \phi_j, j=1,2,\ldots,p$ and $\theta_j, j=1,2,\ldots,q$, are statistical functionals that can be estimated by minimizing the sum of squared residuals~\cite{2008crycha}: $$\underset{c,\{\phi_j, 1 \leq j \leq p\},\{\theta_j, 1\leq j \leq q\}}{\mbox{minimize: }} \,\, \sum_{i=p+1}^n \hat{\epsilon}_{i}^2,$$ where the residuals are given by: $$\hat{\epsilon}_i := \begin{cases} y_i - c - \sum_{j=1}^p \phi_jy_{i-j} - \sum_{j=1}^q \theta_j\hat{\epsilon}_{i-j} & \mbox{ if }i \geq p+1, \\ 0 & \mbox{ otherwise.} \end{cases} $$       \end{itemize} }   

In addition to the above examples, a wide variety of quantities arising within classical statistics, e.g., higher order moments, ratio of moments, \rp{clusters obtained through $k$-means clustering}, $\alpha$-trimmed mean, Mann-Whitney functional, and the \rp{simplicial depth} functional are all statistical functionals, making the question of constructing confidence intervals on statistical functionals of wide interest. (See~\cite[Chapter 7]{2013wel} and~\cite[Chapter 6]{1999leh} for other examples and a full treatment of statistical functionals.)

\begin{remark}
Whereas $\theta(P)$ in some of the examples listed above are naturally $\mathbb{R}^d$-valued with $d>1$, e.g., (g), the treatment in this paper is entirely real-valued, that is, $\theta(P) \in \mathbb{R}$. Extending our methods from $\mathbb{R}$ to $\mathbb{R}^d$ is straightforward but further extension into a function space involves non-trivial technical aspects.  
\end{remark}

\subsection{Notation and Terminology}\label{sec:notterm} (i) $\mathbb{N}$ refers to the set $\{1,2, \ldots,\}$ of natural numbers. (ii) $\mathbb{I}_{A}(x)$ is the indicator variable taking the value $1$ if $x \in A$ and $0$ otherwise. Also, depending on the context, we write $\mathbb{I}(A)$ where $\mathbb{I}(A)=1$ if the event $A$ is true and $0$ otherwise.  (iii) $Z(0,1)$ denotes the standard normal random variable, and $\chi^2_{\nu}$ refers to the chi-square random variable with $\nu$ degrees of freedom. (iv) $\{W(t), t  \geq 0\}$ refers to the standard Wiener process~\cite[Section 37]{1995bil}, and  $\{B(t), t \in [0,T]\}, \rp{B(t) = W(t) - \frac{t}{T}W(T)}$ refers to the Brownian bridge on $[0,T]$. (v) For a random sequence $\{X_n, n \geq 1\}$, we write $X_n \as X$ to refer to almost sure convergence, $X_n \inP X$ to refer to convergence in probability, and $X_n \inD X$ to refer to convergence in distribution (or weak convergence). (vi) We write $X \overset{\scriptsize{\mbox{d}}}{=} Y$ to mean that random variables $X$ and $Y$ have the same distribution.  (vii) The empirical measure $P_n$ constructed from the sequence $\{X_n, n \geq 1\}$ is given by $P_n(A) = n^{-1} \sum_{j=1}^n \mathbb{I}_{A}(X_j)$ for appropriate sets $A$.   

\subsection{Organization of the Paper}\label{sec:org} In the following section, we discuss literature on confidence intervals with a view toward providing perspective on how the proposed methods fit within the existing literature. This is followed by Section~\ref{sec:mainidea} where we present the main idea underlying the interval estimators we propose, along with a synopsis of results. Section~\ref{sec:mathprelim} includes key assumptions, followed by Section~\ref{sec:ob1}--\ref{sec:ob3} which present the theorems corresponding to the OB-I, OB-II, and OB-III limits. \rp{In Section~\ref{sec:postscript}, we present brief discussion on some implementation questions that we consider important.} We end with Section~\ref{sec:numerical} where numerical illustration using three different contexts illustrate the effectiveness of using large batch OB-I and OB-II confidence intervals.   

\section{Existing Literature, Perspective, and contribution}\label{sec:litrev}

\rp{In this section, we provide an overview of CIPs in general through a taxonomy that categorizes CIPs into those that assume a CLT is in effect and those that do not. We discuss CLT-based methods, followed by further perspective and a summary of the current paper's position within this landscape. (We include a concise description of the two most famous non-CLT-based methods, subsampling and bootstrapping, in Appendix~\ref{sec:nonCLT}.)}

\subsection{CLT-based Methods}\label{sec:taxonomy}
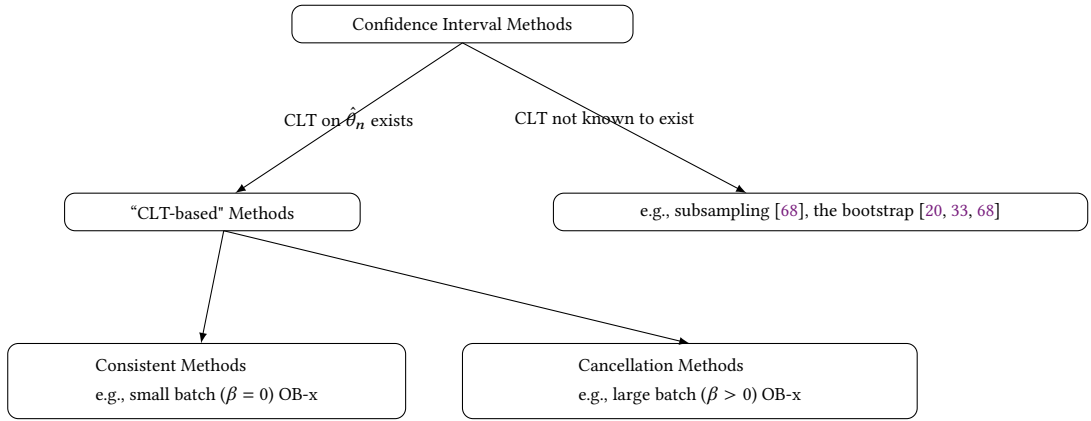
\begin{figure}[htb]\label{fig:taxonomy}
\begin{center}
\begin{tikzpicture}
\draw[rounded corners] (0.65\textwidth, 0) rectangle (0.95\textwidth, 0.5) node[black,midway]{\footnotesize Confidence Interval Methods};
\draw[thin,-{latex}] (0.8\textwidth,0) -- (0.6\textwidth,-2) node[black,midway,xshift=0cm] {\footnotesize CLT on $\hat{\theta}_n$ exists};
\draw[thin,-{latex}] (0.8\textwidth,0) -- (1.05\textwidth,-2) node[black,midway,xshift=0cm] {\footnotesize CLT not known to exist};
\draw[rounded corners] (0.88\textwidth, -2.5) rectangle (1.35\textwidth, -2) node[black,midway]{\footnotesize e.g., subsampling~\cite{1999polromwol}, the bootstrap~\cite{1997davhin,1999polromwol,1998efrtib}};
\draw[rounded corners] (0.45\textwidth, -2.5) rectangle (0.71\textwidth, -2) node[black,midway]{\footnotesize ``CLT-based" Methods};
\draw[thin,-{latex}] (0.59\textwidth,-2.5) -- (0.57\textwidth,-4) node[black,midway,xshift=0cm] {};
\draw[rounded corners] (0.4\textwidth, -5) rectangle (0.75\textwidth, -4) node[black,align=left,midway]{\footnotesize Consistent Methods \\ \footnotesize e.g., small batch ($\beta=0$) OB-x};
\draw[thin,-{latex}] (0.59\textwidth,-2.5) -- (1\textwidth,-4) node[black,midway,xshift=0cm] {};
\draw[rounded corners] (0.8\textwidth, -5) rectangle (1.2\textwidth, -4) node[black,align=left,midway]{\footnotesize Cancellation Methods \\ \footnotesize e.g., large batch ($\beta>0$) OB-x};
\end{tikzpicture}
\end{center}

    \caption{A taxonomy of methods for constructing confidence intervals on statistical functionals. Consistent methods construct a consistent estimator of the variance constant $\sigma$, while cancellation methods allow the use of large batches and construct ratio estimators that ``cancel out" the variance constant $\sigma$. }\label{fig:taxonomy}
\end{figure}
Analogous to the taxonomy~\cite{1985glyigl} of CIPs on the steady-state mean of a real-valued process, it is instructive to categorize CIPs for statistical functionals based on whether a central limit theorem of the form \begin{equation}\label{firstclt}\sqrt{n}(\hat{\theta}_n - \theta(P)) \inD \sigma Z(0,1)\end{equation} exists. In~\eqref{firstclt}, $\hat{\theta}_n$ is an implied point estimator of $\theta(P)$ constructed from the time series $\{X_n, n \geq 1\}$, $Z(0,1)$ is the standard normal random variable, and $\sigma \in (0,\infty)$ is an unknown parameter often called the variance constant. Further, and as depicted in Figure~\ref{fig:taxonomy}, a CIP that assumes~\eqref{firstclt} may either be a \emph{consistent method} by which we mean that the CIP constructs another observable process $\{\hat{\sigma}_n, n \geq 1\}$ from $\{X_n, n\geq 1\}$ to consistently estimate $\sigma$, that is, \begin{equation}\label{consissigma}\hat{\sigma}_n \inP \sigma \mbox{ as } n \to \infty;\end{equation} or \rp{a \emph{cancellation method} by which we mean that the CIP constructs a process $\{Y_n, n\geq 1\}$ such that \begin{equation}\label{jtconvcancel}(\sqrt{n}(\hat{\theta}_n - \theta(P)),Y_n) \inD (\sigma Z(0,1),\sigma Y) \mbox{ as } n \to \infty,\end{equation} and $Y$ is a well-defined non-vanishing random variable whose distribution is free of unknown quantities, e.g., $\sigma$ and $\theta(P)$. (The canonical $\sqrt{n}$ scaling in~\eqref{jtconvcancel} can be generalized to other scalings, as considered in~\cite{1992glywhi}.) } 

In consistent methods, since~\eqref{firstclt} and~\eqref{consissigma} hold, Slutsky's theorem~\eqref{thm:slutsky} assures us that an asymptotically valid two-sided $(1-\alpha)$ confidence interval on $\theta(P)$ is $$(\hat{\theta}_n - z_{1-\alpha/2}\frac{\hat{\sigma}_n}{\sqrt{n}}, \hat{\theta}_n + z_{1-\alpha/2}\frac{\hat{\sigma}_n}{\sqrt{n}}),$$ where $z_{1-\alpha/2}$ is the $1- \alpha/2$ quantile of the standard normal distribution. It is in this sense that a consistent method essentially reduces the confidence interval construction problem into the often nontrivial problem~\cite{2012chunak,1996gly,2007asmgly} of consistently estimating the variance parameter $\sigma$. Various consistent methods exist in the steady-state mean context. For example, the regenerative method~\cite{1977cralem,1978igl}, the spectral procedure~\cite{1950bar,1967wel,1991dam,1994dam,1995dam} with certain restrictions on the bandwidth, and the batch means procedure where the variance parameter is estimated using one of various well-established methods, e.g., nonoverlapping batch means (NBM)~\cite{2007aleetal}, overlapping batch means (OBM)~\cite{2007aleetal}, Cram\'{e}r-von Mises (CvM) estimator~\cite{2007aleetal}, provided the batch size tends to infinity in a way that the batch size expressed as a fraction of the total data size tends to zero. See~\cite{2007aleetal,2009aktalegolwil} and references therein for a thorough account on estimating the variance parameter associated with a steady-state real-valued process.

In contrast to consistent methods, cancellation methods are based on the important idea that $\sigma$ need not be estimated consistently to construct a valid confidence interval on $\theta(P)$. This seems to have been first observed in the seminal account~\cite{1983sch} introducing standardized time series in the context of constructing confidence intervals on the steady state mean. Specifically, in cancellation methods, since~\eqref{firstclt} and~\eqref{jtconvcancel} hold, and $Y$ is non-vanishing, applying the continuous mapping theorem~\cite{1999bil} leads to ``cancellation'' of $\sigma$ in the sense that \begin{equation}\label{cancellation} \frac{\sqrt{n}(\hat{\theta}_n - \theta(P))}{Y_n}  \inD \frac{\cancel{\sigma} \, Z(0,1)}{\cancel{\sigma} \, Y}, \end{equation} leading to the two-sided $(1-\alpha)$ confidence interval $$(\hat{\theta}_n - y_{\alpha/2}\frac{Y_n}{\sqrt{n}}, \hat{\theta}_n + y_{1-\alpha/2}\frac{Y_n}{\sqrt{n}}),$$ where $y_{q}$ is the $q$-quantile of $Z(0,1)/Y$. If constructing a consistent estimator of $\sigma$ is the principal challenge in consistent methods, selecting $Y_n$ and characterizing $Y$ turns out to be the principal challenge in cancellation methods. Cancellation methods have been  studied~\cite{1983sch,1990golsch,1990glyigl,1991mun,2006calnak} in the context of constructing confidence intervals on the steady-state mean, and more recently for quantiles --- see the exceptionally well-written articles~\cite{2013calnak,2020huinak}.    

\subsection{Further Perspective and Summary of Contribution}

The uniqueness of any CIP (including subsampling, the bootstrap, and what we propose here) stems from the manner in which the procedure approximates the sampling distribution of its chosen statistic. So, while subsampling uses the empirical cdf $L_n$ in~\eqref{subsamplingstat2} formed from subsamples, and the bootstrap uses resampling, the methods proposed in this paper approximate the sampling distribution of the Studentized statistic $(\hat{\theta}_n - \theta(P))/\hat{\sigma}_n$ by characterizing its weak limit. In particular, we assume the existence of a functional CLT governing $\hat{\theta}_n$ and exploit the resulting structure to characterize the weak limit of $(\hat{\theta}_n - \theta(P))/\hat{\sigma}_n$. 

To be clear, neither subsampling nor the bootstrap assume a CLT on $\hat{\theta}_n$, and this is their strength. \rp{(Specifically, the bootstrap and subsampling only assume the existence of the scaled weak limit on $\hat{\theta}_n$; they do not assume, for instance, that $J(P)$ in~\eqref{subsamplingstat1} is standard normal.)} However, our argument is that there exist numerous important contexts where a functional CLT on $\hat{\theta}_n$ holds and can be usefully exploited if we can identify the weak limit of the statistic in use. For example, vis-\`{a}-vis subsampling, knowledge of the weak limit allows replacing the empirical quantiles $c_{n,q}$ in~\eqref{subsamplingconf} by their limiting counterparts, in the process allowing to dispense with subsampling's key stipulation that batch sizes be small, that is, $m_n/n \to 0$.

To further clarify, we now provide a summary of contribution.

\begin{enumerate} \item This work presents CLT-based overlapping batch CIPs for constructing confidence intervals on \emph{statistical functionals}. There exists a well-developed literature on CLT-based OB CIPs for the steady-state mean, and more recently for quantiles, but the only treatment of statistical functionals through CLT-based methods that we are aware of is~\cite[Section 2.4]{1991mun}.  
\item We derive the weak limits (called OB-x limits, x=I,II,III) of the statistic underlying each of the proposed OB CIPs. Of these, the OB-II limit and its bias-correction factor (Theorem~\ref{thm:ob2largebatch}) have not appeared in the literature even in the steady-state mean context to the best of our knowledge; OB-II might prove to be especially relevant in computationally intensive settings. The OB-I and OB-III limits (Theorem~\ref{thm:ob1largebatch} and Theorem~\ref{thm:ob3largebatch}, respectively) have appeared in the literature but in the steady-state mean~\cite{2009aktalegolwil,2007aleetal} and the quantile~\cite{2013calnak} contexts. The asymptotic moment expression for the OB-I limit (Theorem~\ref{thm:OB1moments}) has not appeared elsewhere but the corresponding result for the special case of fully overlapping batches in the steady-state mean context appeared in~\cite{1995dam}.

\item To aid future investigation of computationally intensive contexts, our analysis of overlapping batches is general in the sense that it introduces an offset parameter $d_n$ whose value connotes the extent of batching, e.g., \rp{$d_n=1$ connotes fully overlapping batches and $d_n \geq m_n$ connotes non-overlapping batches with $d_n > m_n$ corresponding to what has been called \emph{spaced batch means}~\cite{1991foxgolswa}. We shall see (Theorem~\ref{thm:OB1moments}) that the effect of the extent of overlap features prominently in the asymptotic variance of the variance estimator.}

\item Extensive numerical experimentation over a variety of applications indicates that cancellation methods resulting from the use of large batches, that is, when $m_n/n \to \beta >0$, exhibits behavior that is consistently better. Aspects responsible for such better behavior are not yet fully understood and should form the topic of future investigation. 

\item We provide access to code (that includes a critical value calculation module for OB-I, OB-II, and OB-III) for constructing confidence intervals on a statistical functional using our recommended OB-x methods. 

\end{enumerate}

\section{Main Idea and Synopsis of Results}\label{sec:mainidea}

To set the stage for precisely describing the proposed confidence interval procedure, consider partitioning the available ``data'' $X_1,X_2, \ldots,X_n$ into $b_n$ possibly overlapping batches each of size $m_n$ as shown in Figure~\ref{batching}. The first of these batches consisting of observations $X_1, X_2, \ldots, X_{m_n}$, the second consisting of observations $X_{d_n + 1}, X_{d_n + 2}, \ldots, X_{d_n + m_n}$, and so on, and the last batch consisting of observations $X_{(b_n-1)d_n+1}, X_{(b_n-1)d_n+2}, \ldots, X_{n}$. The quantity $d_n \geq 1$ represents the offset between batches, with the choice $d_n=1$ corresponding to ``fully-overlapping'' batches and any choice $d_n \geq m_n$ corresponding to ``non-overlapping'' batches. Notice then that the offset $d_n$ and the number of batches $b_n$ are related as \begin{equation}\label{offsetbatchrel}d_n = \frac{n-m_n}{b_n-1}.\end{equation} Suppose that the batch  size $m_n$ and the number of batches $b_n$ are chosen so that the following limits exist:

\begin{equation} \label{keylimits} \lim_{n \to \infty} \frac{m_n}{n} = \beta \in [0,1); \quad \lim_{n \to \infty} b_n = b_{\infty} \in \{2,3, \ldots, \infty\}.\end{equation} 

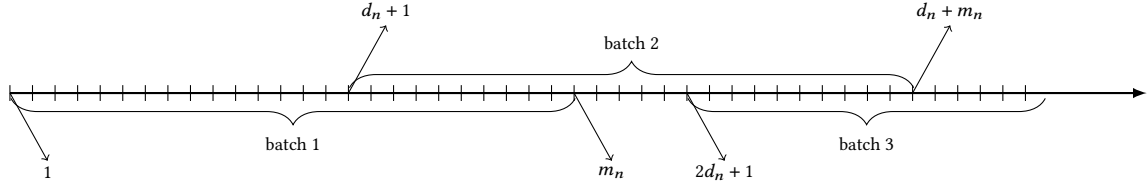
\begin{figure}[h]
\begin{tikzpicture}
\draw[thick,-{latex}] (0,0) -- (\textwidth,0);
\foreach \x in {0,0.3,...,13.75}
  \draw (\x,3pt) -- (\x,-3pt); 
\draw [thin, black,decorate,decoration={brace,amplitude=10pt,mirror},xshift=0.4pt,yshift=-2pt](0,0) -- (7.5,0) node[black,midway,yshift=-0.6cm] {\footnotesize batch $1$};
\draw [thin, black,decorate,decoration={brace,amplitude=10pt},xshift=0.4pt,yshift=2pt](4.5,0) -- (12,0) node[black,midway,yshift=0.6cm] {\footnotesize batch $2$};
\draw [thin, black,decorate,decoration={brace,amplitude=10pt,mirror},xshift=0.4pt,yshift=-2pt](9,0) -- (13.75,0) 
node[black,midway,yshift=-0.6cm] {\footnotesize batch $3$};
\draw[arrows=->,line width=.4pt](0,0)--(0.5,-0.9) node[black,yshift=-0.15cm] {\footnotesize $1$};
\draw[arrows=->,line width=.4pt](7.5,0)--(8,-0.9) node[black,yshift=-0.15cm] {\footnotesize $m_n$};
\draw[arrows=->,line width=.4pt](4.5,0)--(5,0.9) node[black,yshift=0.2cm] {\footnotesize $d_n +1$};
\draw[arrows=->,line width=.4pt](12,0)--(12.5,0.9) node[black,yshift=0.2cm] {\footnotesize $d_n + m_n$};
\draw[arrows=->,line width=.4pt](9,0)--(9.5,-0.9) node[black,yshift=-0.15cm] {\footnotesize $2d_n + 1$};
\end{tikzpicture}
\caption{The figure depicts partially overlapping batches. Batch 1 consisting of observations $X_j, j = 1,2,\ldots,m_n$; batch 2 consisting of observations $X_j, j=d_n + 1, d_n + 2, \ldots, d_n + m_n$, and so on, with batch $i$ consisting $X_j, j=(i-1)d_n + 1, (i-1)d_n + 2, \ldots,(i-1)d_n+m_n.$ There are thus $b_n := d_n^{-1}(n-m_n) + 1$ batches in total, where $n$ is the size of the dataset.} \label{batching}
\end{figure} 

Note that $\beta = 0$ and $b_{\infty}= \infty$ are allowed in~\eqref{keylimits}. We will sometimes refer to $\beta$ as the \emph{asymptotic batch size} and to $b_{\infty}$ as the \emph{asymptotic number of batches.} Also, we will refer to $\beta=0$ as the \emph{small batch} regime, and to $\beta>0$ as the \emph{large batch} regime.

\subsection{``Centering'' the Confidence Interval}\label{sec:ptest}
Suppose we have at our disposal a method to construct a point estimator $\hat{\theta}(\{X_j, \ell \leq j \leq u\})$ of $\theta(P)$ using any batch $(X_{\ell}, X_{\ell+1}, \ldots, X_{u})$, $\ell, u \in \{1,2,\ldots,n\}$ of consecutive observations from the available data $X_j, 1 \leq j \leq n$. For now, we place no restrictions on $\hat{\theta}(\{X_j, \ell \leq j \leq u\})$ but a natural choice for $\hat{\theta}(\{X_j, \ell \leq j \leq u\})$, especially in the non-parametric setting, is the ``plug-in'' estimator $\theta(P_{\ell,u}),$ where $P_{\ell,u}$ is the empirical measure constructed from $X_j, \ell \leq j \leq u$. For example, when $X_j, 1 \leq j \leq n$ are real-valued and $\theta(P)$ is the population mean $\mathbb{E}[X_1] = \int x\,  P(\diff x)$, the point estimator $\hat{\theta}(\{X_j, \ell \leq j \leq u\})$  is the sample mean of the observations $X_j, \ell \leq j \leq u$; and likewise, when $\theta(P)$ is the $\gamma$-quantile $\min\{x: P(X_1 \leq x) \geq \gamma\}$ of $X_1$, the natural choice for the point estimator is the empirical quantile $F_{\ell,u}^{-1}(\gamma) : = \min\{x: F_{\ell,u}(x) \geq \gamma\}$, where $F_{\ell,u}(x) = (u-\ell+1)^{-1}\sum_{j=\ell}^u \mathbb{I}_{\{X_j \leq x\}}, x \in \mathbb{R}$ is the usual empirical cumulative distribution function (cdf) constructed from the observations $X_j, \ell \leq j \leq u$. 

When $\ell=1$ and $u=n$, that is, all available observations are utilized in constructing the point estimator of $\theta(P)$, we obtain what is often called the \emph{sectioning estimator}~\cite{2014nak}, given special notation here since we heavily invoke this estimator throughout the rest of the paper: \begin{equation}\label{est:section} \hat{\theta}_n := \hat{\theta}(\{X_j, 1 \leq j \leq n\}). \end{equation} The  \emph{asymptotic variance parameter} $\sigma^2$, assumed to exist and defined as \begin{equation}\label{varpardef}\sigma^2 : = \lim_{n \to \infty} \mathbb{E}\left[ \left(\sqrt{n} (\hat{\theta}_n - \theta(P))\right)^2\right],\end{equation} will play a crucial role in our later analysis. Also, owing to the manner in which we will construct batches, we use special notation for the point estimators constructed from the observations $(X_{(i-1)d_n + 1}, X_{(i-1)d_n + 2}, \ldots, X_{(i-1)d_n + m_n})$ in the $i$-th batch (see Figure~\ref{batching}):  \begin{equation}\label{batchest} \hat{\theta}_{i,m_n} : = \hat{\theta}(\{X_j, (i-1)d_n + 1 \leq j \leq (i-1)d_n + m_n\}), i = 1,2, \ldots, b_n\end{equation} and $b_n = \frac{n-m_n}{d_n} + 1.$ 

We shall see shortly that the sectioning estimator appearing in~\eqref{est:section} is a candidate for centering the confidence interval that we construct. An alternative to the sectioning estimator is the \emph{batching estimator}~\cite{2014nak}, obtained by averaging the point estimators $\hat{\theta}_{i,m_n}, i = 1,2,\ldots,b_n$, that is, \begin{equation}\label{est:batch} \bar{\theta}_n := \frac{1}{b_n} \sum_{i=1}^{b_n} \hat{\theta}_{i,m_n}. \end{equation} 

The sectioning and batching point estimators are the two natural choices for ``centering'' the confidence intervals on $\theta(P)$. We will see that confidence intervals constructed with the batching estimator might be especially useful in computationally intensive contexts. 

\subsection{Estimating the Variance Constant $\sigma^2$}\label{sec:varest}
Since the variance constant $\sigma^2$ (defined in~\eqref{varpardef}) is a measure of the inherent variability of the point estimator $\hat{\theta}_n$, $\sigma^2$'s estimation plays a key role in the confidence intervals we construct. The expression in~\eqref{varpardef} suggests that a natural estimator of $\sigma^2$ is the sample variance of $\hat{\theta}_{i,m_n}, i=1,2,\ldots,b_n$ defined in~\eqref{batchest}, after appropriate scaling:
\begin{equation} \label{varparest} \hat{\sigma}^2_{\mbox{\tiny OB-I}}(m_n,b_n) := \frac{1}{\kappa_1(\beta)}\frac{m_n}{b_n} \sum_{i=1}^{b_n} (\hat{\theta}_{i,m_n} - \hat{\theta}_n)^2, \quad \kappa_1(\beta) =  1- \beta,\end{equation} where $\beta$ defined in~\eqref{keylimits} is the limiting batch size. It will become clear from our later analysis that $\kappa_1(\beta)$ appearing in~\eqref{varparest} is a ``bias-correction'' constant introduced to make $\hat{\sigma}^2_{\mbox{\tiny OB-I}}(m_n,b_n)$ asymptotically unbiased.

Notice that the estimator $\hat{\sigma}^2_{\mbox{\tiny OB-I}}(m_n,b_n)$ of the variance constant $\sigma^2$ appearing in~\eqref{varparest} uses the sectioning estimator $\hat{\theta}_n$ when computing the sample variance. An alternative is to use the batching estimator $\bar{\theta}_n$ in place of the sectioning estimator to obtain the second candidate estimator of the variance constant $\sigma^2$:
\begin{equation} \label{varparestbatch} \hat{\sigma}^2_{\mbox{\tiny OB-II}}(m_n,b_n) := \frac{1}{\kappa_2(\beta,b_{\infty})}\frac{m_n}{b_n} \sum_{i=1}^{b_n} (\hat{\theta}_{i,m_n} - \bar{\theta}_n)^2,\end{equation} where, as we shall see in Theorem~\ref{thm:ob2largebatch}, the bias-correction constant has the more complicated form \begin{equation}\label{biascorrect2} \kappa_2(\beta,b_{\infty}) := \begin{cases} 1 & \beta =0; \\ 1 - 2\left(\min\{\frac{\beta}{1-\beta},1\}\right) + \frac{1}{\beta}\left(\min\{\frac{\beta}{1-\beta},1\}\right)^2 - \frac{2}{3}\frac{1-\beta}{\beta}\left(\min\{\frac{\beta}{1-\beta},1\}\right)^3 & \beta>0, b_{\infty} = \infty; \\
1- \frac{1}{b_{\infty}} - \frac{2}{b_{\infty}}\sum_{h=1}^{b_{\infty}} \left(1 - \frac{h}{b_{\infty}-1} \frac{1-\beta}{\beta}\right)^+(1-h/b_{\infty}) & b_{\infty} \in \mathbb{N}\setminus {1},\end{cases} \end{equation} and $b_{\infty}$ defined in~\eqref{keylimits} is the limiting number of batches. 

A third estimator $\hat{\sigma}^2_{\mbox{\tiny OB-III}}(m_n,b_n)$ of the variance constant $\sigma^2$ that we consider, called the weighted area estimator~\citep{1983sch,2007aleetal,1990golsch,1990golmeksch}, is given as follows:
\begin{equation} \label{varpareststs} \hat{\sigma}^2_{\mbox{\tiny OB-III}}(m_n,b_n) := \frac{1}{b_n} \sum_{i=1}^{b_n} A_{i,m_n}, \end{equation} where \begin{equation}\label{areaest} A_{i,m_n} := \left(\frac{1}{m_n} \sum_{j=1}^{m_n} f(\frac{j}{m_n}) \, \sigma  \, T_{i,m_n}(\frac{j}{m_n})\right)^2; \quad T_{i,m_n}(t) := \frac{\lfloor m_n t \rfloor \left(\hat{\theta}_{i,\lfloor m_nt \rfloor} - \hat{\theta}_{i,m_n}\right)}{\sigma \sqrt{m_n}}, t \in [0,1],\end{equation} and $f: [0,1] \to \mathbb{R}^+$ is a chosen weighting function that satisfies \rp{\begin{equation}\label{weightingcond} \mathbb{E}\left [ \left(\int_0^1 f(t) B_0(t) \, dt \right)^2 \right ]= 1; \mbox{ and } f \in C^2[0,1],\end{equation} where $B(t), t \in [0,1]\}$ is the Brownian bridge on $[0,1]$ (see Section~\ref{sec:notterm}) and $C^2[0,1]$ is the space of twice continuously differentiable functions on $[0,1]$.} The structure of the ``standardized time series'' $\{T_{i,m_n}(t), t \in [0,1]\}$ in~\eqref{areaest} hints at  why  $\hat{\sigma}^2_{\mbox{\tiny OB-III}}(m_n,b_n)/\sigma^2$ is an analogue of the classical chi-square random variable. Specifically, notice that $\{T_{\lfloor sm_n \rfloor,m_n}(t), t \in [0,1]\}$ for each $s \in [0, \infty)$ should converge weakly (as $m_n \to \infty$), modulo some regularity conditions, to the standard Brownian bridge $$B_s(t) := \left\{W(s+t) - W(s) - t(W(s+1) - W(s)), t \in [0,1]\right\}, \quad s \in [0,\infty).$$ Correspondingly, and since $\int_0^1 f(t) B_s(t) \overset{d}{=} Z(0,1)$ if $f$ is chosen as stipulated in~\eqref{weightingcond}, $A_{\lfloor sm_n \rfloor,m_n}$ should converge weakly to $\sigma^2Z^2(0,1) \overset{d}{=} \sigma^2\chi^2_1$, in effect justifying the weighted area estimator $\hat{\sigma}^2_{\mbox{\tiny OB-III}}(m_n,b_n)$. 

The weighted area estimator appearing in~\eqref{varpareststs} has been the topic of much research over  the last three decades in the context of estimating the variance constant of a stationary time series. See~\cite{2007aleetal,1983sch,1999folgol,1990glyigl,1990golsch} for a detailed account that includes treatment of other estimators of the variance constant.

\subsection{Structure of the Proposed Confidence Intervals}\label{sec:struct}
The proposed interval has the same elements as a classical confidence interval, namely: \begin{enumerate} \item[(A)] a ``centering'' variable, e.g., the sectioning estimator $\hat{\theta}_n \in  \mathbb{R}$, or the batching estimator $\bar{\theta}_n \in  \mathbb{R}$, as described in Section~\ref{sec:ptest}; \item[(B)] a point estimator of the asymptotic variance $\sigma^2$, e.g., $\hat{\sigma}^2_{\mbox{\tiny OB-x}}(m_n,b_n), {\rm{x = I, II, III}}$; and \item[(C)] a statistic  whose weak limit supplies the critical values associated with the confidence interval. \end{enumerate}  
Once the elements in (A)--(C) are specified, a $(1-\alpha)$ confidence interval on $\theta(P)$ can then be constructed in the usual way. 

For example, when the sectioning estimator $\hat{\theta}_n$ is used in (A), the variance estimator $\hat{\sigma}^2_{\mbox{\tiny OB-I}}(m_n,b_n)$ is used in (B), and the Studentized root    \begin{equation}\label{root1}  T_{\mbox{\tiny OB-I}}(m_n,b_n) := \frac{\sqrt{n}(\hat{\theta}_n - \theta(P))}{\hat{\sigma}_{\mbox{\tiny OB-I}}(m_n,b_n)} \inD T_{\mbox{\tiny OB-I}}(\beta,b_{\infty}),\end{equation} is used in (C), we obtain the (two-sided) confidence interval \begin{equation}\label{twosidedconf1} \left\{y \in \mathbb{R}: -t_{\mbox{\tiny OB-I},{\tiny 1-\frac{\alpha}{2}}}(\beta,b_{\infty}) \leq \frac{\sqrt{n}(\hat{\theta}_n - y)}{\hat{\sigma}_{\mbox{\tiny OB-I}}(m_n,b_n)} \leq t_{\mbox{\tiny OB-I},{\tiny 1-\frac{\alpha}{2}}}(\beta,b_{\infty})\right\}, \end{equation} where $$t_{\mbox{\tiny OB-I},q}(\beta,b_{\infty}) = \inf\left\{r: P(T_{\mbox{\tiny OB-I}}(\beta,b_{\infty}) \leq r) = q\right\}, \quad q \in (0,1)$$ is the $q$-quantile (or critical value) of the random variable $T_{\mbox{\tiny OB-I}}(\beta,b_{\infty})$. (A one-sided confidence interval analogous to~\eqref{twosidedconf1} is straightforward.)

Similarly, using the batching estimator $\bar{\theta}_n$ in (A), the variance estimator $\hat{\sigma}^2_{\mbox{\tiny OB-II}}(m_n,b_n)$ in (B), and the Studentized root   \begin{equation}\label{root2}  T_{\mbox{\tiny OB-II}}(m_n,b_n) := \frac{\sqrt{n}(\bar{\theta}_n - \theta(P))}{\hat{\sigma}_{\mbox{\tiny OB-II}}(m_n,b_n)} \inD T_{\mbox{\tiny OB-II}}(\beta,b_{\infty}),\end{equation} in (C), we obtain our second proposed (two-sided) confidence interval \begin{equation}\label{twosidedconf2} \left\{y \in \mathbb{R}:  -t_{\mbox{\tiny OB-II},{\tiny 1-\frac{\alpha}{2}}}(\beta,b_{\infty}) \leq \frac{\sqrt{n}(\bar{\theta}_n - y)}{\hat{\sigma}_{\mbox{\tiny OB-II}}(m_n,b_n)} \leq t_{\mbox{\tiny OB-II},{\tiny 1-\frac{\alpha}{2}}}(\beta,b_{\infty})\right\}, \end{equation} where $$t_{\mbox{\tiny OB-II},q}(\beta,b_{\infty}) = \inf\left\{r: P(T_{\mbox{\tiny OB-II}}(\beta,b_{\infty}) \leq r) = q\right\}, \quad q \in (0,1)$$ is the $q$-quantile (or critical value) of the random variable $T_{\mbox{\tiny OB-II}}(\beta,b_{\infty})$. 

And, finally, using the sectioning estimator $\hat{\theta}_n$ in (A), the variance estimator $\hat{\sigma}^2_{\mbox{\tiny OB-III}}(m_n,b_n)$ in (B), and the Studentized root    \begin{equation}\label{root3}  T_{\mbox{\tiny OB-III}}(m_n,b_n) := \frac{\sqrt{n}(\bar{\theta}_n - \theta(P))}{\hat{\sigma}_{\mbox{\tiny OB-III}}(m_n,b_n)} \inD T_{\mbox{\tiny OB-III}}(\beta,b_{\infty}),\end{equation} in (C), we obtain our third proposed (two-sided) confidence interval \begin{equation}\label{twosidedconf3} \left\{y \in \mathbb{R}: -t_{\mbox{\tiny OB-III},{\tiny 1-\frac{\alpha}{2}}}(\beta,b_{\infty}) \leq  \frac{\sqrt{n}(\hat{\theta}_n - y)}{\hat{\sigma}_{\mbox{\tiny OB-III}}(m_n,b_n)} \leq t_{\mbox{\tiny OB-III},{\tiny 1-\frac{\alpha}{2}}}(\beta,b_{\infty})\right\}, \end{equation} where $$t_{\mbox{\tiny OB-III},q}(\beta,b_{\infty}) = \inf\left\{r: P(T_{\mbox{\tiny OB-III}}(\beta,b_{\infty}) \leq r) = q\right\}, \quad q \in (0,1)$$ is the $q$-quantile (or critical value) of the random variable $T_{\mbox{\tiny OB-III}}(\beta,b_{\infty})$.

\begin{remark}
Sometimes $\theta$ is known to reside in a constrained set $\Theta \subset \mathbb{R}$, in which case the sectioning estimator $\hat{\theta}_n$, and all the batch estimators $\hat{\theta}_{i,m_n}, i = 1,2,\ldots,b_n$ should be suitably projected onto $\Theta$, as should the intervals in~\eqref{twosidedconf1},~\eqref{twosidedconf2}, and~\eqref{twosidedconf3}. \rp{This may cause a corresponding change in the weak limits along with the critical values, a line of investigation we do not pursue.} 
\end{remark}

The preceding discussion should emphasize that the Studentized root $T_{\mbox{\tiny OB-x}}(m_n,b_n), {\rm x = I, II, III}$ forms the essential element of the confidence intervals we propose. And, since the exact distribution of  $T_{\mbox{\tiny OB-x}}(m_n,b_n), {\rm x = I, II, III}$ is unknown in general, the outlined procedure approximates its distribution by the (purported) weak limit $T_{\mbox{\tiny OB-x}}(\beta,b_{\infty}), {\rm x = I, II, III}$.

\subsection{Synopsis of Results}\label{sec:confints}
\begin{table}\label{synopsis}
\caption{A synopsis of results. In the service of constructing confidence intervals on $\theta(P)$, we construct three Studentized roots $T_{\mbox{\tiny OB-x}}(m_n,b_n)$, x=I,II,III obtained using combinations of candidates for the point estimator of $\theta(P)$ and for the point estimator of $\sigma^2$. The three Studentized roots give rise to the OB-\mbox{x, x=I,II,III} weak limits, whose nature depends on the limiting batch size $\beta := \lim_{n\to \infty} m_n/n$ and the limiting number of batches $b_{\infty} := \lim_{n \to \infty} b_n$. Expressions for the weak limits $T_{\mbox{\tiny OB-x}}(\beta,b_{\infty})$, x=I,II,III appear in Theorems~\ref{thm:ob1largebatch}--\ref{thm:ob3largebatch}. Critical values for the OB-I and OB-II distributions appear on page 17 and page 24.}
\begin{tabular}{c c c l c}
\toprule
 & Centering Var. & Batch Regime & {\hspace{0.2in} Variance Estimator}  & {Statistic} \\
 &  & ($m_n/n \to \beta \overset{?}{=} 0)$ & $\hspace{0.25in} (\hat{\sigma}^2_{\mbox{\tiny OB-x}}(m_n,b_n) \to ?)$  & $\hspace{0.0in} (T_{\mbox{\tiny OB-x}}(m_n,b_n) \inD ?)$ \\
\midrule
\multicolumn{1}{l}{\multirow{2}{*}{OB-I}} & \multicolumn{1}{c}{\multirow{2}{*}{$\hat{\theta}_n$}} & $\beta=0$ & $\hat{\sigma}^2_{\mbox{\tiny OB-I}}(m_n,b_n) \inP \sigma^2$ &  $Z$ \\ 
\multicolumn{1}{r}{}  &  & $\beta > 0$          & $\hat{\sigma}^2_{\mbox{\tiny OB-I}}(m_n,b_n) \inD \sigma^2\chi^2_{\mbox{\tiny OB-I}}(\beta,b_{\infty})$  & $T_{\mbox{\tiny OB-I}}(\beta,b_{\infty})$     \\ \hline

\multicolumn{1}{l}{\multirow{2}{*}{OB-II}} & \multicolumn{1}{c}{\multirow{2}{*}{$\bar{\theta}_n$}} & $\beta=0$ & $\hat{\sigma}^2_{\mbox{\tiny OB-II}}(m_n,b_n) \inP \sigma^2$  & $Z$     \\ 
\multicolumn{1}{r}{} &  & $\beta>0$ & $\hat{\sigma}^2_{\mbox{\tiny OB-II}}(m_n,b_n) \inD \sigma^2\chi^2_{\mbox{\tiny OB-II}}(\beta,b_{\infty})$  & $T_{\mbox{\tiny OB-II}}(\beta,b_{\infty})$  \\ \hline

\multicolumn{1}{l}{\multirow{2}{*}{OB-III}} & \multicolumn{1}{c}{\multirow{2}{*}{$\hat{\theta}_n$}} & $\beta=0$ & $\hat{\sigma}^2_{\mbox{\tiny OB-III}}(m_n,b_n) \inP \sigma^2$  & $Z$   \\ 
\multicolumn{1}{r}{} &  & $\beta>0$  & $\hat{\sigma}^2_{\mbox{\tiny OB-III}}(m_n,b_n) \inD \sigma^2\chi^2_{\mbox{\tiny OB-III}}(\beta,b_{\infty})$  & $T_{\mbox{\tiny OB-III}}(\beta,b_{\infty})$  \\ \bottomrule
\end{tabular}\label{synopsis}
\end{table}
The proposed intervals~\eqref{twosidedconf1},~\eqref{twosidedconf2}, and~\eqref{twosidedconf3} rely crucially on the existence of the following weak limits:   \begin{align} T_{\mbox{\tiny OB-I}}(m_n,b_n) &:= \frac{\sqrt{n}(\hat{\theta}_n - \theta(P))}{\hat{\sigma}_{\mbox{\tiny OB-I}}(m_n,b_n)} \quad \inD \quad  T_{\mbox{\tiny OB-I}}(\beta,b_{\infty}); \tag{\mbox{OB-I} Limit} \\
T_{\mbox{\tiny OB-II}}(m_n,b_n) &:= \frac{\sqrt{n}(\bar{\theta}_n - \theta(P))}{\hat{\sigma}_{\mbox{\tiny OB-II}}(m_n,b_n)} \quad \inD \quad  T_{\mbox{\tiny OB-II}}(\beta,b_{\infty}); \tag{\mbox{OB-II} Limit}\\T_{\mbox{\tiny OB-III}}(m_n,b_n) &:= \frac{\sqrt{n}(\hat{\theta}_n - \theta(P))}{\hat{\sigma}_{\mbox{\tiny OB-III}}(m_n,b_n)} \quad \inD \quad  T_{\mbox{\tiny OB-III}}(\beta,b_{\infty}), \tag{\mbox{OB-III} Limit}
\end{align} where $\beta$ and $b_{\infty}$ are the limiting batch size and number of batches as defined in~\eqref{keylimits}. The existence of the weak limits $T_{\mbox{\tiny OB-x}}$, x = I,II,III, however, needs to be established and their characterization will occupy much of the rest of the paper. Furthermore, on our way to characterizing $T_{\mbox{\tiny OB-x}}$, x = I,II,III, we will also establish the weak limits of the estimators $\hat{\sigma}^2_{\mbox{\tiny OB-x}}(m_n,b_n), {\rm x = I,II,III}$ of the variance constant $\sigma^2$. The random variables $T_{\mbox{\tiny OB-x}}, {\rm x = I,II,III}$ and  $\hat{\sigma}_{\mbox{\tiny OB-x}}, {\rm x = I,II,III}$ should be seen as distribution-free statistical functional analogues of the Student's $t$ and $\chi^2$ random variables, respectively.  

As summarized in Table~\ref{synopsis}, the nature of $T_{\mbox{\tiny OB-x}}, {\rm x = I,II,III}$ (and those of $\hat{\sigma}_{\mbox{\tiny OB-x}},{\rm x = I,II,III}$)  depend on the limiting batch size $\beta$ and the limiting number of batches $b_{\infty}$. In particular, depending on whether $\beta=0$ (small batch regime) or $\beta>0$ (large batch regime), the statistics behave quite differently. For example, the small batch regime ($\beta =0$) produces the normal limit ($Z$ statistics) along with consistent estimation of $\sigma^2$, whereas the large batch regime ($\beta>0$) produces limits that are functionals of the Wiener process along with no consistent estimation of $\sigma^2$. The asymptotic number of batches $b_{\infty}$ affects the nature of the limiting distributions in the large batch regime. See Table~\ref{synopsis} for a synopsis. 

\section{Key Assumptions}\label{sec:mathprelim} In this section, we state and comment on various regularity assumptions that will be invoked when proving the technical results. Not all of these assumptions are ``standing assumptions'' in that some of the results to follow (especially when $\beta=0$) will need only a subset of the assumptions. 

\begin{assumption}[Stationarity]\label{ass:stat}  The $S$-valued sequence $\{X_n, n \geq 1\}$ is stationary, that is, for any $n_j, j=1,2,\ldots,k < \infty$ and $k < \infty$, the distribution of $(X_{n_1+\tau}, X_{n_2+\tau}, X_{n_3 + \tau}, \ldots,X_{n_k+\tau})$ does not depend on $\tau \in \mathbb{Z}^+$. \end{assumption} 

\begin{assumption}[Strong Mixing Condition]\label{ass:mixing} Suppose $\mathcal{G},\mathcal{H}$ are sub-$\sigma$-algebras of $\mathcal{F}$ in the probability space $(\Omega, \mathcal{F},P)$. Recall that the strong mixing constant $\alpha(\mathcal{G},\mathcal{H})$ is given by \begin{align} \alpha(\mathcal{G},\mathcal{H}) &= \sup_{A \in \mathcal{G}} \sup_{B \in \mathcal{H}} \left | P(AB) - P(A)P(B) \right| \nonumber \\ &= \frac{1}{2} \sup_{A \in \mathcal{G}} \mathbb{E}\left[ \left|P(A \vert \mathcal{H}) - P(A)\right|\right] \nonumber \\ &= \frac{1}{2} \sup_{A \in \mathcal{H}} \mathbb{E}\left[ \left|P(A \vert \mathcal{G}) - P(A)\right|\right].\end{align} We assume that the $S$-valued sequence $\{X_n, n \geq 1\}$ has strong mixing~\cite[pp. 347]{2009ethkur} constants $\alpha_n := \alpha(\mathcal{F}_{k},\mathcal{F}_{k,n})$ satisfying $\alpha_n \searrow 0$ as $n \to \infty$, where $\mathcal{F}_{k} : =\sigma(X_1,X_2,\ldots,X_k)$, $\mathcal{F}_{k,n} : =\sigma(X_{k+n},X_{k+n+1},\ldots)$ denote  sub-$\sigma$-algebras of $\mathcal{F}$ ``separated by $n$.''    \end{assumption}

\begin{assumption}[Central Limit Theorem]\label{ass:clt}  The sequence $\{\hat{\theta}_n, n \geq 1\}$ of sectioning estimators satisfies a central limit theorem (CLT), that is, \begin{equation}\label{clt} \sqrt{n}(\hat{\theta}_n - \theta(P)) \inD \sigma Z(0,1),\end{equation} where $Z(0,1)$ is the standard normal random variable and \rp{$\sigma \in (0,\infty)$} is a finite, positive constant. \end{assumption}

\begin{assumption}[Asymptotic Moment Existence] \label{ass:unif} The sequence $\{\hat{\theta}_n, n \geq 1\}$ of sectioning estimators is such that, for some $\delta_0 >0$, \begin{equation}\label{sigmaexistence} \mathbb{E}\left[\left(\sqrt{n}(\hat{\theta}_n - \theta(P))\right)^{2+\delta_0}\right] \to \sigma^{2+\delta_0} \mbox{ as } n \to \infty,\end{equation} where $\sigma$ is the constant appearing in Assumption~\ref{ass:clt}. 
\end{assumption} 

\rp{\begin{assumption}[Strong Invariance]\label{ass:stronginvar} The sequence $\{\hat{\theta}_n, n \geq 1\}$ of sectioning estimators satisfies the following strong invariance principle. There exists a standard Wiener process $\{W(t), t \geq 0\}$ and a stationary stochastic process $\{\tilde{X}_n, n \geq 1\} \overset{\rm d}{=} \{X_n, n \geq 1\}$ defined on a common probability space such that as $n \to \infty$, \begin{equation}\label{inv} \left |\sigma^{-1}\left(\hat{\theta}_{\lfloor n \rfloor} - \theta(P)\right) - n^{-1}W(n)\right | \leq \Gamma \, n^{-1/2-\delta}\sqrt{\log n} \quad \emph{ a.s.},\end{equation} where the constant $\delta>0$ and the real-valued random variable $\Gamma$ satisfies $\mathbb{E}[\Gamma] < \infty$. \end{assumption}}

Assumption~\ref{ass:stat} on the stationarity of the sequence $\{X_n, n \geq 1\}$ is mild and standard in settings  where a confidence interval is sought. Assumption~\ref{ass:mixing} on strong mixing is a weak asymptotic independence condition imposed to rigorize the intuitive idea that the dependence between events formed from subsets of the sequence $\{X_n, n \geq 1\}$ in the far past and the far future decays to zero as their separation diverges. Assumption~\ref{ass:stat} and Assumption~\ref{ass:mixing} are used only in our results involving small batches, that is, when $m_n/n \to 0$.  

\rp{As discussed in the introductory part of the paper, Assumption~\ref{ass:clt} on the existence of a CLT on $\hat{\theta}_n$, is fundamental to the methods presented here. (Assumption~\ref{ass:stronginvar} implies Assumption~\ref{ass:clt}.) While there are exceptions, a CLT holds in numerous useful settings where a confidence interval is desired, e.g., mean estimation~\cite[pp. 73]{1980ser}, quantile estimation~\cite[pp. 77]{1980ser}, gradient estimation~\cite[Section 7]{2003gla}, M-estimation~\cite[Chapter 6]{2006gee}, CVaR estimation~\cite{2022dedmer}, acf and spectral density estimation~\cite[Section 8.4]{1971and}, and robust statistics~\cite{1975gasrub}, apart from other more standard estimation settings in statistics. Assumption~\ref{ass:unif} goes a little further than Assumption~\ref{ass:clt} to stipulate the existence of the $(2+\delta_0)$-th moment of $\hat{\theta}_n$ (for some $\delta_0 >0$) and its convergence to $\sigma^{2+\delta_0}$. It can be shown that Assumption~\ref{ass:unif} implies the uniform integrability of the sequence $\{\sqrt{n}(\hat{\theta}_n - \theta(P)), n \geq 1\}$.}

The inequality in~\eqref{inv} of Assumption~\ref{ass:stronginvar}, sometimes called ``strong invariance,'' essentially stipulates that the scaled process $\left\{\sqrt{n}\sigma^{-1}\left(\hat{\theta}_{\lfloor t \rfloor} - \theta(\tilde{P})\right), t \leq n\right\}$ can be approximated uniformly to within $n^{-\delta}$ \emph{almost surely}, by a suitable standard Wiener process on a rich enough probability space. As argued in Philipp and Stout~\cite{1975phisto}, and Glynn and Iglehart~\cite{1988glyigl}, Assumption~\ref{ass:stronginvar} holds for a variety of weakly dependent processes. See~\cite{1981csorev} for strong invariance theorems on partial sums, empirical processes, and quantile processes.  

As will become evident, Assumption~\ref{ass:stronginvar} is used only in proving results that involve large batches, that is, when $m_n/n \to \beta >0$. We believe all these results will still hold with a functional CLT on $\hat{\theta}_n$ instead of Assumption~\ref{ass:stronginvar}. (Loosely, strong approximation  $\Rightarrow$ functional CLT $\Rightarrow$ CLT --- see, for instance,~\cite{1998gly,1980ser}.) Despite this increased generality that a functional CLT affords, we have chosen to remain with Assumption~\ref{ass:stronginvar} since the resulting proofs are more intuitive.  

\rp{\begin{remark} It is likely that Assumption~\ref{ass:stronginvar} can be relaxed, e.g., by replacing the canonical scaling $\sqrt{n}$ appearing in in~\eqref{clt} with $n^{\alpha}(\hat{\theta}_n - \theta(P)) \inD \sigma Z(0,1)$ for some known $\alpha>0$, without changing most of the results reported in this paper. Such generalization is part of an ongoing investigation and entails identifying alterations needed on the technical conditions involving batch size and number of batches.     
\end{remark}}

\section{The OB-I Limit}\label{sec:ob1}
In this section, we characterize the weak limit of \begin{equation}\label{ob1prelimit} T_{\mbox{\tiny OB-I}}(m_n,b_n) := \frac{\sqrt{n}(\hat{\theta}_n - \theta(P))}{\hat{\sigma}_{\mbox{\tiny OB-I}}(m_n,b_n)},\end{equation} as described in Section~\ref{sec:confints}. Along the way, we also characterize the asymptotic behavior of the variance estimator $\hat{\sigma}^2_{\mbox{\tiny OB-I}}(m_n,b_n)$. The ensuing  Section~\ref{sec:ob1largebatch} treats the $\beta := \lim_{n \to \infty} m_n/n > 0$ (large batch) regime, and Section~\ref{sec:ob1smallbatch} treats the $\beta=0$ (small batch) regime.

\subsection{Large Batch Regime for OB-I} \label{sec:ob1largebatch}
Theorem~\ref{thm:ob1largebatch} that follows asserts that $\hat{\sigma}^2_{\mbox{\tiny OB-I}}(m_n,b_n)/\sigma^2$ and $T_{\mbox{\tiny OB-I}}(m_n,b_n)$ converge weakly to certain functionals of the Wiener process that we denote $\chi^2_{\mbox{\tiny OB-I}}(\beta,b_{\infty})$ and $T_{\mbox{\tiny OB-I}}(\beta,b_{\infty})$, respectively. It is important that Theorem~\ref{thm:ob1largebatch} needs the strong invariance Assumption~\ref{ass:stronginvar} to hold so that the dependence across batches can be characterized precisely.  

\rp{\begin{theorem}[OB-I Large Batch Regime]\label{thm:ob1largebatch} Suppose  Assumption~\ref{ass:stronginvar} holds, and that $\beta = \lim_{n \to \infty} m_n/n \in (0,1).$ Assume also that $b_n \to b_{\infty} \in \{2,3,\ldots,\infty\}$ as $n \to \infty$. Define \begin{numcases}{\label{chisq1}\chi^2_{\emph{\mbox{\tiny OB-I}}}(\beta,b_{\infty}) :=} \frac{1}{\kappa_1(\beta,b_{\infty})}\frac{1}{\beta(1-\beta)} \int_{0}^{1-\beta} \left( W(u+\beta) - W(u) - \beta W(1)\right)^2 \, du & $b_{\infty} = \infty$; \nonumber \\ \frac{1}{\kappa_1(\beta,b_{\infty})} \frac{1}{\beta b_{\infty}}\sum_{j=1}^{b_{\infty}} \left( W(c_j+\beta) - W(c_j) - \beta W(1)\right)^2 & $b_{\infty} \in \mathbb{N}\setminus \{1\}$, \nonumber \\ \end{numcases} where $\kappa_1(\beta,b_{\infty}) = 1- \beta$ and $c_j := (j-1)\frac{1-\beta}{b_{\infty}-1}$. Then, as $n \to \infty$, \begin{equation}\label{swag} \hat{\sigma}^2_{\emph{\mbox{\tiny OB-I}}}(m_n,b_n) \inD \sigma^2\chi^2_{\emph{\mbox{\tiny OB-I}}}(\beta,b_{\infty}); \quad \emph{ and } \quad T_{\emph{\mbox{\tiny OB-I}}}(m_n,b_n) \inD  \frac{W(1)}{\sqrt{\chi^2_{\emph{\mbox{\tiny OB-I}}}(\beta,b_{\infty})}}.\end{equation} \end{theorem}
}

The following theorem characterizes the (asymptotic) moments of the OB-I variance estimator $\hat{\sigma}^2_{\emph{\mbox{\tiny OB-I}}}(m_n,b_n)$.

\begin{theorem}[OB-I Moments]\label{thm:OB1moments} Let the postulates of Theorem~\ref{thm:ob1largebatch} hold. If the random variable $\Gamma$ appearing in Assumption~\ref{ass:stronginvar} satisfies $\mathbb{E}[\Gamma^4] < \infty$, and $|m_n/n - \beta| = o(n^{-\delta})$, then \begin{align}\label{defn:epsnugget} \mathbb{E}[\hat{\sigma}^2_{\emph{\mbox{\tiny OB-I}}}(m_n,b_n)] &=  \sigma^2 + O(\epsilon_{1,n}); \quad \epsilon_{1,n} := \frac{n^{-\delta}}{\beta^{\delta} \kappa_1(\beta,b_{\infty}) } (\sqrt{2\log^2 \beta n \, \log^2 n} + \sqrt{2}\log^2\,n)\end{align} and $\delta >0 $ is the constant appearing in Assumption~\ref{ass:stronginvar}. 

Recalling that $b_{\infty} := \lim_n b_n \in \{2,3,\ldots,\infty\}$, suppose further that $$\eta := \lim_n \frac{b_n}{n}  \in [0,\infty),$$ implying that necessarily $$d:= \lim_n d_n = \lim_n \frac{n-m_n}{b_n-1} = \begin{cases}\frac{1-\beta}{\eta} & \eta >0;  \\
\infty & \eta =0. \end{cases} $$ Then, after redefining $\infty \times 0 =0$, we have that
\begin{align}\lim_{n}\mbox{\emph{Var}}(\hat{\sigma}^2_{\emph{\mbox{\tiny OB-I}}}(m_n,b_n)) =   
\frac{\sigma^4}{(1-\beta)^2}\bigg(2\left(1-2\beta + 3\beta^2\right)\tilde{\mu}_0 + 6\left(1-\beta\right)^2\mu_0 -8d(1-\beta)\mu_1 + 4\mu_2\bigg), \label{asympexp2} \end{align} where $\tilde{\mu}_0, \mu_0, \mu_1,\mu_2$ are given by \begin{align} \tilde{\mu}_0 &:=  \begin{cases} \frac{1}{2} \left(\frac{1-2\beta}{1-\beta} \right)^2 \mathbb{I}_{\{\beta \leq 1/2\}} & \mbox{ if } b_{\infty} = \infty, \\ \frac{1}{2}\left(1- \frac{1}{b_{\infty}}\lceil \frac{\beta}{1-\beta}(b_{\infty}-1) \rceil \right)\left(1- \frac{1}{b_{\infty}}\lceil \frac{\beta}{1-\beta}(b_{\infty}-1) \rceil + \frac{1}{b_{\infty}} \right)\mathbb{I}_{\{\beta \leq 1/2\}} & \mbox{ if } b_{\infty} < \infty; \end{cases} \end{align} and defining $\gamma := \frac{\beta}{1-\beta} \wedge 1$, \begin{align} \mu_0 &:= \begin{cases} \gamma \left(1- \frac{\gamma}{2}\right) & \mbox{ if } b_{\infty} = \infty, \\ \frac{1}{b_{\infty}}\lfloor \gamma(b_{\infty}-1) \rfloor \bigg(1 - \frac{1}{2}\frac{1}{b_{\infty}}\lfloor \gamma(b_{\infty}-1) \rfloor - \frac{1}{2}\bigg) & \mbox{ if } b_{\infty} < \infty; \end{cases}  \nonumber \\ \mu_1 &:= \frac{1}{6}\gamma^2 \frac{\eta}{\beta} \left(3- 2\gamma\right)\mathbb{I}_{\{b_{\infty} = \infty\}}, \nonumber \\ \mu_2 &:= \frac{1}{2}\gamma^3 \left(\frac{\eta}{\beta}\right)^2 \left(\frac{2}{3}- \frac{1}{2}\gamma\right)\mathbb{I}_{\{b_{\infty} = \infty\}}.\end{align}\end{theorem} 
We make some further observations before providing the proofs of Theorem~\ref{thm:ob1largebatch} and Theorem~\ref{thm:OB1moments}.
\begin{enumerate} \item[(a)] The estimator $\hat{\sigma}^2_{\mbox{\tiny OB-I}}(m_n,b_n)$ does not consistently estimate the variance parameter $\sigma^2$, but converges weakly to the product of $\sigma^2$ and the random variable $\chi^2_{\emph{\mbox{\tiny OB-I}}}(\beta,b_{\infty})$ appearing in~\eqref{swag}. As in all cancellation methods, the weak limit of $T_{\mbox{\tiny OB-I}}(m_n,b_n)$ does not involve $\sigma^2$ since it ``cancels out.'' We slightly abuse notation for ease of exposition and \rp{use~\eqref{swag}} to define the $T_{\mbox{\tiny OB-I}}$ random variable:  $$T_{\mbox{\tiny OB-I}}(\beta,b_{\infty}) := \frac{W(1)}{\sqrt{\chi^2_{\emph{\mbox{\tiny OB-I}}}(\beta,b_{\infty})}}, \quad (\beta,b_{\infty}) \in (0,1) \times \mathbb{N} \setminus \{1\}.$$ \item[(b)] The factor $\kappa_1(\beta,b_{\infty})=1-\beta$ is a ``bias correction'' factor introduced to ensure that $\hat{\sigma}^2_{\mbox{\tiny OB-I}}(m_n,b_n)$ is asymptotically unbiased. \item[(c)] The expression for $\chi^2_{\mbox{\tiny OB-I}}(\beta,b_{\infty})$ in Theorem~\ref{thm:ob1largebatch} seems to have appeared first in~\cite[pp. 326]{2009aktalegolwil} for the steady-state mean context and assuming fully overlapping batches, that is, for $d_n=1$ and $b_{\infty}=\infty$. (The reader should be aware that while $b_{\infty}$ in the current paper refers to the limiting number of batches, $b_{\infty}$ in~\cite{2009aktalegolwil} refers to the ratio $n/m_n \to \beta^{-1}$. Furthermore, a simple re-scaling of the Wiener process is needed to see that the expression appearing in Theorem~\ref{thm:ob1largebatch} and that in ~\cite[pp. 326]{2009aktalegolwil} are equivalent.)  Similarly, the special case of fully overlapping batches and $b_{\infty}=\infty$ for $\mbox{Var}(\hat{\sigma}^2_{\mbox{\tiny OB-I}}(m_n,b_n))$ in Theorem~\ref{thm:OB1moments} appears in~\cite[pp. 290]{1995dam} for the context of the steady-state mean. \rp{\item[(d)] We can show through calculus on~\eqref{asympexp2} that $\inf_{\beta \in (0,1)} \left\{\lim_{n \to \infty} \mbox{Var}(\hat{\sigma}^2_{\mbox{\tiny OB-I}}(m_n,b_n))\right\} = 0$ is approached as $\beta \to 0$. (The infimum is not attained although there is a local minimum around $\beta = 0.467$.) This suggests using small batches but this is counter to what is seen in practice. Our numerical experience here and elsewhere suggests rather strongly that the asymptotic batch size $\beta$ has a ``first-order effect'' on coverage probability (with large $\beta$ being better), and a ``second-order effect'' on expected half-width (with large $\beta$ being bad), whereas the asymptotic number of batches $b_{\infty}$ has a ``first-order effect on expected half-width'' (with large $b_{\infty}$ good) but a ``second-order effect'' on coverage probability. These arguments suggest that using~\eqref{asympexp2} as the sole means of deciding the quality of confidence intervals is misleading.}

\item[(e)] The offset parameter $d_n$ comes into play through its effect on the limiting number of batches $b_{\infty}$. Specifically, notice that since $b_n = 1 + (n-m_n)/d_n$ and $m_n/n \to \beta,$ the asymptotic number of batches $b_{\infty}=\infty$ if $d_n/n = o(1)$, and $b_{\infty}< \infty$ if $\lim_n d_n/n >0$ (assuming it exists). 

\item[(f)] The table in Figure~\ref{fig:obt1} displays the critical values $t_{\mbox{\tiny OB-I},{\tiny 1-\alpha}}(\beta,b_{\infty}):= \min_x P(T_{\mbox{\tiny OB-I}}(\beta,b_{\infty}) \leq x) \geq 1-\alpha$ associated with the $T_{\mbox{\tiny OB-I}}$ distribution as a function of $1-\alpha$ and for different values of the parameters $\beta, b_{\infty}$. R and MATLAB code for calculating the critical values can be obtained through \texttt{https://web.ics.purdue.edu/$\sim$pasupath}.
\end{enumerate}

\begin{figure}[ht]
\centering                
{\includegraphics[clip, trim=0cm 2.5cm 0cm 0cm, width=1.00\textwidth]{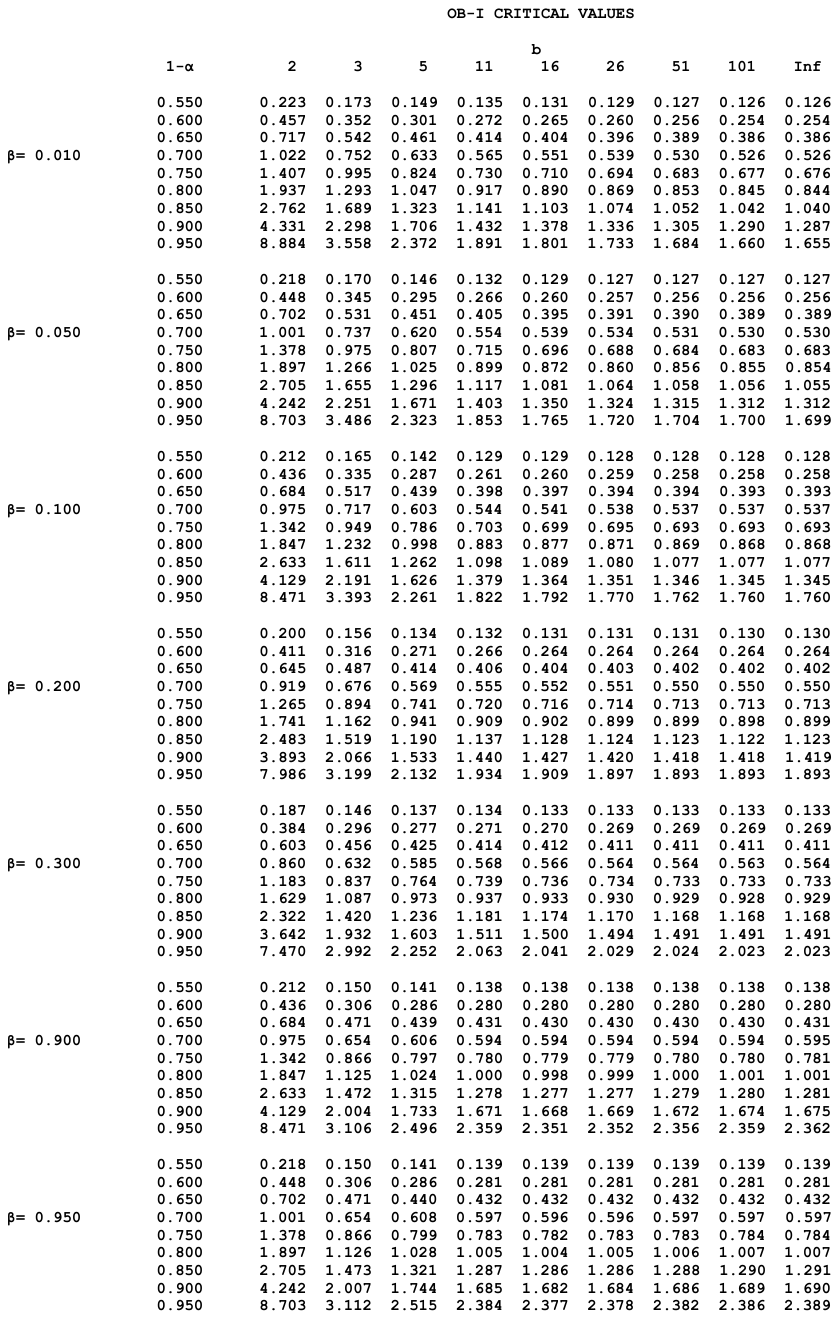}}
\caption{$T_{\mbox{\tiny OB-I}}$ critical values. The table displays critical values $t_{\mbox{\tiny OB-I},{\tiny 1-\alpha}}(\beta,b_{\infty}):= \inf \{r: P(T_{\mbox{\tiny OB-I}}(\beta,b_{\infty}) \leq r) = 1-\alpha \}$ associated with the $T_{\mbox{\tiny OB-I}}$ distribution as a function of $1-\alpha$, the asymptotic batch size $\beta$, and the asymptotic number of batches $b_{\infty}$.}\label{fig:obt1} 
\end{figure}

\subsection{Proofs of Theorem~\ref{thm:ob1largebatch} and Theorem~\ref{thm:OB1moments}} 
\begin{proof}[Proof of Theorem~\ref{thm:ob1largebatch}] 
Since Assumption~\ref{ass:stronginvar} holds, we will establish the first assertion in~\eqref{swag} by comparing individual terms that comprise $\hat{\sigma}^2_{\mbox{\tiny OB-I}}(m_n,b_n)$ against corresponding terms constructed from the Wiener process. Specifically, let's define \begin{equation}\label{browniandiff}\tilde{B}_{j,m_n}: = m_n^{-1}\left(W((j-1)\frac{n-m_n}{b_n-1} +m_n) - W((j-1)\frac{n-m_n}{b_n-1} )\right), \quad j = 1,2,\ldots,b_n\end{equation} and observe that \begin{align}\label{firstsplit} (1-\beta)\hat{\sigma}^2_{\mbox{\tiny OB-I}}(m_n,b_n) &= \underbrace{\frac{m_n}{b_n}\sum_{j=1}^{b_n} \left[\left(\hat{\theta}_{j,m_n} - \hat{\theta}_n\right)^2 - \sigma^2\left(\tilde{B}_{j,m_n} - n^{-1}W(n)\right)^2\right]}_{E_n(m_n,b_n)} \\ & \hspace{2in} + \underbrace{\sigma^2\frac{1}{b_n}\sum_{j=1}^{b_n}\left(\sqrt{m_n}\tilde{B}_{j,m_n} - \frac{\sqrt{m_n}}{n}W(n))\right)^2}_{I_n} \nonumber \\ &=: E_n(m_n,b_n) + I_n.\end{align} Noticing that \begin{equation}\label{error} \hat{\theta}_{j,m_n} - \hat{\theta}_n = \underbrace{\left(\hat{\theta}_{j,m_n} - \sigma\tilde{B}_{j,m_n}\right)}_{U_{j,m_n}} + \sigma\underbrace{\left(\tilde{B}_{j,m_n} - n^{-1}W(n)\right)}_{H_{j,m_n}} + \underbrace{\left(n^{-1}\sigma W(n) - \hat{\theta}_n\right)}_{C_{n}},\end{equation} we can then write \begin{align}\label{split2} E_n(m_n,b_n) = m_n\left(\frac{1}{b_n}\sum_{j=1}^{b_n} U_{j,m_n}^2 + 2 \frac{\sigma}{b_n}\sum_{j=1}^{b_n}  U_{j,m_n}H_{j,m_n} + 2 \frac{C_n}{b_n}\sum_{j=1}^{b_n} U_{j,m_n} +  2 \frac{\sigma C_n}{b_n}\sum_{j=1}^{b_n} H_{j,m_n} + C_n^2\right).\end{align} Now, we see that except for a set of measure zero in the probability space implied by Assumption~\ref{ass:stronginvar}, there exists $\Gamma(\omega)$ such that, uniformly in $j$,  \begin{equation}\label{errbds1} |U_{j,m_n}| \leq \Gamma(\omega) m_n^{-1/2-\delta}\left(\log^2 m_n\right)^{1/2}; \quad |C_n|  \leq \Gamma(\omega) n^{-1/2-\delta}\left(\log^2 n\right)^{1/2}, \end{equation} Furthermore, due to Theorem~\ref{thm:brownianinv}, for any given $\epsilon>0$, except for a set of measure zero in the probability space implied by Assumption~\ref{ass:stronginvar}, there exists $n_0(\omega,\epsilon)$ such that for all $n \geq n_0(\omega,\epsilon)$, and uniformly in $j$, \begin{align}\label{browniannuggetbds} |H_{j,m_n}| & \leq (1+\epsilon)\left(m_n^{-1/2} \left(2(\log^2 n - \log \frac{m_n}{n})\right)^{1/2} + n^{-1/2}\left(2\log^2 n\right)^{1/2}\right) \nonumber \\ & \leq (1+\epsilon)  m_n^{-1/2} \left( \left(2(\log^2 n - \log \frac{m_n}{n})\right)^{1/2} + \left(2\log^2 n\right)^{1/2} \right) 
\end{align} after ignoring non-integralities. 

Plugging~\eqref{errbds1} and~\eqref{browniannuggetbds} in~\eqref{split2}, we get
\begin{align}\label{errfinbd} E_n(m_n,b_n) &\leq \Gamma^2(\omega) m_n^{-2\delta}\log^2 m_n
\nonumber \\ & + 2 \sigma (1+\epsilon) \Gamma(\omega) m_n^{-\delta}\left(\left(\log^2 m_n\right)^{1/2} + \left(\log^2 n\right)^{1/2}  \right) \left((2(\log^2 n - \log \frac{m_n}{n}))^{1/2} + (2 \log^2 n)^{1/2} \right) \nonumber \\ &  + 2\sigma(\frac{m_n}{n})^{\frac{1}{2}} \Gamma^2(\omega)m_n^{-2\delta}\left(\log^2 m_n \log^2 n\right)^{1/2} + \Gamma^2(\omega) (\frac{m_n}{n}) m_n^{-2\delta}\log^2 n.\end{align} Notice that the second term appearing on the right-hand side of~\eqref{errfinbd} is dominant and goes to zero almost surely. 

Now lets calculate the weak limit of $I_n :=  \sigma^2\frac{1}{b_n}\sum_{j=1}^{b_n}\left(\sqrt{m_n}\tilde{B}_{j,m_n} - \frac{\sqrt{m_n}}{n} W(n)\right)^2$ appearing in~\eqref{firstsplit}. To facilitate calculation, define the lattice $\{0, \delta_n, 2 \delta_n, \ldots, \lfloor \frac{1}{\delta_n} \rfloor \delta_n \}$ having resolution $\delta_n : = \frac{1-(m_n/n)}{b_n-1},$ and a corresponding projection operation $\lfloor u \rfloor_{\scriptsize{\delta_n}} := \max\{k \delta_n: u \geq k \delta_n, k \in \mathbb{Z}\}, \quad u \in [0, 1-\delta_n].$ 

Now, recalling that $b_n = 1 + d_n^{-1}(n - m_n)$, we  can rewrite \begin{align}\label{Inlimitbinf} I_n &= \sigma^2\frac{1}{b_n}\sum_{j=1}^{b_n}\left(\sqrt{m_n}\tilde{B}_{j,m_n} - \frac{\sqrt{m_n}}{n} W(n)\right)^2 \nonumber \\ &= \sigma^2\frac{1}{b_n} \int_0^{1- \frac{m_n}{n} + \delta_n} \left( \frac{1}{\sqrt{m_n}} \left( W(n \lfloor  u \rfloor_{\delta_n} + m_n) - W(n \lfloor  u \rfloor_{\delta_n} ) \right) - \frac{\sqrt{m_n}}{n} W(n) \right)^2 \, \frac{1}{\delta_n} \, du \nonumber \\ &=  \sigma^2\frac{b_n-1}{b_n} \frac{n}{n-m_n} \int_0^{1- \frac{m_n}{n} + \delta_n} \left( \frac{1}{\sqrt{m_n}} \left( W(n \lfloor  u \rfloor_{\delta_n} + m_n) - W(n \lfloor  u \rfloor_{\delta_n}) \right) - \frac{\sqrt{m_n}}{n} W(n) \right)^2 \, du \nonumber \\ & \overset{d}{=} \sigma^2\frac{b_n-1}{b_n} \frac{n}{n-m_n} \frac{n}{m_n} \int_0^{1- \frac{m_n}{n} + \delta_n}  \left( W( \lfloor  u \rfloor_{\delta_n} + \frac{m_n}{n}) - W( \lfloor  u \rfloor_{\delta_n} ) - \frac{m_n}{n} W(1) \right)^2 \, du \nonumber  \\ & \to \sigma^2 \frac{\beta^{-1}}{1- \beta} \int_0^{1-\beta} \left( W(u + \beta) - W(u) - \beta W(1) \right)^2 \, du\end{align} if $\delta_n \to 0$ as $n \to \infty$ which happens when $b_n \to b_{\infty} = \infty$. This proves the $b_{\infty}= \infty$ case appearing in~\eqref{swag}.  

To prove the $b_{\infty} \in \{2,3,\ldots\}$ case, we observe that \begin{align} \label{Inlimitbfin} I_n &:= \sigma^2 \frac{1}{b_n} \sum_{j=1}^{b_n} \left(\sqrt{m_n} \tilde{B}_{j,m_n} - \frac{\sqrt{m_n}}{n}W(n)\right)^2 \nonumber \\ & \overset{d}{=}
\sigma^2 \frac{n}{m_n} \frac{1}{b_n} \sum_{j=1}^{b_n} \left(W((j-1)\frac{1-(m_n/n)}{b_n-1} + \frac{m_n}{n}) -  W((j-1)\frac{1-(m_n/n)}{b_n-1}) - \frac{m_n}{n}W(1)\right)^2 \nonumber \\ & \to \sigma^2 \frac{1}{\beta} \frac{1}{b_{\infty}} \sum_{j=1}^{b_{\infty}} \left(W((j-1)\frac{1-\beta}{b_{\infty}-1} + \beta) -  W((j-1)\frac{1-\beta}{b_{\infty}-1}) - \beta W(1)\right)^2, \end{align} as $n \to \infty,$ proving the $b_{\infty} < \infty$ case appearing in~\eqref{swag}. 

Let's now prove the second statement in~\eqref{swag} holds. From Assumption~\ref{ass:stronginvar} we have \begin{equation}\label{stronginvarapp} \left| \sqrt{n}\frac{\left(\hat{\theta}_n - \theta(P)\right)}{\sigma} - \frac{1}{\sqrt{n}}W(n)\right| \leq \Gamma \, \frac{1}{n^{\delta}}\sqrt{\log n} \quad \mbox{ a.s.},\end{equation} where $\Gamma$ is a well-defined random variable with finite mean, and $\delta>0.$ Hence \begin{equation} \label{Tnumer} \sqrt{n}(\hat{\theta}_n - \theta(P)) = \frac{1}{\sqrt{n}}W(n) + \tilde{E}_n; \quad \tilde{E}_n = o(\frac{1}{n^{\delta/2}}) \mbox{ a.s.}\end{equation}  We can then write \begin{align}\label{Tcmtready} T_{\mbox{\tiny OB-I}}(m_n,b_n) := \frac{\sqrt{n}(\hat{\theta}_n - \theta(P))}{\hat{\sigma}_{\mbox{\tiny OB-I}}(m_n,b_n)} = \frac{\frac{1}{\sqrt{n}}W(n) + \tilde{E}_n}{\sqrt{\frac{1}{1-\beta}(I_n + E_n(m_n,b_n))}},\end{align} where $I_n$ and $E_n(m_n,b_n)$ were introduced in~\eqref{firstsplit}, and both $E_n(m_n,b_n)$ and $\tilde{E}_n$ go to zero almost surely. \rp{Now apply to~\eqref{Tcmtready} the same steps leading to weak limits in~\eqref{Inlimitbinf} and~\eqref{Inlimitbfin} --- first replace by an object that is equal in distribution and then take limit as $n \to \infty$ --- to conclude that the second assertion in~\eqref{swag} holds.}
\end{proof}

\begin{proof}[Proof of Theorem~\ref{thm:OB1moments}]
Let's next prove the asymptotic expansion appearing in ~\eqref{asympexp2}. Simple algebra yields, for all $j = 1,2,\ldots,b_n,$ \begin{equation} \mathbb{E}[(\tilde{B}_{j,m_n} - n^{-1}W(n))^2] = \frac{1}{m_n} - \frac{1}{n},\end{equation} implying that \begin{equation} \label{Inasympexp} \mathbb{E}[I_n] = (1 - \frac{m_n}{n})\,\sigma^2. \end{equation}  Plugging~\eqref{Inasympexp} and the inequality~\eqref{errfinbd} in~\eqref{firstsplit} \rp{(after noticing that we have assumed $\Gamma$ appearing in Assumption~\ref{ass:stronginvar} satisfies $\mathbb{E}[\Gamma^4] < \infty$)}, we conclude that as $n \to \infty$, \begin{equation} \label{Vnasympexp} \mathbb{E}[\hat{\sigma}^2_{\mbox{\tiny OB-I}}(m_n,b_n)] = \sigma^2 + O(\epsilon_{1,n}),\end{equation} where $\epsilon_{1,n}$ is defined in~\eqref{defn:epsnugget} and we recall that $\delta$ is the constant appearing in Assumption~\ref{ass:stronginvar}. This proves the assertion in~\eqref{defn:epsnugget}.

Using a similar but tedious calculation, we find that \begin{equation} \label{In2asympexp} \mbox{Var}(I_n) = \beta^4 \,\left(\frac{4\beta^{-3} - 11\beta^{-2} + 4 \beta^{-1} + 6}{3(1- \beta)^4}\right) \, \sigma^4.\end{equation} Again plugging~\eqref{In2asympexp} and the inequality~\eqref{errfinbd} in~\eqref{firstsplit} \rp{(after noticing that we have assumed $\Gamma$ appearing in Assumption~\ref{ass:stronginvar} satisfies $\mathbb{E}[\Gamma^4] < \infty$)}, we conclude that as $n \to \infty$, \begin{equation} \label{Vnasympexp} \mbox{{Var}}(\hat{\sigma}^2_{\mbox{\tiny OB-I}}(m_n,b_n)) = \beta^4 \,\left(\frac{4\beta^{-3} - 11\beta^{-2} + 4 \beta^{-1} + 6}{3(1- \beta)^4}\right) \, \sigma^4 + O(\epsilon_{1,n}^2),\end{equation} thus proving the assertion in~\eqref{asympexp2}.
\end{proof}

\subsection{Small Batch Regime for OB-I}\label{sec:ob1smallbatch}
Theorem~\ref{thm:ob1largebatch} characterizes the effect of using large batch sizes, that is, $\lim_{n \to \infty} m_n/n = \beta > 0$ on the asymptotic behavior of $T_{\mbox{\tiny OB-I}}(m_n,b_n)$ and $\hat{\sigma}^2_{\mbox{\tiny OB-I}}(m_n,b_n)$. Theorem~\ref{thm:ob1smbatch} does the same but for the small batch ($\beta=0$) context. In particular, Theorem~\ref{thm:ob1smbatch} asserts that when small batches are used, $\hat{\sigma}^2_{\mbox{\tiny OB-I}}(m_n,b_n)$ consistently estimates $\sigma^2$, and that $T_{\mbox{\tiny OB-I}}(m_n,b_n)$ converges to the standard normal distribution.

\begin{theorem}[OB-I Small Batch Regime]\label{thm:ob1smbatch}  Suppose Assumptions~\ref{ass:stat}--\ref{ass:unif} hold, and that $\beta = \lim_{n \to \infty} m_n/n =0.$ Assume that the asymptotic number of batches $b_{\infty} := \lim_n b_n = \infty$. Then, as $n \to \infty$, \begin{equation}\label{statements} \hat{\sigma}^2_{\emph{\mbox{\tiny OB-I}}}(m_n,b_n) \inP \sigma^2; \quad \emph{ and } \quad  T_{\emph{\mbox{\tiny OB-I}}}(m_n,b_n) \inD Z(0,1). \end{equation} \end{theorem}
\begin{proof}
Since $\beta=0,$ $\kappa_1(\beta) = 1-\beta = 0$ and \begin{equation} \label{varparestagain} \hat{\sigma}^2_{\mbox{\tiny OB-I}}(m_n,b_n) := \frac{m_n}{b_n} \sum_{i=1}^{b_n} (\hat{\theta}_{i,m_n} - \hat{\theta}_n)^2.\end{equation} Also, define \begin{equation} \label{varparesthypo} \tilde{\sigma}^2_{\mbox{\tiny OB-I}}(m_n,b_n) := \frac{1}{b_n} \sum_{i=1}^{b_n} \underbrace{m_n(\hat{\theta}_{i,m_n} - \theta(P))^2}_{R_{i,m_n}}; \quad \tilde{\sigma}^2_{\mbox{\tiny OB-I}}(m_n,d_n;r) := \frac{1}{b_n} \sum_{i=1}^{b_n} \underbrace{R_{i,m_n}\mathbb{I}_{[0,r]}(R_{i,m_n})}_{R_{i,m_n}(r)},\end{equation}

We will first demonstrate that \begin{equation}\label{firststep} \tilde{\sigma}^2_{\mbox{\tiny OB-I}}(m_n,b_n) \inP \sigma^2. \end{equation} Notice that \begin{align}\label{expslit}  \mathbb{E}\left[\left |\tilde{\sigma}^2_{\mbox{\tiny OB-I}}(m_n,b_n) - \sigma^2 \right | \right] & \leq \overbrace{\mathbb{E}\left[\left |\tilde{\sigma}^2_{\mbox{\tiny OB-I}}(m_n,b_n) - \tilde{\sigma}^2_{\mbox{\tiny OB-I}}(m_n,d_n;r) \right |\right]}^{\scriptsize\mbox{I}}  \nonumber \\ & \hspace{1.5in} + \underbrace{\mathbb{E}\left[\left|\tilde{\sigma}^2_{\mbox{\tiny OB-I}}(m_n,d_n;r) - \sigma^2(r) \right | \right]}_{\scriptsize\mbox{II}} + \underbrace{\mathbb{E}\left[\left|\sigma^2(r) - \sigma^2 \right | \right]}_{\scriptsize\mbox{III}},\end{align} where $$\sigma^2(r) := \mathbb{E}\left[\sigma^2Z^2\mathbb{I}_{[0,r]}(\sigma^2 Z^2)\right]; \quad Z \overset{d}{=} Z(0,1).$$ Let's consider the first and last terms on the right-hand side of~\eqref{expslit}. Since $ \tilde{\sigma}^2_{\mbox{\tiny OB-I}}(m_n,b_n) - \tilde{\sigma}^2_{\mbox{\tiny OB-I}}(m_n,d_n;r) = \frac{1}{b_n} \sum_{i=1}R_{i,m_n}\mathbb{I}_{(r,\infty)}(R_{i,m_n}) \geq 0,$ and $R_{i,m_n}, i=1,2,\ldots,b_n$ are identically distributed, we have \begin{equation}\label{firsttermbd} \mathbb{E}\left[\left|\tilde{\sigma}^2_{\emph{\mbox{\tiny OB-I}}}(m_n,b_n) - \tilde{\sigma}^2_{\emph{\mbox{\tiny OB-I}}}(m_n,d_n;r)\right|\right] = \mathbb{E}\left[R_{i,m_n}\mathbb{I}_{(r,\infty)}(R_{i,m_n})\right]. \end{equation} Furthermore, due to Assumption~\ref{ass:unif}, we know that $R_{i,m_n}$ is uniformly integrable (for each $i$), and hence for any given $\epsilon>0$, there exists $r_0=r_0(\epsilon)$ (not dependent on $i$) such that for $r \geq r_0$, \begin{equation}\label{firsttermunifbd} \mathbb{E}\left[R_{i,m_n}\mathbb{I}_{(r,\infty)}(R_{i,m_n})\right] \leq \epsilon; \quad \mathbb{E}\left[ \sigma^2 Z^2 \mathbb{I}_{(r,\infty)}(\sigma^2Z^2)\right] \leq \epsilon.\end{equation} From~\eqref{firsttermunifbd}, we see that the terms I and III in~\eqref{expslit} satisfy, for $r \geq r_0$, \begin{equation}\label{firstandthird}   \mathbb{E}\left[\left|\tilde{\sigma}^2_{\emph{\mbox{\tiny OB-I}}}(m_n,b_n) - \tilde{\sigma}^2_{\emph{\mbox{\tiny OB-I}}}(m_n,d_n;r)\right|\right] \leq \epsilon; \quad \mathbb{E}\left[\left|\sigma^2(r) - \sigma^2 \right | \right] \leq \epsilon.\end{equation} Let's now analyze the term II in~\eqref{expslit}. Write 
\begin{equation}\label{term2}
\mathbb{E}\left[\left|\tilde{\sigma}^2_{\mbox{\tiny OB-I}}(m_n,d_n;r) - \sigma^2(r) \right | \right] \leq \underbrace{\sqrt{\mbox{Var}(\tilde{\sigma}^2_{\emph{\mbox{\tiny OB-I}}}(m_n,d_n;r))}}_{I_3} + \underbrace{\left|\mathbb{E}\left[\tilde{\sigma}^2_{\emph{\mbox{\tiny OB-I}}}(m_n,d_n;r)\right] - \sigma^2(r) \right |}_{I_4}. \end{equation} From Assumption~\ref{ass:unif} and since $R_{i,m_n}, i=1,2,\ldots,b_n$ have identical distributions, we know that $\mathbb{E}\left[\tilde{\sigma}^2_{\emph{\mbox{\tiny OB-I}}}(m_n,b_n)\right] \to \sigma^2$ as $n \to \infty$. This fact and the uniform integrability of $\tilde{\sigma}^2_{\emph{\mbox{\tiny OB-I}}}(m_n,b_n)$ mean that for any $\epsilon>0$, there exist $\ell = \ell(\epsilon)$ and $r_1 = r_1(\epsilon)$ such that for $n \geq \ell$ and $r \geq r_1$ the term $I_4$ in~\eqref{term2} satisfies \begin{equation}\label{term3} \left |\mathbb{E}\left[\tilde{\sigma}^2_{\emph{\mbox{\tiny OB-I}}}(m_n,d_n;r)\right] - \sigma^2(r)\right | \leq \epsilon.\end{equation}  To quantify term $I_3$ in~\eqref{term2}, write \begin{align}\label{varbddep}  \mbox{Var}(\tilde{\sigma}^2_{\emph{\mbox{\tiny OB-I}}}(m_n,d_n;r)) &= \frac{1}{b_n}\mbox{Var}(R_{1,m_n}(r)) + \frac{2}{b_n^2} \sum_{j=1}^{b_n} (b_n-j)\mbox{Cov}(R_{1,m_n}(r), R_{1+j,m_n}(r)) \nonumber \\ &\leq \frac{1}{b_n}\mbox{Var}(R_{1,m_n}(r)) + 16r^2 \left(\frac{1}{b_n}\sum_{j=1}^{b_n}(1- \frac{j}{b_n})\alpha_j\right) \nonumber \\ 
& \leq  \frac{r^2}{4b_n} + 16r^2 \left(\frac{1}{b_n}\sum_{j=1}^{b_n -1}\alpha_j\right) \nonumber \\ & \to 0, \end{align} where $\alpha_j := \alpha(\mathcal{F}_{1,m_n},\mathcal{F}_{1+j,m_n})$ is the strong mixing constant associated the sigma algebras $\sigma(X_1,X_2,\ldots,X_{m_n})$, $\sigma(X_{jd_n+1}, X_{jd_n + 2}, \ldots, X_{jd_n+m_n})$ formed by random variables in batch $1$ and batch $1+j$, the first inequality in~\eqref{varbddep} follows upon application of Corollary 2.5 in~\cite[pp. 347]{2009ethkur} with $u=1,v= \infty,w=\infty$, the second inequality in~\eqref{varbddep} follows since $R_{1,m_n}(r) \in [0,r]$, and the last inequality in~\eqref{varbddep} follows since Assumption~\ref{ass:mixing} implies $\alpha_j \to 0$ implying in turn that the C\'{e}saro sum $b_n^{-1} \sum_{j=1}^{b_n} \alpha_j \to 0$.  

Now by applying~\eqref{firstandthird},~\eqref{term2},~\eqref{term3} and~\eqref{varbddep} in~\eqref{expslit}, and since $\epsilon$ is arbitrary, we see that~\eqref{firststep} holds, that is, $\tilde{\sigma}^2_{\emph{\mbox{\tiny OB-I}}}(m_n,b_n) \inP \sigma^2.$ To complete the first part of the theorem's assertion in~\eqref{statements}, we write  \begin{align}\label{finalssmbatch} \hat{\sigma}^2_{\emph{\mbox{\tiny OB-I}}}(m_n,b_n) &= \frac{1}{b_n} \sum_{j=1}^{b_n} m_n(\hat{\theta}_{i,m_n} - \theta(P))^2 + m_n(\hat{\theta}_n - \theta(P))^2 + \frac{2}{b_n} \sum_{j=1}^{b_n}m_n(\hat{\theta}_{i,m_n} - \theta(P))(\hat{\theta}_n - \theta(P)) \nonumber \\ &= \tilde{\sigma}^2_{\emph{\mbox{\tiny OB-I}}}(m_n,b_n) + \left(\frac{m_n}{n}\right) n(\hat{\theta}_n - \theta(P))^2 \nonumber \\ & \hspace{2in} + 2\left(\sqrt{\frac{m_n}{n}}\right)\sqrt{n}(\hat{\theta}_n - \theta(P))\frac{1}{b_n} \sum_{j=1}^{b_n}\sqrt{m_n} (\hat{\theta}_{i,m_n} - \theta(P)).\end{align} Through prior arguments, we proved that the first term on the right-hand side of~\eqref{finalssmbatch} tends to $\sigma^2$ in probability; also, because $ \sqrt{n}(\hat{\theta}_n - \theta(P)) \inD \sigma Z(0,1)$, and $\beta := \lim_{n \to \infty} m_n/n = 0$, Slutsky's theorem~\eqref{thm:slutsky} ensures that the second term on the right-hand side of~\eqref{finalssmbatch} is $o_P(1).$ To see that the third term on the right-hand side of~\eqref{finalssmbatch} also tends to zero in probability, notice again that $ \sqrt{n}(\hat{\theta}_n - \theta(P)) \inD \sigma Z(0,1)$ and that \begin{align} \sqrt{\frac{m_n}{n}}\mathbb{E}\left[\frac{1}{b_n} \sum_{j=1}^{b_n}\sqrt{m_n} (\hat{\theta}_{i,m_n} - \theta(P)\right] \leq \sqrt{\frac{m_n}{n}}\frac{1}{b_n} \sum_{j=1}^{b_n} \mathbb{E}\left[ \sqrt{m_n} \left | \hat{\theta}_{i,m_n} - \theta(P)\right |\right] \to 0, \end{align} and make use of Slutsky's theorem~\eqref{thm:slutsky}. This proves the first assertion of the theorem in~\eqref{statements}.

To prove the second assertion in~\eqref{statements}, we again apply Slutsky's theorem~\eqref{thm:slutsky} to \begin{equation}\label{slutskytarget}T_{\mbox{\tiny OB-I}}(m_n,b_n) := \frac{\sqrt{n}(\hat{\theta}_n - \theta(P))}{\hat{\sigma}_{\mbox{\tiny OB-I}}(m_n,b_n)}\end{equation} after noticing that the numerator in the expression for $T_{\mbox{\tiny OB-II}}(m_n,b_n)$ converges weakly to $\sigma Z(0,1)$ due to Assumption~\ref{ass:clt} and the denominator converges in probability to $\sigma$ from the first assertion.  

\end{proof}

We now make a few observations regarding Theorem~\ref{thm:ob1smbatch}. \begin{enumerate} \item[(a)] Unlike in the large batch setting ($\beta >0$) of Theorem~\ref{thm:ob1largebatch}, the first assertion of Theorem~\ref{thm:ob1smbatch} guarantees that $\hat{\sigma}^2_{\mbox{\tiny OB-I}} (m_n,b_n)$ is a   consistent estimator of $\sigma^2$. 
\item[(b)] Unlike Theorem~\ref{thm:ob1largebatch}, Theorem~\ref{thm:ob1smbatch} does not need Assumption~\ref{ass:stronginvar} simply due to the fact that $\sigma^2$ is being estimated consistently, implying that the dependence between the numerator and the denominator of $T_{\mbox{\tiny OB-I}}(m_n,b_n)$ does not have to be explicitly modeled. This is what allows using Slutsky's theorem in Theorem~\ref{thm:ob1smbatch}.  
\rp{\item[(c)] Theorem~\ref{thm:ob1smbatch} assumes very little about the overlapping requirement of the batches apart from requiring the number of batches to diverge. In this sense, Theorem 5.3 is fundamentally different from Theorem 5.1; Theorem 5.3 relies on the point estimator $\hat{\sigma}^2_{\mbox{\tiny OB-I}}(m_n,b_n)$ being a consistent estimator of $\sigma^2$, whereas Theorem 5.1 results in a cancellation method that does not rely on the consistency of $\hat{\sigma}^2_{\mbox{\tiny OB-I}}(m_n,b_n)$. This is why Theorem 5.3 insists that $b_{\infty} = \infty$ whereas Theorem 5.1 does not. }
\item[(d)] As is evident from~\eqref{finalssmbatch}, characterizing the next order term for the mean and variance of $\hat{\sigma}^2_{\mbox{\tiny OB-I}}(m_n,b_n)$ (akin to Theorem~\ref{thm:OB1moments}) will involve assuming the nature of higher order terms in the uniform convergence assumption appearing as Assumption~\ref{ass:unif}.
\end{enumerate}

\section{The OB-II Limit}\label{sec:ob2}
In this section, we characterize the weak limit of \begin{equation}\label{ob1prelimit} T_{\mbox{\tiny OB-II}}(m_n,b_n) := \frac{\sqrt{n}(\bar{\theta}_n - \theta(P))}{\hat{\sigma}_{\mbox{\tiny OB-II}}(m_n,b_n)}.\end{equation} As described in Section~\ref{sec:confints}, recall that the OB-II limit $T_{\mbox{\tiny OB-II}}(m_n,b_n)$ differs from the OB-I limit in that it replaces the sectioning estimator $\hat{\theta}_n$ with the batching estimator $\bar{\theta}_n$ as the centering variable. As in the OB-I context, the ensuing sections treat the large batch and small batch regimes separately.

\subsection{Large Batch ($\beta>0$) Regime for OB-II}\label{sec:ob2largebatch}
Theorem~\ref{thm:ob2largebatch} that follows treats the large batch setting ($\beta := \lim_{n \to \infty} m_n/n > 0)$ and asserts that $\hat{\sigma}^2_{\mbox{\tiny OB-II}}(m_n,b_n)/\sigma^2$ and $T_{\mbox{\tiny OB-II}}(m_n,b_n)$ converge weakly to certain functionals of the Wiener process that we denote $\chi^2_{\mbox{\tiny OB-II}}(\beta,b_{\infty})$ and $T_{\mbox{\tiny OB-II}}(\beta,b_{\infty})$, respectively. The proof of Theorem~\ref{thm:ob2largebatch} follows closely along the lines of Theorem~\ref{thm:ob1largebatch}, and we include it in Appendix~\ref{app:ob2largebatch}. 

\rp{\begin{theorem}[OB-II Large Batch Regime] \label{thm:ob2largebatch} Suppose Assumption~\ref{ass:stronginvar} holds, and that $\beta := \lim_{n \to \infty} m_n/n >0.$ Assume also that $b_n \to b_{\infty} \in \{2,3,\ldots,\infty\}$ as $n \to \infty$. Define \begin{numcases}  {\chi^2_{\emph{\mbox{\tiny OB-II}}} (\beta,b_{\infty}) :=} \frac{1}{\kappa_2(\beta,\infty)}\frac{\beta^{-1}}{1- \beta} \int_0^{1-\beta} \left( \tilde{W}_u(\beta) - \frac{1}{1-\beta} \int_{0}^{1-\beta} \tilde{W}_s(\beta)  \, ds \right)^2 du & $b_{\infty} = \infty$; \nonumber \\ \frac{1}{\kappa_2(\beta,b_{\infty})}\frac{1}{\beta} \frac{1}{b_{\infty}} \sum_{j=1}^{b_{\infty}} \left(\tilde{W}_{c_j}(\beta) -  \frac{1}{b_{\infty}} \sum_{i=1}^{b_{\infty}} \tilde{W}_{c_i}(\beta)  \right)^2  & $b_{\infty} \in \mathbb{N}\setminus {1},$ \nonumber \\\end{numcases} where $\tilde{W}_x(\beta) := W(x+\beta) - W(x), x \in [0,1-\beta]$, $\{W(t), t \in [0,1]\}$ is the standard Brownian motion~\citep{1995bil}, $c_i := (i-1)\frac{1-\beta}{b_{\infty}-1}, i = 1,2, \ldots, b_{\infty}$, and $\kappa_2(\beta,b_{\infty})$ is the ``bias-correction" factor  given by \begin{numcases} {\label{biascorrect2again} \kappa_2(\beta,b_{\infty}) :=} 1 & $\beta =0$; \nonumber \\ 1 - 2\left(\frac{\beta}{1-\beta} \wedge 1\right) + \frac{1}{\beta}\left(\frac{\beta}{1-\beta} \wedge 1\right)^2 - \frac{2}{3}\frac{1-\beta}{\beta}\left(\frac{\beta}{1-\beta} \wedge 1 \right)^3 & $\beta>0, b_{\infty} = \infty;$ \nonumber \\
1- \frac{1}{b_{\infty}} - \frac{2}{b_{\infty}}\sum_{h=1}^{b_{\infty}} \left(1 - \frac{h}{b_{\infty}-1} \frac{1-\beta}{\beta}\right)^+(1-h/b_{\infty}) & $\beta>0, b_{\infty} \in \mathbb{N}\setminus {1}.$ \nonumber \\ \end{numcases}
Then, as $n \to \infty$, \begin{equation}\label{swag-samplemean} \hat{\sigma}^2_{\emph{\mbox{\tiny OB-II}}} (m_n,b_n) \inD \sigma^2 \chi^2_{\emph{\mbox{\tiny OB-II}}}(\beta,b_{\infty});\end{equation} and  \begin{numcases} {\label{ob2t-stat} T_{\emph{\mbox{\tiny OB-II}}}(m_n,b_n) \inD } \frac{1}{\sqrt{\chi^2_{\emph{\mbox{\tiny OB-II}}}(\beta,b_{\infty})}}\,\frac{1}{\beta}\frac{1}{ (1-\beta)} \int_0^{1-\beta} \left(W(s+\beta) - W(s)\right) \, ds  & $b_{\infty} = \infty;$ \nonumber \\ \frac{1}{\sqrt{\chi^2_{\emph{\mbox{\tiny OB-II}}}(\beta,b_{\infty})}}\frac{1}{\beta} \frac{1}{b_{\infty}}\sum_{i=1}^{b_{\infty}} W(c_i + \beta) - W(c_i)  & $b_{\infty} \in \mathbb{N} \setminus 1,$ \nonumber \\ \end{numcases} where $c_i := (i-1)\frac{1-\beta}{b_{\infty}-1}, i = 1,2, \ldots, b_{\infty}$.
\end{theorem}}

\begin{proof}{Proof} See Appendix~\ref{app:ob2largebatch}.

\end{proof}

\begin{figure}[ht]
\centering                
{\includegraphics[clip, trim=0cm 2.5cm 0cm 0cm, width=1.00\textwidth]{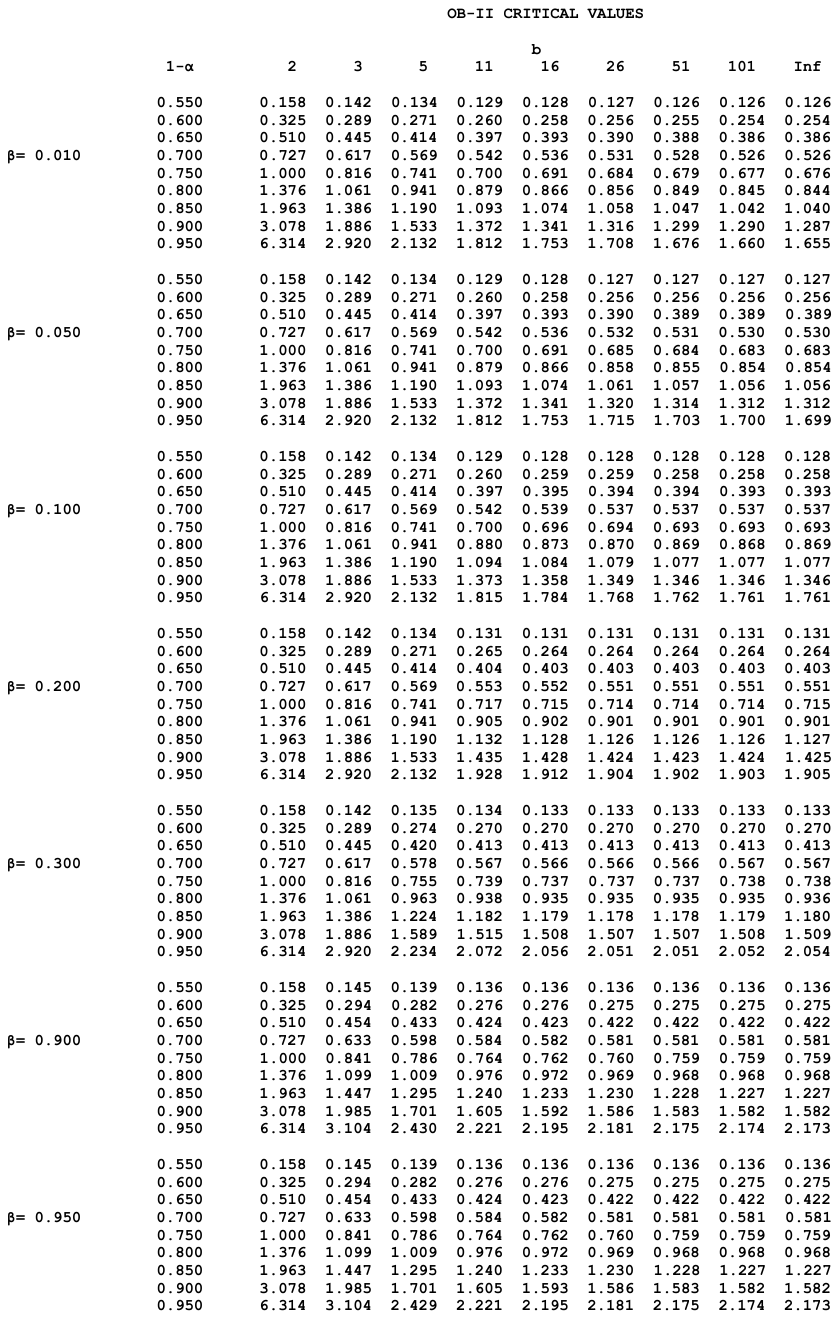}}
\caption{$T_{\mbox{\tiny OB-II}}$ Critical Values. The table displays critical values $t_{\mbox{\tiny OB-II},{\tiny 1-\alpha}}:= \inf \{r: P(T_{\mbox{\tiny OB-II}}(\beta,b_{\infty})  \leq r) = 1-\alpha \}$ associated with the OB-II distribution as a function of $1-\alpha$, the asymptotic batch size $\beta$ and the asymptotic number of batches $b_{\infty}$.}\label{fig:obt2} 
\end{figure}

We make a number of observations in light of Theorem~\ref{thm:ob2largebatch}.
\begin{enumerate} \item[(a)] As in Theorem~\ref{thm:ob1largebatch}, we see that the variance parameter $\sigma^2$ is not estimated consistently in Theorem~\ref{thm:ob2largebatch}. Instead the estimator $\hat{\sigma}^2_{\mbox{\tiny OB-II}}(m_n,b_n)$ converges weakly to the product of $\sigma^2$ and $\chi^2_{\mbox{\tiny OB-II}}(\beta,b_{\infty})$.  Again, we slightly abuse notation and define the weak limit appearing in~\eqref{ob2t-stat} as the $T_{\mbox{\tiny OB-II}}(\beta,b_{\infty})$ random variable.   \item[(b)] Unlike the the OB-I interval estimator, the OB-II interval estimator uses $\bar{\theta}_n$ as the centering variable and when estimating the variance constant. For this reason, and as we shall briefly discuss later, this makes the OB-II estimator attractive from a computational standpoint. \rp{\item[(c)] Like Theorem~\ref{thm:ob1largebatch}, Theorem~\ref{thm:ob2largebatch} requires Assumption~\ref{ass:stronginvar} to hold. }
\item[(d)] As can be seen, the ``bias correction'' factor $\kappa_2(\beta,b_{\infty})$ in~\eqref{biascorrect2again} for the OB-II context is much more complicated. The OB-II analogue of the OB-I asymptotic variance appearing in~\eqref{asympexp2} of Theorem~\ref{thm:OB1moments} has been elusive. \item[(e)] The table in Figure~\ref{fig:obt2} displays the critical values $t_{\mbox{\tiny OB-II},{\tiny 1-\alpha}}(\beta,b_{\infty}):= \min_x P(T_{\mbox{\tiny OB-II}}(\beta,b_{\infty}) \leq x) \geq 1-\alpha$ associated with the $T_{\mbox{\tiny OB-II}}$ distribution as a function of $1-\alpha$ and for different values of the parameters $\beta, b_{\infty}$. R and MATLAB code for calculating the critical values can be obtained through \texttt{https://web.ics.purdue.edu/$\sim$pasupath}.
\end{enumerate}

\subsection{Small Batch ($\beta=0$) Regime for OB-II}\label{sec:ob2smallbatch}
We now treat the small batch regime ($\beta := \lim_{n \to \infty} m_n/n = 0)$ for OB-II. Like Theorem~\ref{thm:ob1largebatch}, Theorem~\ref{thm:ob2largebatch} needs the strong invariance Assumption~\ref{ass:stronginvar} to hold so that the dependence across batches can be characterized.

\begin{theorem}[OB-II Small Batch Regime]\label{thm:ob2smbatch}  Suppose Assumptions~\ref{ass:stat}--\ref{ass:stronginvar} hold, and that $\beta = \lim_{n \to \infty} m_n/n =0.$ Assume that the number of batches $b_n \to \infty$. Then, as $n \to \infty$, \begin{equation}\label{statements2} \hat{\sigma}^2_{\emph{\mbox{\tiny OB-II}}}(m_n,b_n) \inP \sigma^2; \quad \emph{ and } \quad  T_{\emph{\mbox{\tiny OB-II}}}(m_n,b_n) \inD Z(0,1). \end{equation} \end{theorem}
\begin{proof}
Since $\beta=0,$ $\kappa_2(\beta) = 1$ and recall that \begin{equation} \label{varparest2again} \hat{\sigma}^2_{\mbox{\tiny OB-II}}(m_n,b_n) := \frac{m_n}{b_n} \sum_{i=1}^{b_n} (\hat{\theta}_{i,m_n} - \bar{\theta}_n)^2; \quad \bar{\theta}_n:= \frac{1}{b_n}\sum_{i=1}^{b_n} \hat{\theta}_{i,m_n}.\end{equation} Also, define \begin{equation} \label{varparesthypo2} \tilde{\sigma}^2_{\mbox{\tiny OB-II}}(m_n,b_n) := \frac{1}{b_n} \sum_{i=1}^{b_n} \underbrace{m_n(\hat{\theta}_{i,m_n} - \theta(P))^2}_{R_{i,m_n}}; \quad \tilde{\sigma}^2_{\mbox{\tiny OB-II}}(m_n,d_n;r) := \frac{1}{b_n} \sum_{i=1}^{b_n} \underbrace{R_{i,m_n}\mathbb{I}_{[0,r]}(R_{i,m_n})}_{R_{i,m_n}(r)},\end{equation} From arguments identical to that in the proof of Theorem~\ref{thm:ob1smbatch} (specifically, \eqref{expslit}--\eqref{varbddep}), we see that $\tilde{\sigma}^2_{\mbox{\tiny OB-II}}(m_n,b_n)$ consistently estimates $\sigma^2$, that is,
\begin{equation}\label{firststep2} \tilde{\sigma}^2_{\mbox{\tiny OB-II}}(m_n,b_n) \inP \sigma^2. \end{equation} To complete the first part of the theorem's assertion in~\eqref{statements2}, we write  \begin{align}\label{finalssmbatch2} \hat{\sigma}^2_{\mbox{\tiny OB-II}}(m_n,b_n) &= \frac{1}{b_n} \sum_{j=1}^{b_n} m_n(\hat{\theta}_{i,m_n} - \theta(P))^2 + m_n(\bar{\theta}_n - \theta(P))^2 + \frac{2}{b_n} \sum_{j=1}^{b_n}m_n(\hat{\theta}_{i,m_n} - \theta(P))(\bar{\theta}_n - \theta(P)) \nonumber \\ &= \tilde{\sigma}^2_{\mbox{\tiny OB-II}}(m_n,b_n) + \left(\frac{m_n}{m_nb_n}\right) m_nb_n(\bar{\theta}_n - \theta(P))^2 \nonumber \\ & \hspace{1.5in} + 2\left(\sqrt{\frac{m_n}{m_nb_n}}\right)\sqrt{m_nb_n}(\bar{\theta}_n - \theta(P))\frac{1}{b_n} \sum_{j=1}^{b_n}\sqrt{m_n} (\hat{\theta}_{i,m_n} - \theta(P)).\end{align} From~\eqref{firststep2}, we see that the first term on the right-hand side of~\eqref{finalssmbatch2} tends to $\sigma^2$ in probability; also, because $ \sqrt{n}(\hat{\theta}_n - \theta(P)) \inD \sigma Z(0,1)$, and $\beta := \lim_{n \to \infty} m_n/n = 0$, Slutsky's theorem~\eqref{thm:slutsky} ensures that the second term on the right-hand side of~\eqref{finalssmbatch2} is $o_P(1).$ To see that the third term on the right-hand side of~\eqref{finalssmbatch2} also tends to zero in probability, notice again that $ \sqrt{n}(\hat{\theta}_n - \theta(P)) \inD \sigma Z(0,1)$ and that \begin{align} \sqrt{\frac{m_n}{n}}\mathbb{E}\left[\frac{1}{b_n} \sum_{j=1}^{b_n}\sqrt{m_n} (\hat{\theta}_{i,m_n} - \theta(P)\right] \leq \sqrt{\frac{m_n}{n}}\frac{1}{b_n} \sum_{j=1}^{b_n} \mathbb{E}\left[ \sqrt{m_n} \left | \hat{\theta}_{i,m_n} - \theta(P)\right |\right] \to 0, \end{align} and make use of Slutsky's theorem~\eqref{thm:slutsky}. This proves the first assertion of the theorem in~\eqref{statements2}.

To prove the second assertion in~\eqref{statements2}, we again apply Slutsky's theorem~\eqref{thm:slutsky} to \begin{equation}\label{slutskytarget}T_{\mbox{\tiny OB-II}}(m_n,b_n) := \frac{\sqrt{n}(\hat{\theta}_n - \theta(P))}{\hat{\sigma}_{\mbox{\tiny OB-II}}(m_n,b_n)}\end{equation} after noticing that the numerator in the expression for $T_{\mbox{\tiny OB-II}}(m_n,b_n)$ converges weakly to $\sigma Z(0,1)$ due to Assumption~\ref{ass:clt} and the denominator converges in probability to $\sigma$ from the first assertion. 
\end{proof}

\section{The OB-III Limit}\label{sec:ob3}
Theorem~\ref{thm:ob3largebatch} that follows treats the large batch setting ($\beta := \lim_{n \to \infty} m_n/n > 0)$ and asserts that $\hat{\sigma}^2_{\mbox{\tiny OB-III}}(m_n,b_n)/\sigma^2$ and $T_{\mbox{\tiny OB-III}}(m_n,b_n)$ converge weakly to certain functionals of the Wiener process that we denote $\chi^2_{\mbox{\tiny OB-III}}(\beta,b_{\infty})$ and $T_{\mbox{\tiny OB-III}}(\beta,b_{\infty})$, respectively. Since the proof of Theorem~\ref{thm:ob3largebatch} follows closely along the lines of Theorem~\ref{thm:ob1largebatch} and Theorem~\ref{thm:ob2largebatch}, we do not provide a proof. 

\rp{\begin{theorem}[OB-III Large Batch Regime]\label{thm:ob3largebatch} Suppose that  Assumption~\ref{ass:stat}, Assumption~\ref{ass:mixing} and Assumption~\ref{ass:stronginvar} hold, and that $\beta = \lim_{n \to \infty} m_n/n \in (0,1).$ Assume also that $b_n \to b_{\infty} \in \{2,3,\ldots,\infty\}$ as $n \to \infty$, and that the weighting function $f:[0,1] \to \mathbb{R}^+$ satisfies the stipulations in~\eqref{weightingcond}.  Define \begin{numcases}{\label{SwagT}\chi^2_{\emph{\mbox{\tiny OB-III}}}(\beta,b_{\infty}) :=} \frac{1}{\beta^{-1} - 1 } \int_{0}^{\beta^{-1} - 1} \left( \int_0^1 f(v) B_u(v) \diff v \right)^2 \, \diff u & $b_{\infty} = \infty$; \nonumber \\ \frac{1}{b_{\infty}}\sum_{j=1}^{b_{\infty}} \left( \int_0^1 f(v) B_{c_j/\beta}(v) \diff v\right)^2 \diff u & $b_{\infty} \in \mathbb{N}\setminus \{1\}$, \nonumber\end{numcases} where $c_j := (j-1)\frac{1-\beta}{b_{\infty}-1}$ and $$B_s(t) := W(s+t) - W(s) - t(W(s+1) - W(s)), \quad s \in [0,1-t], t \in [0,1].$$ Then, as $n \to \infty$, \begin{equation}\label{swag3} \hat{\sigma}^2_{\emph{\mbox{\tiny OB-III}}}(m_n,b_n) \inD \sigma^2\chi^2_{\emph{\mbox{\tiny OB-III}}}(\beta,b_{\infty}); \quad \emph{ and } \quad T_{\emph{\mbox{\tiny OB-III}}}(m_n,b_n) \inD \frac{W(1)}{\sqrt{\chi^2_{\emph{\mbox{\tiny OB-III}}}(\beta,b_{\infty})}}.\end{equation} \end{theorem}}

We conclude with a corresponding result in the small batch regime. 

\begin{theorem}[OB-III Small Batch Regime]\label{thm:ob3smbatch}  Suppose Assumptions~\ref{ass:stat}--\ref{ass:unif} hold, and that $\beta = \lim_{n \to \infty} m_n/n =0.$ Assume that the asymptotic number of batches $b_{\infty} := \lim_n b_n = \infty$, and that the weighting function $f:[0,1] \to \mathbb{R}^+$ satisfies the stipulations in~\eqref{weightingcond}. Then, as $n \to \infty$, \begin{equation}\label{statementsagain} \hat{\sigma}^2_{\emph{\mbox{\tiny OB-III}}}(m_n,b_n) \inP \sigma^2; \quad \emph{ and } \quad  T_{\emph{\mbox{\tiny OB-III}}}(m_n,b_n) \inD Z(0,1). \end{equation} \end{theorem}

\section{Considerations During Implementation}\label{sec:postscript}
In this section, we discuss ``practitioner'' questions that seem to arise repeatedly. 
\subsection{OB Critical Values versus Gaussian or Student's $t$ Critical Values.} In the absence of the OB-I and OB-II critical value tables on page 17 and page 24 respectively, it has been customary to use critical values from the $z$-table or the Student's $t$ table with an appropriate number of degrees of freedom. From a practical standpoint, how much difference does it make if one uses the $z$-table or the Student's $t$ table versus the OB critical value table? 

When the batch size is large, that is, if $\beta := \lim_n m_n/n >0$, and when the limiting number of batches $b_{\infty} = \infty$, the OB-I and OB-II critical values correspond to the rightmost columns of the tables appearing on page 17 and page 24, respectively. Looking at these columns, it should be immediately clear that the OB-I, OB-II critical values can be quite different from those of the standard normal distribution. For instance, when $\beta = 0.1$, the $0.95$-quantile of the OB-I and OB-II distributions are each around $1.76$ whereas the corresponding standard normal quantile $\Phi^{-1}(0.95) = 1.645$, a difference of more than $7 \%$. This difference increases as $\beta$ increases, and vanishes as $\beta \to 0$. 

When $\beta := \lim_n m_n/n >0$ but the limiting number of batches $b_{\infty} < \infty$, the natural temptation, in absence of the OB-I and OB-II distributions, might be to use the Student's $t$ critical value with $b_{\infty}-1$ degrees of freedom. (Some algebra reveals that when $\beta>0$, $b_{\infty} < \beta^{-1}$ results in non-overlapping batches and $b_{\infty} \geq \beta^{-1}$ results in overlapping batches.) However, notice again the quantiles reported on pages 17 and 24 can be quite different from the corresponding Student's $t$ critical value with $b_{\infty}-1$ degrees of freedom. For instance, when $\beta = 0.2$ and $b_{\infty}=51$, the $0.95$-quantile for the OB-I and OB-II distributions are $1.893$ and $1.902$ respectively, whereas the $0.95$-quantile of the Student's $t$ distribution with $50$ degrees of freedom is $1.6749$, a difference of more than $11 \%$. As $\beta \to 0$ and assuming $b_{\infty}< \infty$, the quantiles of the OB-II distribution converge to those of the Student's $t$ distribution with $b_{\infty}-1$ degrees of freedom; the difference between the quantiles of the $T_{\mbox{\tiny OB-I}}$ distribution and those of the Student's $t$ distribution with $b_{\infty}-1$ degrees of freedom persist even as $\beta \to 0$. 

In summary, substituting the normal or Student's $t$ critical value for the OB critical values will not provide the correct coverage unless $\beta =0.$ And, the deviation from the nominal coverage with such substitution can become substantial as the asymptotic batch size $\beta$ becomes large.

\subsection{Which OB CIP?}
We've presented three statistics along with their weak convergence limits OB-x, x=I,II,III, amounting to three possible CIPs. Numerical evidence to be provided in the ensuing section suggests that using these CIPs with large overlapping batches tends to result in confidence intervals having good behavior across a variety of contexts. How do the OB CIPs compare against each other?

Unfortunately, providing a satisfactory answer appears to be context-dependent and requires much further investigation, especially around the question of batch size choice. The sectioning estimator $\hat{\theta}_n$ used within the OB-I CIP typically has variance $O(\frac{1}{n})$ and bias $O(\frac{1}{n^{\lambda}})$ for some $\lambda \geq 1/2$, whereas the batching estimator $\bar{\theta}_n$ used within OB-II has typical variance $O(\frac{1}{b_nm_n})$ and bias $O(\frac{1}{m_n^{\lambda}})$. These expressions reveal that the batching estimator has lower variance (when using overlapping batches) and higher bias than the sectioning estimator; how these collude to decide the quality of the resulting confidence intervals is a context-dependent question. 

In summary, from the standpoint of interval quality as assessed by coverage probability and expected half-width, little is known theoretically on the relative behavior of OB-x, x=I,II,III especially when implemented with their corresponding optimal batch sizes. This should form the agenda for future investigation.

\begin{table}
\centering
\caption{Time complexities of the three OB CIPs. Recall that $n$ represents the size of the dataset, $m_n$ represents the batch size, $b_n$ represents the number of batches, and $\tau(i)$ is the time complexity of constructing the estimator of the statistical functional $\theta(P)$ using a batch of size $i$.   }
\begin{tabular}{l l}
\toprule
CIP & Time Complexity \\
\midrule
OB-I & $O(\tau(n) + b_n \tau(m_n))$ \\ \hline 
OB-II & $O(b_n \tau(m_n))$ \\ \hline
OB-III & $O\left(\tau(n) + b_n \sum_{i=2}^{m_n} \tau(i)\right) $ \\ \bottomrule
\end{tabular}\label{table:complexity}
\end{table}

The difference between the proposed procedures is much clearer from the standpoint of computational complexity. Suppose $\tau(|u-\ell|)$ is the time complexity of calculating the estimator $\hat{\theta}(\{X_j, \ell \leq j \leq u\})$ described in Section~\ref{sec:ptest}. Then, as can be seen in Table~\ref{table:complexity}, simple calculations reveal that OB-II CIP is the most computationally efficient and the OB-III CIP the least computationally efficient. The relative complexities of the three CIPs become stark when using large batches with significant overlap, that is, when $m_n/n \to \beta >0$ and $d_n = O(1)$. This leads to $O(n\tau(n))$ complexity for OB-I and OB-II, but $O(n\sum_{i=2}^{m_n} \tau(i))$ complexity for OB-III. With sparse overlap resulting in finite number of asymptotic batches, that is, if $b_{\infty} < \infty$, OB-I has complexity $O(\tau(n))$, OB-II has complexity $O(\tau(m_n))$, and OB-III has complexity $O(\tau(n) + \sum_{i=2}^{m_n} \tau(i) ).$  

\rp{An important qualification to the above discussion is that, depending on the specific context, the complexities listed in Table~\ref{table:complexity} can be conservative and should only be used as broad guidance. Specifically, in the sequential context where the data are revealed one (or a few) at a time, instead of all at once, the estimators $\hat{\theta}_n$ and $\hat{\theta}_{i,m_n}$ can often be constructed sequentially and in a way where the resulting complexities are much better than the ``one shot'' complexities listed in Table~\ref{table:complexity}. Nevertheless, we expect OB-III to be the most computationally expensive, and OB-II to be the least computationally expensive.}

\section{Numerical Illustration}\label{sec:numerical}

We now present numerical results from three popular contexts to gain further insight on the behavior of confidence intervals produced by OB-I, OB-II, and subsampling. 
\subsection{Example 1 : CVaR Estimation.} Let $\theta_{\gamma}$ the CVaR associated with the standard normal random variable. From the definition of CVaR~\cite{2008sarserury}, we have $$\theta_{\gamma} := \frac{1}{1-\gamma} \int_{q_{\gamma}}^{\infty} z \, \phi(z) \, dz; \quad q_{\gamma} := \Phi^{-1}(\gamma),$$ where $\phi(\cdot), \Phi(\cdot)$ are the standard normal density and cdf, respectively. With observations from an iid sequence $\{Z_n, n \geq 1\}$ of standard normal random variables, we can construct a point estimator for $\theta_{\gamma}$ as follows: $$\hat{\theta}(\{X_j, \ell \leq j \leq u\}) := \frac{1}{1-\gamma} \sum_{j=\ell}^u Z_j \, \mathbb{I}_{[q_{\gamma},\infty)}(Z_j).$$ We wish to construct a $0.95$-confidence interval on $\theta_{\gamma}$ for $\gamma=0.7,0.9,0.95$ with number of observations $n=100, 500, 1000, 2000, 3000$ and $5000$.

Tables~\ref{table:cvar70}--\ref{table:cvar95} display the estimated coverage probability (along with the estimated expected half-width in parenthesis) of confidence intervals constructed using fully-overlapping OB-I, OB-II CIPs having asymptotic batch size $\beta=0, 0.1, 0.25$, and using subsampling with the recommended~\cite{1999polromwol} sample size $m_n = \sqrt{n}.$ The coverage was estimated with a large ($m=100000$) number of replications.

\begin{table}\label{coverageprob1} \caption{The table summarizes coverage probabilities obtained using small batch OB-I ($\beta=0$), large batch OB-I ($\beta=0.1, 0.25$), small batch OB-II ($\beta=0$), large batch OB-II ($\beta=0.1, 0.25$), and subsampling (SS) for the   CVaR problem with $\gamma = 0.7$. The numbers in parenthesis are estimated expected half-widths of the confidence intervals.}
\centering
\begin{tabular}{|l|l|l|l|}
\hline  & OB-I ($\beta=0, 0.1, 0.25$) & OB-II ($\beta=0, 0.1, 0.25$) & SS ($m_n = \sqrt{n}$)\\
\hline $n = 100$ & \tabincell{l}{$\ \, 0.390 \quad 0.402 \quad  0.947$ \\ $(0.184) \, (0.212) \, (0.233)$} &\tabincell{l}{$\ \, 0.382 \quad 0.397 \quad  0.935$ \\ $(0.182) \, (0.213) \, (0.236)$}  & \tabincell{l}{$\ \, 0.388 $\\ $(0.205)$}  \\
\hline $n = 500$ & \tabincell{l}{$\ \, 0.897 \quad 0.954 \quad  0.950$ \\ $(0.085) \, (0.091) \, (0.099)$} &\tabincell{l}{$\ \, 0.887 \quad 0.944 \quad  0.935$ \\ $(0.085) \, (0.091) \, (0.099)$}  & \tabincell{l}{$\ \, 0.907 $\\ $(0.097)$}  \\
\hline $n = 1000$ & \tabincell{l}{$\ \, 0.946 \quad 0.952 \quad  0.951$ \\ $(0.060) \, (0.063) \, (0.070)$} &\tabincell{l}{$\ \, 0.938 \quad 0.944 \quad  0.938$ \\ $(0.060) \, (0.063) \, (0.071)$}  & \tabincell{l}{$\ \, 0.956 $\\ $(0.066)$}  \\
\hline $n = 2000$ & \tabincell{l}{$\ \, 0.950 \quad 0.952 \quad  0.951$ \\ $(0.042) \, (0.044) \, (0.049)$} &\tabincell{l}{$\ \, 0.943 \quad 0.945 \quad  0.936$ \\ $(0.042) \, (0.045) \, (0.050)$}  & \tabincell{l}{$\ \, 0.961 $\\ $(0.045)$}  \\
\hline $n = 3000$ & \tabincell{l}{$\ \, 0.949 \quad 0.950 \quad  0.950$ \\ $(0.034) \, (0.036) \, (0.040)$} &\tabincell{l}{$\ \, 0.945 \quad 0.944 \quad  0.937$ \\ $(0.034) \, (0.036) \, (0.040)$}  & \tabincell{l}{$\ \, 0.961 $\\ $(0.037)$}  \\
\hline $n = 5000$ & \tabincell{l}{$\ \, 0.950 \quad 0.952 \quad  0.954$ \\ $(0.026) \, (0.028) \, (0.032)$} &\tabincell{l}{$\ \, 0.947 \quad 0.943 \quad  0.936$ \\ $(0.026) \, (0.028) \, (0.031)$}  & \tabincell{l}{$\ \, 0.961 $\\ $(0.028)$}  \\
\hline
\end{tabular}\label{table:cvar70}
\end{table} 

\begin{table}\label{coverageprob2} \caption{The table summarizes coverage probabilities obtained using small batch OB-I ($\beta=0$), large batch OB-I ($\beta=0.1, 0.25$), small batch OB-II ($\beta=0$), large batch OB-II ($\beta=0.1,0.25$), and subsampling (SS) for the CVaR problem with $\gamma = 0.9$. The numbers in parenthesis are estimated expected half-widths of the confidence intervals.}
\centering
\begin{tabular}{|l|l|l|l|}
\hline  & OB-I ($\beta=0, 0.1, 0.25$) & OB-II ($\beta=0, 0.1, 0.25$) & SS ($m_n = \sqrt{n}$)\\
\hline $n = 100$ & \tabincell{l}{$\ \, 0.000 \quad 0.000 \quad  0.444$ \\ $(0.196) \, (0.227) \, (0.311)$} &\tabincell{l}{$\ \, 0.000 \quad 0.000 \quad  0.437$ \\ $(0.195) \, (0.228) \, (0.314)$}  & \tabincell{l}{$\ \, 0.000 $\\ $(0.233)$}  \\
\hline $n = 500$ & \tabincell{l}{$\ \, 0.001 \quad 0.755 \quad  0.949$ \\ $(0.110) \, (0.135) \, (0.142)$} &\tabincell{l}{$\ \, 0.001 \quad 0.744 \quad  0.932$ \\ $(0.110) \, (0.134) \, (0.141)$}  & \tabincell{l}{$\ \, 0.001 $\\ $(0.142)$}  \\
\hline $n = 1000$ & \tabincell{l}{$\ \, 0.014 \quad 0.953 \quad  0.950$ \\ $(0.084) \, (0.091) \, (0.098)$} &\tabincell{l}{$\ \, 0.014 \quad 0.944 \quad  0.937$ \\ $(0.084) \, (0.091) \, (0.100)$}  & \tabincell{l}{$\ \, 0.014 $\\ $(0.101)$}  \\
\hline $n = 2000$ & \tabincell{l}{$\ \, 0.134 \quad 0.952 \quad  0.950$ \\ $(0.062) \, (0.063) \, (0.069)$} &\tabincell{l}{$\ \, 0.132 \quad 0.945 \quad  0.935$ \\ $(0.062) \, (0.063) \, (0.069)$}  & \tabincell{l}{$\ \, 0.136 $\\ $(0.073)$}  \\
\hline $n = 3000$ & \tabincell{l}{$\ \, 0.350 \quad 0.951 \quad  0.948$ \\ $(0.050) \, (0.050) \, (0.055)$} &\tabincell{l}{$\ \, 0.345 \quad 0.944 \quad  0.936$ \\ $(0.050) \, (0.051) \, (0.056)$}  & \tabincell{l}{$\ \, 0.356 $\\ $(0.058)$}  \\
\hline $n = 5000$ & \tabincell{l}{$\ \, 0.705 \quad 0.952 \quad  0.952$ \\ $(0.039) \, (0.039) \, (0.044)$} &\tabincell{l}{$\ \, 0.698 \quad 0.943 \quad  0.935$ \\ $(0.039) \, (0.039) \, (0.043)$}  & \tabincell{l}{$\ \, 0.717 $\\ $(0.044)$}  \\
\hline
\end{tabular}\label{table:cvar90}
\end{table}

\begin{table}\label{coverageprob3} \caption{The table summarizes coverage probabilities obtained using small batch OB-I ($\beta=0$), large batch OB-I ($\beta=0.1,0.25$), small batch OB-II ($\beta=0$), large batch OB-II ($\beta=0.1,0.25$), and subsampling (SS) for the CVaR problem with $\gamma = 0.95$. The numbers in parenthesis are estimated expected half-widths of the confidence intervals.}
\centering
\begin{tabular}{|l|l|l|l|}
\hline  & OB-I ($\beta=0, 0.1, 0.25$) & OB-II ($\beta=0, 0.1, 0.25$) & SS ($m_n = \sqrt{n}$)\\
\hline $n = 100$ & \tabincell{l}{$\ \, \text{NA} \quad \text{NA} \quad  0.082$ \\ $(\text{NA}) \, (\text{NA}) \, (0.322)$} &\tabincell{l}{$\ \, \text{NA} \quad \text{NA} \quad  0.080$ \\ $(\text{NA}) \, (\text{NA}) \, (0.327)$}  & \tabincell{l}{$\ \, \text{NA} $\\ $(\text{NA})$}  \\
\hline $n = 500$ & \tabincell{l}{$\ \, \text{NA} \quad 0.096 \quad  0.915$ \\ $(\text{NA}) \, (0.161) \, (0.189)$} &\tabincell{l}{$\ \, \text{NA} \quad 0.094 \quad  0.899$ \\ $(\text{NA}) \, (0.160) \, (0.188)$}  & \tabincell{l}{$\ \, \text{NA} $\\ $(\text{NA})$}  \\
\hline $n = 1000$ & \tabincell{l}{$\ \, \text{NA} \quad 0.725 \quad  0.948$ \\ $(\text{NA}) \, (0.121) \, (0.129)$} &\tabincell{l}{$\ \, \text{NA} \quad 0.715 \quad  0.936$ \\ $(\text{NA}) \, (0.121) \, (0.130)$}  & \tabincell{l}{$\ \, \text{NA} $\\ $(\text{NA})$}  \\
\hline $n = 2000$ & \tabincell{l}{$\ \, \text{NA} \quad 0.953 \quad  0.951$ \\ $(\text{NA}) \, (0.083) \, (0.089)$} &\tabincell{l}{$\ \, \text{NA} \quad 0.943 \quad  0.935$ \\ $(\text{NA}) \, (0.083) \, (0.089)$}  & \tabincell{l}{$\ \, \text{NA} $\\ $(\text{NA})$}  \\
\hline $n = 3000$ & \tabincell{l}{$\ \, 0.000 \quad 0.952 \quad  0.948$ \\ $(0.061) \, (0.066) \, (0.071)$} &\tabincell{l}{$\ \, 0.000 \quad 0.943 \quad  0.934$ \\ $(0.061) \, (0.066) \, (0.072)$}  & \tabincell{l}{$\ \, 0.000 $\\ $(0.069)$}  \\
\hline $n = 5000$ & \tabincell{l}{$\ \, 0.000 \quad 0.953 \quad  0.952$ \\ $(0.050) \, (0.051) \, (0.056)$} &\tabincell{l}{$\ \, 0.000 \quad 0.943 \quad  0.934$ \\ $(0.050) \, (0.050) \, (0.056)$}  & \tabincell{l}{$\ \, 0.000 $\\ $(0.059)$}  \\
\hline
\end{tabular}\label{table:cvar95}
\end{table}

Tables~\ref{table:cvar70}--\ref{table:cvar95} display clear trends that will be repeated, more or less, across the different experiments we present. All methods seem to tend to the nominal coverage as the available data increases. However, OB-I and OB-II with $\beta >0$ seem to get to the nominal coverage much faster than the rest. For example, in Table~\ref{table:cvar70}, OB-I and OB-II with $\beta = 0.25$ seem to get to the vicinity of the nominal coverage after only about $n=100$ observations; and OB-I and OB-II with $\beta = 0.1$ seem to get to the vicinity of the nominal coverage after about $n=500$ observations. Similarly, in Table~\ref{table:cvar95}, OB-I and OB-II with $\beta = 0.25$ seem to get to the vicinity of the nominal coverage after  about $n=1000$ observations, while for $\beta=0.1$, the corresponding number is $n=2000$. The performance of OB confidence intervals with small batches seems comparable to that of subsampling; both OB-x with $\beta=0$ and subsampling seem to struggle on the CVaR problem with $\gamma =0.95$.  

 \subsection{Example 2: Parameter Estimation for AR($1$).} Consider the AR($1$) process given by \begin{align}
    X_t = c + \phi X_{t-1} + \epsilon_t, \quad \epsilon_t \stackrel{\scriptsize \textrm{iid}}{\sim} N(0, \sigma_{\epsilon}^2), \quad t = 1,2,\ldots \nonumber
\end{align} 

\begin{table}\label{coverageprob4} \caption{The table summarizes coverage probabilities obtained using small batch OB-I ($\beta=0$), large batch OB-I ($\beta=0.25$), small batch OB-II ($\beta=0$), large batch OB-II ($\beta=0.25$), and subsampling (SS) for the AR(1) problem with $\phi = 0.5, c = 0, \sigma_{\epsilon}^2 = 1$. The numbers in parenthesis are estimated expected half-widths of the confidence intervals.}
\centering
\begin{tabular}{|l|l|l|l|}
\hline  & OB-I ($\beta=0, 0.1, 0.25$) & OB-II ($\beta=0, 0.1, 0.25$) & SS ($m_n = \sqrt{n}$)\\
\hline $n = 100$ & \tabincell{l}{$\ \, 0.898 \quad 0.931 \quad 0.932$ \\ $(0.191) \, (0.222) \, (0.245)$} &\tabincell{l}{$\ \, 0.321 \quad 0.420 \quad  0.833$ \\ $(0.154) \, (0.180) \, (0.230)$}  & \tabincell{l}{$\ \, 0.644 $\\ $(0.110)$}  \\
\hline $n = 500$ & \tabincell{l}{$\ \, 0.931 \quad 0.947 \quad 0.944$ \\ $(0.089) \, (0.100) \, (0.110)$} &\tabincell{l}{$\ \, 0.366 \quad 0.821 \quad 0.908$ \\ $(0.081) \, (0.095) \, (0.109)$}  & \tabincell{l}{$\ \, 0.831 $\\ $(0.071)$}  \\
\hline $n = 1000$ & \tabincell{l}{$\ \, 0.940 \quad 0.949 \quad 0.946$ \\ $(0.064) \, (0.070) \, (0.078)$} &\tabincell{l}{$\ \, 0.376 \quad 0.883 \quad  0.924$ \\ $(0.059) \, (0.069) \, (0.079)$}  & \tabincell{l}{$\ \, 0.866 $\\ $(0.053)$}  \\
\hline $n = 5000$ & \tabincell{l}{$\ \, 0.946 \quad 0.951 \quad  0.948$ \\ $(0.029) \, (0.031) \, (0.035)$} &\tabincell{l}{$\ \, 0.419 \quad 0.933 \quad  0.931$ \\ $(0.028) \, (0.031) \, (0.035)$}  & \tabincell{l}{$\ \, 0.912 $\\ $(0.026)$}  \\
\hline $n = 10000$ & \tabincell{l}{$\ \, 0.949 \quad 0.947 \quad  0.950$ \\ $(0.021) \, (0.022) \, (0.025)$} &\tabincell{l}{$\ \, 0.425 \quad 0.935 \quad  0.936$ \\ $(0.020) \, (0.022) \, (0.025)$}  & \tabincell{l}{$\ \, 0.925 $\\ $(0.019)$}  \\
\hline
\end{tabular}\label{table:ar1phi5}
\end{table}

\begin{table}\label{coverageprob6} \caption{The table summarizes coverage probabilities obtained using small batch OB-I ($\beta=0$), large batch OB-I ($\beta=0.25$), small batch OB-II ($\beta=0$), large batch OB-II ($\beta=0.25$), and subsampling (SS) for the AR(1) problem with $\phi = 0.9, c = 0, \sigma_{\epsilon}^2 = 1$. The numbers in parenthesis are estimated expected half-widths of the confidence intervals.}
\centering
\begin{tabular}{|l|l|l|l|}
\hline  & OB-I ($\beta=0, 0.1, 0.25$) & OB-II ($\beta=0, 0.1, 0.25$) & SS ($m_n = \sqrt{n}$)\\
\hline $n = 100$ & \tabincell{l}{$\ \, 0.698 \quad 0.745 \quad 0.799$ \\ $(0.630) \, (0.726) \, (0.925)$} &\tabincell{l}{$\ \, \text{NA} \quad \text{NA} \quad  0.332$ \\ $(0.246) \, (0.288) \, (0.593)$}  & \tabincell{l}{$\ \, 0.045 $\\ $(0.005)$}  \\
\hline $n = 500$ & \tabincell{l}{$\ \, 0.859 \quad 0.901 \quad 0.912$ \\ $(0.365) \, (0.437) \, (0.505)$} &\tabincell{l}{$\ \, \text{NA} \quad 0.260 \quad 0.780$ \\ $(0.195) \, (0.323) \, (0.452)$}  & \tabincell{l}{$\ \, 0.282 $\\ $(0.086)$}  \\
\hline $n = 1000$ & \tabincell{l}{$\ \, 0.900 \quad 0.929 \quad 0.934$ \\ $(0.272) \, (0.320) \, (0.365)$} &\tabincell{l}{$\ \, \text{NA} \quad 0.546 \quad  0.859$ \\ $(0.169) \, (0.276) \, (0.350)$}  & \tabincell{l}{$\ \, 0.438 $\\ $(0.096)$}  \\
\hline $n = 5000$ & \tabincell{l}{$\ \, 0.924 \quad 0.942 \quad 0.945$ \\ $(0.131) \, (0.148) \, (0.165)$} &\tabincell{l}{$\ \, \text{NA} \quad 0.861 \quad  0.919$ \\ $(0.105) \, (0.144) \, (0.164)$}  & \tabincell{l}{$\ \, 0.743 $\\ $(0.083)$}  \\
\hline $n = 10000$ & \tabincell{l}{$\ \, 0.935 \quad 0.945 \quad 0.946$ \\ $(0.094) \, (0.104) \, (0.116)$} &\tabincell{l}{$\ \, 0.000 \quad 0.900 \quad  0.923$ \\ $(0.081) \, (0.103) \, (0.117)$}  & \tabincell{l}{$\ \, 0.816 $\\ $(0.068)$}  \\
\hline
\end{tabular}\label{table:ar1phi9}
\end{table}

With observations from the time series $\{X_n, n \geq 1\}$, the least-squares point estimator for $\theta(P) := \phi$ (after fixing $c=0$) is \begin{align} \hat{\theta}(\{X_j, \ell \leq j \leq u\}) := \underset{\phi \in \mathbb{R}}{\arg\min} \sum_{j=\ell}^{u-1} (X_{j+1} - \phi X_j)^2.\end{align} We wish to construct a $0.95$-confidence interval on $\phi=0.5, 0.9$ for $\sigma_{\epsilon}=1$ and with number of observations $n=100, 500, 1000, 5000$ and $10000$.

Tables~\ref{table:ar1phi5}--\ref{table:ar1phi9} are in the same format as Tables~\ref{table:cvar70}--\ref{table:cvar95} and display the results for the AR(1) example. The trends in coverage probabilities appear to be similar to those observed in Example 1, with large batches playing a seemingly important role in ensuring close to nominal coverage. Interestingly, Example 2 seems to do a better job in distinguishing between OB-I and OB-II for the same $\beta$, and in distinguishing between OB methods and subsampling. For example, due to the increased estimator bias associated with $\phi=0.9$, OB-I with $\beta=0,0.1,0.25$ appear to dominate OB-II with corresponding $\beta=0,0.1,0.25$. Subsampling clearly generates intervals with smaller expected half-width but the coverage is substantially lower than nominal especially when $\phi=0.9$. Such differences were not as evident in Example 1, probably because of the more muted effects of bias.   

\subsection{Example 3: Non-Homogeneous Poisson Process (NHPP) Rate Estimation}

In the final example, we consider a nonhomogeneous Poisson process $\{N(t), t \in [0,1]\}$~\cite{cinlar1} having rate $\lambda(t) = 4 + 8t, t \in [0,1]$. Suppose also that we have iid realizations of the process $\{N(t), t \in [0,1]\}$, using which we wish to construct a $0.95$-confidence interval on $$\theta_t(P):= \lambda(t) \mbox{ for } t=0.25, 0.5, 0.75.$$ We emphasize that this problem constructs ``marginal confidence intervals'' and is different from that of identifying a confidence region on the vector $(\lambda(0.25),\lambda(0.5),\lambda(0.75))$ or on the function $\lambda(t), t \in [0,1]$. The latter two problems, while very useful, lie outside the current paper's scope of real-valued $\theta(P)$. 

Given iid realizations $\{X_{j}(t), t \in [0,1]\}, j=1,2,\ldots,$ of $\{N(t), t \in [0,1]\}$, a simple point estimator for $\lambda(t)$ ($t$ fixed) can be constructed as follows: $$\hat{\theta}_t(\{X_j, \ell \leq j \leq u\}) := \frac{1}{u-\ell +1} \sum_{j=\ell}^u \frac{1}{\delta}\left(X_j(t+\delta) - X_j(t)\right).$$ The realizations $X_j, j=1,2,\ldots$ were generated using Algorithm 6 in~\cite{2011pasb,2011pasa}, the constant $\delta$ was fixed at $10^{-4}$, and the number of observations $n=1000, 2000, 5000, 10000, 20000$ and $50000$. 

Table~\ref{table:nhpp} presents results on coverage probability delivered by OB-I (with $\beta=0.1, 0.25$) and subsampling (with $m_n = \sqrt{n})$ for each of the three ``marginal'' confidence intervals associated with $t=0.25,0.5,0.75$. As in previous examples, the numbers in parenthesis refer to the estimated expected half-width. 

The trends in Table~\ref{table:nhpp} are consistent with those from the previous examples with OB-I delivering confidence intervals that are clearly better in terms of coverage, although nominal coverage seems to need a higher value of $n$ than in previous examples. Subsampling does not reach nominal coverage even with $n=50000$ although the generated intervals have much smaller half-widths.

\rp{\subsection{The Effect of Overlap in Batches}
Towards understanding the effect of the batch offset parameter $d_n$, we conducted additional numerical experiments for the ``more difficult versions'' of the CVaR problem ($\gamma=0.9$) and the AR(1) problem ($\phi=0.9$), with a dataset of size $n=1000$. As can be seen in Table~\ref{table:overlap_cvar} and Table~\ref{table:overlap_ar1}, various values of $d_n$ were chosen, expressed as a function of the batch size $m_n$ or the dataset size $n$.

The trends in Table~\ref{table:overlap_cvar} and Table~\ref{table:overlap_ar1} are interesting, although predictable. Increasing $d_n$ values clearly helps hasten the rate of coverage probability convergence (to nominal). This effect is pronounced in the small batch regime and most muted in the OB-I large batch regime.  Correspondingly, there is also an increase in the expected half-widths, with the small batch regime exhibiting a sharp rise when $d_n$ is very large, or correspondingly, the number of batches very small. Again, this effect is most muted in the OB-I large batch regime.

The main insight is that large $d_n$ values help with estimating the variance constant correctly in the small batch regime, since dependence between batch estimates reduces as $d_n$ increases. However, the price is larger half-widths due to the necessarily smaller number of batches. The large batch regime avoids this problem by modeling the dependence structure between batch estimates.}

\begin{table}\label{coverageprob8} \caption{The table summarizes coverage probabilities obtained using large batch OB-I ($\beta=0.1, 0.25$) and subsampling for the NHPP rate estimation problem with $\lambda(t) = 4 + 8t$. The three numbers in each column indicate the coverage probability estimates corresponding to a confidence interval on $\lambda(t)$ for $t = 0.25, 0.5, 0.75$. The numbers in parenthesis are estimated expected half-widths of the confidence intervals.}
\centering
\begin{tabular}{|l|l|l|}
\hline & OB-I ($\beta=0.1, 0.25$) & Subsampling ($m_n = \sqrt{n})$ \\
\hline $n = 1000$ & \tabincell{l}{$\ \, 0.447, \ \ 0.446; \ \ 0.542, \ \ 0.542; \ \ 0.629, \ \ 0.628$ \\ $(19.913, 78.657); \, (28.100, 106.502); \, (35.863, 135.759)$} & \tabincell{l}{$\ \, 0.315; \ \ 0.330; \ \ 0.353$\\ $(7.671); \, (8.040); \, (8.574)$}  \\
\hline $n = 2000$ & \tabincell{l}{$\ \, 0.694, \ \ 0.692; \ \ 0.793, \ \ 0.790; \ \ 0.863, \ \ 0.859$ \\ $(15.505, 42.562); \, (21.772, 57.227); \, (28.120, 70.158)$} & \tabincell{l}{$\ \, \text{NA}; \ \ 0.252; \ \ 0.453$\\ $(1.900); \, (2.889); \, (3.748)$}  \\
\hline $n = 5000$ & \tabincell{l}{$\ \, 0.937, \ \ 0.939; \ \ 0.912, \ \ 0.969; \ \ 0.962, \ \ 0.976 $ \\ $(11.582, 20.017); \, (15.546, 24.972); \, (19.767, 27.624)$} & \tabincell{l}{$\ \, 0.599; \ \ 0.549; \ \ 0.501$\\ $(2.220); \, (2.763); \, (3.243)$}  \\
\hline $n = 10000$ & \tabincell{l}{$\ \, 0.940, \ \ 0.966; \ \ 0.985, \ \ 0.974; \ \ 0.987, \ \ 0.974$ \\ $(8.686, 11.062); \, (11.210, 12.073); \, (13.189, 12.922)$} & \tabincell{l}{$\ \, 0.459; \ \ 0.628; \ \ 0.579$\\ $(1.765); \, (2.128); \, (2.419)$}  \\ 
\hline $n = 20000$ & \tabincell{l}{$\ \, 0.989, \ \ 0.967; \ \ 0.989, \ \ 0.963; \ \ 0.988, \ \ 0.962$ \\ $(5.917, 5.634); \, (6.910, 6.053); \, (7.730, 6.520)$} & \tabincell{l}{$\ \, 0.534; \ \ 0.620; \ \ 0.574$\\ $(1.358); \, (1.563); \,(1.744)$}  \\
\hline $n = 50000$ & \tabincell{l}{$\ \, 0.978, \ \ 0.957; \ \ 0.973, \ \ 0.959; \ \ 0.970, \ \ 0.955$ \\ $(2.963, 2.845); \, (3.319, 3.266); \, (3.547, 3.537)$} & \tabincell{l}{$\ \, 0.595; \ \ 0.537; \ \ 0.611$\\ $(0.877); \, (1.007); \,(1.220)$}  \\
\hline
\end{tabular}\label{table:nhpp}
\end{table}

\begin{table} \caption{The table summarizes coverage probabilities obtained using small batch OB-I using student's $t$ with $b_n - 1$ degrees of freedom, large batch OB-I ($\beta=0.25$), small batch OB-II using student's t with $b_n - 1$ degrees of freedom, and large batch OB-II ($\beta=0.25$) for the CVaR problem with $\gamma = 0.9, n = 1000, m = 10000$. The numbers in parenthesis are estimated expected half-widths of the confidence intervals.}
\centering
\begin{tabular}{|l|l|l|l|l|}
\hline  & OB-I ($\beta=0, 0.25$) & OB-II ($\beta=0, 0.25$) & $b_n$ ($\beta=0, 0.25$) & $b_{\infty}$ ($\beta=0, 0.25$)\\
\hline $d_n = 1$ & \tabincell{l}{$\ 0.013 \quad 0.950$ \\ $(0.085) \, (0.099)$} &\tabincell{l}{$\ 0.013 \quad 0.937$ \\ $(0.085) \, (0.099)$} & $\ \, 970 \quad 751$ & $\ \, \infty \quad \infty$\\
\hline $d_n = \sqrt{m_n}$ & \tabincell{l}{$\ 0.032 \quad 0.947$ \\ $(0.086) \, (0.098)$} &\tabincell{l}{$\ 0.031 \quad 0.936$ \\ $(0.085) \, (0.100)$}  & $\ \, 194 \quad 51$ & $\ \, \infty \quad \infty$\\
\hline $d_n = \frac{m_n}{4}$ & \tabincell{l}{$\ 0.050 \quad 0.953$ \\ $(0.086) \, (0.101)$} &\tabincell{l}{$\ 0.049 \quad 0.941$ \\ $(0.086) \, (0.101)$} & $\ \, 139 \quad 13$ & $\ \, \infty \quad 13$\\
\hline $d_n = \frac{m_n}{2}$ & \tabincell{l}{$\ 0.110 \quad 0.951$ \\ $(0.090) \, (0.105)$} &\tabincell{l}{$\ 0.109 \quad 0.947$ \\ $(0.090) \, (0.105)$} & $\ \, 65 \quad 7$ & $\ \, \infty \quad 7$\\
\hline $d_n = \frac{3m_n}{4}$ & \tabincell{l}{$\ 0.202 \quad 0.955$ \\ $(0.091) \, (0.107)$} &\tabincell{l}{$\ 0.196 \quad 0.945$ \\ $(0.090) \, (0.110)$}  & $\ \, 43 \quad 5$ & $\ \, \infty \quad 5$\\
\hline $d_n = m_n$ & \tabincell{l}{$\ 0.275 \quad 0.948$ \\ $(0.092) \, (0.121)$} &\tabincell{l}{$\ 0.268 \quad 0.949$ \\ $(0.092) \, (0.122)$}  & $\ \, 32 \quad 4$ & $\ \, \infty \quad 4$\\
\hline $d_n = 0.1 n$ & \tabincell{l}{$\ 0.643 \quad 0.946$ \\ $(0.099) \, (0.101)$} &\tabincell{l}{$\ 0.491 \quad 0.931$ \\ $(0.095) \, (0.101)$}  & $\ \, 10 \quad 8$ & $\ \, 11 \quad 8$\\
\hline $d_n = 0.2 n$ & \tabincell{l}{$\ 0.789 \quad 0.962$ \\ $(0.118) \, (0.125)$} &\tabincell{l}{$\ 0.499 \quad 0.919$ \\ $(0.106) \, (0.115)$} & $\ \, 5 \quad 4$ & $\ \, 6 \quad 4$\\
\hline $d_n = 0.5 n$ & \tabincell{l}{$\ 0.919 \quad 0.949$ \\ $(0.499) \, (0.163)$} &\tabincell{l}{$\ 0.664 \quad 0.927$ \\ $(0.319) \, (0.429)$}  & $\ \, 2 \quad 2$ & $\ \, 3 \quad 2$\\
\hline
\end{tabular}\label{table:overlap_cvar}
\end{table}

\begin{table} \caption{The table summarizes coverage probabilities obtained using small batch OB-I ($\beta=0$), small batch OB-I using student's t with $b_n - 1$ degrees of freedom, large batch OB-I ($\beta=0.25$), small batch OB-II ($\beta=0$), small batch OB-II using student's t with $b_n - 1$ degrees of freedom, and large batch OB-II ($\beta=0.25$) for the AR(1) problem with $\phi = 0.9, c = 0, \sigma_{\epsilon}^2 = 1$, $n = 1000, m = 10000$. The numbers in parenthesis are estimated expected half-widths of the confidence intervals.}
\centering
\begin{tabular}{|l|l|l|l|l|}
\hline  & OB-I ($\beta=0, 0.25$) & OB-II ($\beta=0, 0.25$) & $b_n$ ($\beta=0, 0.25$) & $b_{\infty}$ ($\beta=0, 0.25$)\\
\hline $d_n = 1$ & \tabincell{l}{$\ 0.895 \quad 0.930$ \\ $(0.272) \, (0.366)$} &\tabincell{l}{$\ \text{NA} \quad 0.855$ \\ $(0.169) \, (0.348)$} & $\ \, 970 \quad 751$ & $\ \, \infty \quad \infty$\\
\hline $d_n = \sqrt{m_n}$ & \tabincell{l}{$\ 0.900 \quad 0.936$ \\ $(0.274) \, (0.364)$} &\tabincell{l}{$\ \text{NA} \quad 0.862$ \\ $(0.170) \, (0.348)$}  & $\ \, 194 \quad 51$ & $\ \, \infty \quad \infty$\\
\hline $d_n = \frac{m_n}{4}$ & \tabincell{l}{$\ 0.903 \quad 0.937$ \\ $(0.275) \, (0.372)$} &\tabincell{l}{$\ \text{NA} \quad 0.868$ \\ $ (0.171) \, (0.351)$} & $\ \, 139 \quad 13$ & $\ \, \infty \quad 13$\\
\hline $d_n = \frac{m_n}{2}$ & \tabincell{l}{$\ 0.903 \quad 0.943$ \\ $(0.277) \, (0.392)$} &\tabincell{l}{$\ \text{NA} \quad 0.885$ \\ $(0.172) \, (0.373)$}  & $\ \, 65 \quad 7$ & $\ \, \infty \quad 7$\\
\hline $d_n = \frac{3m_n}{4}$ & \tabincell{l}{$\ 0.908 \quad 0.940$ \\ $(0.280) \, (0.400)$} &\tabincell{l}{$\ \text{NA} \quad 0.879$ \\ $(0.173) \, (0.388)$}  & $\ \, 43 \quad 5$ & $\ \, \infty \quad 5$\\
\hline $d_n = m_n$ & \tabincell{l}{$\ 0.902 \quad 0.947$ \\ $(0.283) \, (0.463)$} &\tabincell{l}{$\ \text{NA} \quad 0.892$ \\ $(0.174) \, (0.430)$}  & $\ \, 32 \quad 4$ & $\ \, \infty \quad 4$\\
\hline $d_n = 0.1n$ & \tabincell{l}{$\ 0.923 \quad 0.937$ \\ $(0.311) \, (0.375)$} &\tabincell{l}{$\ 0.014 \quad 0.861$ \\ $(0.182) \, (0.356)$}  & $\ \, 10 \quad 8$ & $\ \, 11 \quad 8$\\
\hline $d_n = 0.2n$ & \tabincell{l}{$\ 0.956 \quad 0.954$ \\ $(0.379) \, (0.465)$} &\tabincell{l}{$\ 0.068 \quad 0.867$ \\ $(0.202) \, (0.412)$}  & $\ \, 5 \quad 4$ & $\ \, 6 \quad 4$\\
\hline $d_n = 0.5n$ & \tabincell{l}{$\ 0.996 \quad 0.947$ \\ $(1.671) \, (0.611)$} &\tabincell{l}{$\ 0.400 \quad 0.918$ \\ $(0.606) \, (1.514)$}  & $\ \, 2 \quad 2$ & $\ \, 3 \quad 2$\\
\hline
\end{tabular}\label{table:overlap_ar1}
\end{table} 

\section*{Acknowledgments.} This work was greatly influenced by conversations and original ideas of Bruce Schmeiser, especially in viewing the proposed statistics in analogy with the $\chi^2_{\nu}$ and Student's $t$ distributions. Raghu Pasupathy also gratefully acknowledges the Office of Naval Research for support provided through the grants N000141712295 and 13000991.

\bibliographystyle{ACM-Reference-Format}
\bibliography{stochastic_optimization,adaptiveSeqSamp}


\begin{thebibliography}{75}


\ifx \showCODEN    \undefined \def \showCODEN     #1{\unskip}     \fi
\ifx \showDOI      \undefined \def \showDOI       #1{#1}\fi
\ifx \showISBNx    \undefined \def \showISBNx     #1{\unskip}     \fi
\ifx \showISBNxiii \undefined \def \showISBNxiii  #1{\unskip}     \fi
\ifx \showISSN     \undefined \def \showISSN      #1{\unskip}     \fi
\ifx \showLCCN     \undefined \def \showLCCN      #1{\unskip}     \fi
\ifx \shownote     \undefined \def \shownote      #1{#1}          \fi
\ifx \showarticletitle \undefined \def \showarticletitle #1{#1}   \fi
\ifx \showURL      \undefined \def \showURL       {\relax}        \fi
\providecommand\bibfield[2]{#2}
\providecommand\bibinfo[2]{#2}
\providecommand\natexlab[1]{#1}
\providecommand\showeprint[2][]{arXiv:#2}

\bibitem[{Aktaran-Kalayc\i } et~al\mbox{.}(2009)]%
        {2009aktalegolwil}
\bibfield{author}{\bibinfo{person}{T. {Aktaran-Kalayc\i }}, \bibinfo{person}{C.
  Alexopoulos}, \bibinfo{person}{D. Goldsman}, {and} \bibinfo{person}{J.~R.
  Wilson}.} \bibinfo{year}{2009}\natexlab{}.
\newblock \showarticletitle{Optimal Linear Combinations of Overlapping Variance
  Estimators for Steady-State Simulation}.
\newblock In \bibinfo{booktitle}{\emph{Advancing the Frontiers of Simulation}},
  \bibfield{editor}{\bibinfo{person}{C.~Alexopoulos},
  \bibinfo{person}{D.~Goldsman}, {and} \bibinfo{person}{J.~R. Wilson}} (Eds.).
  \bibinfo{publisher}{Springer, NY}.
\newblock


\bibitem[Alexopoulos et~al\mbox{.}(2007)]%
        {2007aleetal}
\bibfield{author}{\bibinfo{person}{Christos Alexopoulos},
  \bibinfo{person}{Nilay~Tanik Argon}, \bibinfo{person}{David Goldsman},
  \bibinfo{person}{Gamze Tokol}, {and} \bibinfo{person}{James~R Wilson}.}
  \bibinfo{year}{2007}\natexlab{}.
\newblock \showarticletitle{Overlapping Variance Estimators for Simulation}.
\newblock \bibinfo{journal}{\emph{Operations research}} \bibinfo{volume}{55},
  \bibinfo{number}{6} (\bibinfo{year}{2007}), \bibinfo{pages}{1090--1103}.
\newblock
\showISSN{0030-364X}


\bibitem[Asmussen and Glynn(2007)]%
        {2007asmgly}
\bibfield{author}{\bibinfo{person}{S. Asmussen} {and} \bibinfo{person}{P.~W.
  Glynn}.} \bibinfo{year}{2007}\natexlab{}.
\newblock \bibinfo{booktitle}{\emph{Stochastic Simulation: Algorithms and
  Analysis}}.
\newblock \bibinfo{publisher}{Springer}, \bibinfo{address}{New York, NY}.
\newblock


\bibitem[Bartlett(1950)]%
        {1950bar}
\bibfield{author}{\bibinfo{person}{M.S. Bartlett}.}
  \bibinfo{year}{1950}\natexlab{}.
\newblock \showarticletitle{Periodogram analysis and continuous spectra}.
\newblock \bibinfo{journal}{\emph{Biometrika}} \bibinfo{number}{37}
  (\bibinfo{year}{1950}), \bibinfo{pages}{1--16}.
\newblock


\bibitem[Billingsley(1995)]%
        {1995bil}
\bibfield{author}{\bibinfo{person}{P. Billingsley}.}
  \bibinfo{year}{1995}\natexlab{}.
\newblock \bibinfo{booktitle}{\emph{Probability and Measure}}.
\newblock \bibinfo{publisher}{Wiley}, \bibinfo{address}{New York, NY}.
\newblock


\bibitem[Billingsley(1999)]%
        {1999bil}
\bibfield{author}{\bibinfo{person}{Patrick Billingsley}.}
  \bibinfo{year}{1999}\natexlab{}.
\newblock \bibinfo{booktitle}{\emph{Convergence of probability measures}
  (\bibinfo{edition}{2nd ed.} ed.)}.
\newblock \bibinfo{publisher}{Wiley}, \bibinfo{address}{New York}.
\newblock
\showISBNx{0471197459}
\showLCCN{99030372}


\bibitem[Bose(1990)]%
        {1990bos}
\bibfield{author}{\bibinfo{person}{A Bose}.} \bibinfo{year}{1990}\natexlab{}.
\newblock \showarticletitle{Bootstrap in moving average models}.
\newblock \bibinfo{journal}{\emph{Annals of the Institute of Statistical
  Mathematics}} \bibinfo{volume}{42}, \bibinfo{number}{4}
  (\bibinfo{year}{1990}), \bibinfo{pages}{753--768}.
\newblock
\showISSN{0020-3157}


\bibitem[Calvin and Nakayama(2013)]%
        {2013calnak}
\bibfield{author}{\bibinfo{person}{James Calvin} {and} \bibinfo{person}{Marvin
  Nakayama}.} \bibinfo{year}{2013}\natexlab{}.
\newblock \showarticletitle{Confidence intervals for quantiles with
  standardized time series}. In \bibinfo{booktitle}{\emph{Proceedings of the
  2013 Winter Simulation Conference}} \emph{(\bibinfo{series}{WSC '13})}.
  \bibinfo{publisher}{IEEE Press}, \bibinfo{pages}{601--612}.
\newblock
\showISBNx{9781479920778}
\showISSN{0891-7736}


\bibitem[Calvin and Nakayama(2006)]%
        {2006calnak}
\bibfield{author}{\bibinfo{person}{James~M Calvin} {and}
  \bibinfo{person}{Marvin~K Nakayama}.} \bibinfo{year}{2006}\natexlab{}.
\newblock \showarticletitle{Permuted Standardized Time Series for Steady-State
  Simulations}.
\newblock \bibinfo{journal}{\emph{Mathematics of operations research}}
  \bibinfo{volume}{31}, \bibinfo{number}{2} (\bibinfo{year}{2006}),
  \bibinfo{pages}{351--368}.
\newblock
\showISSN{0364-765X}


\bibitem[\c{C}inlar(1975)]%
        {cinlar1}
\bibfield{author}{\bibinfo{person}{E. \c{C}inlar}.}
  \bibinfo{year}{1975}\natexlab{}.
\newblock \bibinfo{booktitle}{\emph{Introduction to Stochastic Processes}}.
\newblock \bibinfo{publisher}{Prentice-Hall}, \bibinfo{address}{New Jersey}.
\newblock


\bibitem[Chu and Nakayama(2012)]%
        {2012chunak}
\bibfield{author}{\bibinfo{person}{F. Chu} {and} \bibinfo{person}{M.~K.
  Nakayama}.} \bibinfo{year}{2012}\natexlab{}.
\newblock \showarticletitle{Confidence intervals for quantiles when applying
  variance-reduction techniques}.
\newblock \bibinfo{journal}{\emph{ACM Transactions on Modeling and Computer
  Simulation (TOMACS)}} \bibinfo{volume}{22}, \bibinfo{number}{2}
  (\bibinfo{year}{2012}), \bibinfo{pages}{1--25}.
\newblock


\bibitem[Conway(1963)]%
        {1963con}
\bibfield{author}{\bibinfo{person}{R.~W. Conway}.}
  \bibinfo{year}{1963}\natexlab{}.
\newblock \showarticletitle{Some tactical problems in digital simulation}.
\newblock \bibinfo{journal}{\emph{Management science}} \bibinfo{volume}{10},
  \bibinfo{number}{1} (\bibinfo{year}{1963}), \bibinfo{pages}{47--61}.
\newblock


\bibitem[Crane and Lemoine(1977)]%
        {1977cralem}
\bibfield{author}{\bibinfo{person}{M.A. Crane} {and} \bibinfo{person}{A.~J.
  Lemoine}.} \bibinfo{year}{1977}\natexlab{}.
\newblock \bibinfo{booktitle}{\emph{An Introduction to the Regenerative Method
  for Simulation Analysis} (\bibinfo{edition}{1st ed. 1977.} ed.)}.
\newblock \bibinfo{publisher}{Springer Berlin Heidelberg},
  \bibinfo{address}{Berlin, Heidelberg}.
\newblock
\showISBNx{3-540-37175-3}


\bibitem[Cryer and Chan(2008)]%
        {2008crycha}
\bibfield{author}{\bibinfo{person}{J.~D. Cryer} {and} \bibinfo{person}{{K-S}.
  Chan}.} \bibinfo{year}{2008}\natexlab{}.
\newblock \bibinfo{booktitle}{\emph{Time Series Analysis With Applications in
  R} (\bibinfo{edition}{2nd ed. 2008.} ed.)}.
\newblock \bibinfo{publisher}{Springer New York}, \bibinfo{address}{New York,
  NY}.
\newblock
\showISBNx{0-387-75959-X}


\bibitem[Cs\"{o}rg\"{o} and R\'{e}v\'{e}sz(1981)]%
        {1981csorev}
\bibfield{author}{\bibinfo{person}{M. Cs\"{o}rg\"{o}} {and}
  \bibinfo{person}{R\'{e}v\'{e}sz}.} \bibinfo{year}{1981}\natexlab{}.
\newblock \bibinfo{booktitle}{\emph{Strong approximations in probability and
  statistics}}.
\newblock \bibinfo{publisher}{Academic Press}, \bibinfo{address}{New York}.
\newblock
\showISBNx{1-322-55862-0}
\showLCCN{79057112}


\bibitem[Damerdji(1991)]%
        {1991dam}
\bibfield{author}{\bibinfo{person}{Halim Damerdji}.}
  \bibinfo{year}{1991}\natexlab{}.
\newblock \showarticletitle{Strong Consistency and Other Properties of the
  Spectral Variance Estimator}.
\newblock \bibinfo{journal}{\emph{Management Science}} \bibinfo{volume}{37},
  \bibinfo{number}{11} (\bibinfo{year}{1991}), \bibinfo{pages}{1424--1440}.
\newblock
\showISSN{0025-1909}


\bibitem[Damerdji(1994)]%
        {1994dam}
\bibfield{author}{\bibinfo{person}{Halim Damerdji}.}
  \bibinfo{year}{1994}\natexlab{}.
\newblock \showarticletitle{Strong Consistency of the Variance Estimator in
  Steady-State Simulation Output Analysis}.
\newblock \bibinfo{journal}{\emph{Mathematics of Operations Research}}
  \bibinfo{volume}{19}, \bibinfo{number}{2} (\bibinfo{year}{1994}),
  \bibinfo{pages}{494--512}.
\newblock
\showISSN{0364-765X}


\bibitem[Damerdji(1995)]%
        {1995dam}
\bibfield{author}{\bibinfo{person}{Halim Damerdji}.}
  \bibinfo{year}{1995}\natexlab{}.
\newblock \showarticletitle{Mean-Square Consistency of the Variance Estimator
  in Steady-State Simulation Output Analysis}.
\newblock  \bibinfo{volume}{43}, \bibinfo{number}{2} (\bibinfo{year}{1995}),
  \bibinfo{pages}{282--291}.
\newblock
\showISSN{0030-364X}


\bibitem[Das{G}upta(2011)]%
        {das2011}
\bibfield{author}{\bibinfo{person}{A. Das{G}upta}.}
  \bibinfo{year}{2011}\natexlab{}.
\newblock \bibinfo{booktitle}{\emph{Probability for Statistics and Machine
  Learning}}.
\newblock \bibinfo{publisher}{Springer}.
\newblock


\bibitem[Davison(1997)]%
        {1997davhin}
\bibfield{author}{\bibinfo{person}{A.~C. (Anthony~Christopher) Davison}.}
  \bibinfo{year}{1997}\natexlab{}.
\newblock \bibinfo{booktitle}{\emph{Bootstrap methods and their application}}.
\newblock \bibinfo{publisher}{Cambridge University Press},
  \bibinfo{address}{Cambridge ;}.
\newblock
\showISBNx{0521573912}
\showLCCN{96030064}


\bibitem[Dedecker and Merleve\`{e}de(2022)]%
        {2022dedmer}
\bibfield{author}{\bibinfo{person}{J. Dedecker} {and} \bibinfo{person}{F
  Merleve\`{e}de}.} \bibinfo{year}{2022}\natexlab{}.
\newblock \showarticletitle{Central limit theorem and almost sure results for
  the empirical estimator of superquantiles/CVaR in the stationary case}.
\newblock \bibinfo{journal}{\emph{Statistics}} \bibinfo{volume}{0},
  \bibinfo{number}{0} (\bibinfo{year}{2022}), \bibinfo{pages}{1--20}.
\newblock
\urldef\tempurl%
\url{https://doi.org/10.1080/02331888.2022.2043325}
\showDOI{\tempurl}
\showeprint{https://doi.org/10.1080/02331888.2022.2043325}


\bibitem[Diciccio and Romano(1988)]%
        {1988dicrom}
\bibfield{author}{\bibinfo{person}{Thomas~J Diciccio} {and}
  \bibinfo{person}{Joseph~P Romano}.} \bibinfo{year}{1988}\natexlab{}.
\newblock \showarticletitle{A review of bootstrap confidence intervals}.
\newblock \bibinfo{journal}{\emph{Journal of the Royal Statistical Society:
  Series B (Methodological)}} \bibinfo{volume}{50}, \bibinfo{number}{3}
  (\bibinfo{year}{1988}), \bibinfo{pages}{338--354}.
\newblock


\bibitem[Dong and Nakayama(2014)]%
        {2014donnak}
\bibfield{author}{\bibinfo{person}{H. Dong} {and} \bibinfo{person}{M.~K.
  Nakayama}.} \bibinfo{year}{2014}\natexlab{}.
\newblock \showarticletitle{Constructing confidence intervals for a quantile
  using batching and sectioning when applying Latin hypercube sampling}. In
  \bibinfo{booktitle}{\emph{Proceedings of the Winter Simulation Conference
  2014}}. IEEE, \bibinfo{pages}{640--651}.
\newblock


\bibitem[Dong and Nakayama(2018)]%
        {2018donnak}
\bibfield{author}{\bibinfo{person}{Hui Dong} {and} \bibinfo{person}{Marvin~K
  Nakayama}.} \bibinfo{year}{2018}\natexlab{}.
\newblock \showarticletitle{A tutorial on quantile estimation via Monte Carlo}.
  In \bibinfo{booktitle}{\emph{International Conference on Monte Carlo and
  Quasi-Monte Carlo Methods in Scientific Computing}}. Springer,
  \bibinfo{pages}{3--30}.
\newblock


\bibitem[Dong and Nakayama(2020)]%
        {2020huinak}
\bibfield{author}{\bibinfo{person}{Hui Dong} {and} \bibinfo{person}{Marvin~K
  Nakayama}.} \bibinfo{year}{2020}\natexlab{}.
\newblock \showarticletitle{A Tutorial on Quantile Estimation via Monte Carlo}.
\newblock In \bibinfo{booktitle}{\emph{Monte Carlo and Quasi-Monte Carlo
  Methods}}. \bibinfo{publisher}{Springer International Publishing},
  \bibinfo{address}{Cham}, \bibinfo{pages}{3--30}.
\newblock
\showISBNx{9783030434649}
\showISSN{2194-1009}


\bibitem[Durrett(2010)]%
        {2010dur}
\bibfield{author}{\bibinfo{person}{R. Durrett}.}
  \bibinfo{year}{2010}\natexlab{}.
\newblock \bibinfo{booktitle}{\emph{Probability: Theory and Examples}}.
\newblock \bibinfo{publisher}{Cambridge University Press},
  \bibinfo{address}{New York, NY}.
\newblock


\bibitem[Efron(1979)]%
        {1979efr}
\bibfield{author}{\bibinfo{person}{Bradley Efron}.}
  \bibinfo{year}{1979}\natexlab{}.
\newblock \showarticletitle{Bootstrap Methods: Another Look at the Jackknife}.
\newblock In \bibinfo{booktitle}{\emph{Breakthroughs in Statistics}}.
  \bibinfo{publisher}{Springer New York}, \bibinfo{address}{New York, NY},
  \bibinfo{pages}{569--593}.
\newblock
\showISBNx{9780387940397}
\showISSN{0172-7397}


\bibitem[Efron(1981)]%
        {1981efr}
\bibfield{author}{\bibinfo{person}{Bradley Efron}.}
  \bibinfo{year}{1981}\natexlab{}.
\newblock \showarticletitle{Nonparametric standard errors and confidence
  intervals}.
\newblock \bibinfo{journal}{\emph{canadian Journal of Statistics}}
  \bibinfo{volume}{9}, \bibinfo{number}{2} (\bibinfo{year}{1981}),
  \bibinfo{pages}{139--158}.
\newblock


\bibitem[Efron(1982)]%
        {1982efr}
\bibfield{author}{\bibinfo{person}{Bradley Efron}.}
  \bibinfo{year}{1982}\natexlab{}.
\newblock \bibinfo{booktitle}{\emph{The jackknife, the bootstrap and other
  resampling plans}}.
\newblock \bibinfo{publisher}{SIAM}.
\newblock


\bibitem[Efron(1985)]%
        {1985efr}
\bibfield{author}{\bibinfo{person}{Bradley Efron}.}
  \bibinfo{year}{1985}\natexlab{}.
\newblock \showarticletitle{Bootstrap confidence intervals for a class of
  parametric problems}.
\newblock \bibinfo{journal}{\emph{Biometrika}} \bibinfo{volume}{72},
  \bibinfo{number}{1} (\bibinfo{year}{1985}), \bibinfo{pages}{45--58}.
\newblock


\bibitem[Efron(1987)]%
        {1987efr}
\bibfield{author}{\bibinfo{person}{Bradley Efron}.}
  \bibinfo{year}{1987}\natexlab{}.
\newblock \showarticletitle{Better bootstrap confidence intervals}.
\newblock \bibinfo{journal}{\emph{Journal of the American statistical
  Association}} \bibinfo{volume}{82}, \bibinfo{number}{397}
  (\bibinfo{year}{1987}), \bibinfo{pages}{171--185}.
\newblock


\bibitem[Efron(1992)]%
        {1992efr}
\bibfield{author}{\bibinfo{person}{Bradley Efron}.}
  \bibinfo{year}{1992}\natexlab{}.
\newblock \showarticletitle{Bootstrap methods: another look at the jackknife}.
\newblock In \bibinfo{booktitle}{\emph{Breakthroughs in statistics}}.
  \bibinfo{publisher}{Springer}, \bibinfo{pages}{569--593}.
\newblock


\bibitem[Efron and Tibshirani(1998)]%
        {1998efrtib}
\bibfield{author}{\bibinfo{person}{B. Efron} {and} \bibinfo{person}{R.~J.
  Tibshirani}.} \bibinfo{year}{1998}\natexlab{}.
\newblock \bibinfo{booktitle}{\emph{An Introduction to the Bootstrap}}.
\newblock \bibinfo{publisher}{Chapman \& Hall/CRC}, \bibinfo{address}{Boca
  Raton, FL}.
\newblock


\bibitem[Ethier and Kurtz(2009)]%
        {2009ethkur}
\bibfield{author}{\bibinfo{person}{S.~N. Ethier} {and} \bibinfo{person}{T.~G.
  Kurtz}.} \bibinfo{year}{2009}\natexlab{}.
\newblock \bibinfo{booktitle}{\emph{Markov processes: Characterization and
  Convergence}}.
\newblock \bibinfo{publisher}{John Wiley \& Sons}.
\newblock


\bibitem[Fishman(1978)]%
        {1979fis}
\bibfield{author}{\bibinfo{person}{G.~S. Fishman}.}
  \bibinfo{year}{1978}\natexlab{}.
\newblock \showarticletitle{Grouping observations in digital simulation}.
\newblock \bibinfo{journal}{\emph{Management Science}} \bibinfo{volume}{24},
  \bibinfo{number}{5} (\bibinfo{year}{1978}), \bibinfo{pages}{510--521}.
\newblock


\bibitem[Foley and Goldsman(1999)]%
        {1999folgol}
\bibfield{author}{\bibinfo{person}{Robert Foley} {and} \bibinfo{person}{David
  Goldsman}.} \bibinfo{year}{1999}\natexlab{}.
\newblock \showarticletitle{Confidence intervals using orthonormally weighted
  standardized time series}.
\newblock  \bibinfo{volume}{9}, \bibinfo{number}{4} (\bibinfo{year}{1999}),
  \bibinfo{pages}{297--325}.
\newblock
\showISSN{1049-3301}


\bibitem[Fox et~al\mbox{.}(1991)]%
        {1991foxgolswa}
\bibfield{author}{\bibinfo{person}{B.~L. Fox}, \bibinfo{person}{D. Goldsman},
  {and} \bibinfo{person}{J.~J. Swain}.} \bibinfo{year}{1991}\natexlab{}.
\newblock \showarticletitle{Spaced batch means}.
\newblock \bibinfo{journal}{\emph{Operations Research Letters}}
  \bibinfo{volume}{10} (\bibinfo{year}{1991}), \bibinfo{pages}{255--263}.
\newblock


\bibitem[Gastwirth and Rubin(1975)]%
        {1975gasrub}
\bibfield{author}{\bibinfo{person}{J.~L. Gastwirth} {and} \bibinfo{person}{H.
  Rubin}.} \bibinfo{year}{1975}\natexlab{}.
\newblock \showarticletitle{The behavior of robust estimators on dependent
  data}.
\newblock \bibinfo{journal}{\emph{The Annals of Statistics}}
  (\bibinfo{year}{1975}), \bibinfo{pages}{1070--1100}.
\newblock


\bibitem[Gin{\'e} and Zinn(1989)]%
        {1989ginzin}
\bibfield{author}{\bibinfo{person}{Evarist Gin{\'e}} {and}
  \bibinfo{person}{Joel Zinn}.} \bibinfo{year}{1989}\natexlab{}.
\newblock \showarticletitle{Necessary conditions for the bootstrap of the
  mean}.
\newblock \bibinfo{journal}{\emph{The annals of statistics}}
  (\bibinfo{year}{1989}), \bibinfo{pages}{684--691}.
\newblock


\bibitem[Glasserman(1971)]%
        {1971and}
\bibfield{author}{\bibinfo{person}{P. Glasserman}.}
  \bibinfo{year}{1971}\natexlab{}.
\newblock \bibinfo{booktitle}{\emph{The Statistical Analysis of Time Series}}.
\newblock \bibinfo{publisher}{John Wiley and Sons, Inc.}, \bibinfo{address}{New
  York, NY}.
\newblock


\bibitem[Glasserman(2003)]%
        {2003gla}
\bibfield{author}{\bibinfo{person}{P. Glasserman}.}
  \bibinfo{year}{2003}\natexlab{}.
\newblock \bibinfo{booktitle}{\emph{Monte Carlo Methods in Financial
  Engineering}}.
\newblock \bibinfo{publisher}{Springer}, \bibinfo{address}{New York, NY}.
\newblock


\bibitem[Glynn and Iglehart(1985)]%
        {1985glyigl}
\bibfield{author}{\bibinfo{person}{Peter Glynn} {and} \bibinfo{person}{Donald
  Iglehart}.} \bibinfo{year}{1985}\natexlab{}.
\newblock \showarticletitle{Large-sample theory for standardized time series:
  an overview}. In \bibinfo{booktitle}{\emph{Proceedings of the 17th conference
  on winter simulation}} \emph{(\bibinfo{series}{WSC '85})}.
  \bibinfo{publisher}{ACM}, \bibinfo{pages}{129--134}.
\newblock
\showISBNx{9780911801071}


\bibitem[Glynn(1996)]%
        {1996gly}
\bibfield{author}{\bibinfo{person}{Peter~W Glynn}.}
  \bibinfo{year}{1996}\natexlab{}.
\newblock \showarticletitle{Importance sampling for Monte Carlo estimation of
  quantiles}. In \bibinfo{booktitle}{\emph{Mathematical Methods in Stochastic
  Simulation and Experimental Design: Proceedings of the 2nd St. Petersburg
  Workshop on Simulation}}. Citeseer, \bibinfo{pages}{180--185}.
\newblock


\bibitem[Glynn(1998)]%
        {1998gly}
\bibfield{author}{\bibinfo{person}{P.~W. Glynn}.}
  \bibinfo{year}{1998}\natexlab{}.
\newblock \showarticletitle{Strong approximations in queueing theory}.
\newblock In \bibinfo{booktitle}{\emph{Asymptotic methods in probability and
  statistics}}. \bibinfo{publisher}{Elsevier}, \bibinfo{pages}{135--150}.
\newblock


\bibitem[Glynn and Iglehart(1988)]%
        {1988glyigl}
\bibfield{author}{\bibinfo{person}{P.~W. Glynn} {and} \bibinfo{person}{D.~L.
  Iglehart}.} \bibinfo{year}{1988}\natexlab{}.
\newblock \showarticletitle{A New Class of Strongly Consistent Variance
  Estimators Simulations}.
\newblock \bibinfo{journal}{\emph{Stochastic Processes and Their Applications}}
   \bibinfo{volume}{28} (\bibinfo{year}{1988}), \bibinfo{pages}{71--80}.
\newblock


\bibitem[Glynn and Iglehart(1990)]%
        {1990glyigl}
\bibfield{author}{\bibinfo{person}{Peter~W Glynn} {and}
  \bibinfo{person}{Donald~L Iglehart}.} \bibinfo{year}{1990}\natexlab{}.
\newblock \showarticletitle{Simulation Output Analysis Using Standardized Time
  Series}.
\newblock  \bibinfo{volume}{15}, \bibinfo{number}{1} (\bibinfo{year}{1990}),
  \bibinfo{pages}{1--16}.
\newblock
\showISSN{0364-765X}


\bibitem[Glynn and Whitt(1992)]%
        {1992glywhi}
\bibfield{author}{\bibinfo{person}{P.~W. Glynn} {and} \bibinfo{person}{W.
  Whitt}.} \bibinfo{year}{1992}\natexlab{}.
\newblock \showarticletitle{The asymptotic validity of sequential stopping
  rules for stochastic simulations}.
\newblock \bibinfo{journal}{\emph{The Annals of Applied Probability}}
  \bibinfo{volume}{2}, \bibinfo{number}{1} (\bibinfo{year}{1992}),
  \bibinfo{pages}{180--197}.
\newblock


\bibitem[Goldsman et~al\mbox{.}(1990)]%
        {1990golmeksch}
\bibfield{author}{\bibinfo{person}{D. Goldsman}, \bibinfo{person}{M. Meketon},
  {and} \bibinfo{person}{L.~W. Schruben}.} \bibinfo{year}{1990}\natexlab{}.
\newblock \showarticletitle{Properties of standardized time series weighted
  area variance estimators}.
\newblock \bibinfo{journal}{\emph{Management Science}} \bibinfo{volume}{36},
  \bibinfo{number}{5} (\bibinfo{year}{1990}), \bibinfo{pages}{602--612}.
\newblock


\bibitem[Goldsman and Schruben(1990)]%
        {1990golsch}
\bibfield{author}{\bibinfo{person}{David Goldsman} {and} \bibinfo{person}{Lee
  Schruben}.} \bibinfo{year}{1990}\natexlab{}.
\newblock \showarticletitle{Note--New Confidence Interval Estimators Using
  Standardized Time Series}.
\newblock  \bibinfo{volume}{36}, \bibinfo{number}{3} (\bibinfo{year}{1990}),
  \bibinfo{pages}{393--397}.
\newblock
\showISSN{0025-1909}


\bibitem[Grabaskas et~al\mbox{.}(2016)]%
        {2016graetal}
\bibfield{author}{\bibinfo{person}{D. Grabaskas}, \bibinfo{person}{M.~K.
  Nakayama}, \bibinfo{person}{R. Denning}, {and} \bibinfo{person}{T. Aldemir}.}
  \bibinfo{year}{2016}\natexlab{}.
\newblock \showarticletitle{Advantages of variance reduction techniques in
  establishing confidence intervals for quantiles}.
\newblock \bibinfo{journal}{\emph{Reliability Engineering \& System Safety}}
  \bibinfo{volume}{149} (\bibinfo{year}{2016}), \bibinfo{pages}{187--203}.
\newblock


\bibitem[Hall(1992)]%
        {1992hall}
\bibfield{author}{\bibinfo{person}{Peter Hall}.}
  \bibinfo{year}{1992}\natexlab{}.
\newblock \bibinfo{booktitle}{\emph{Principles of Edgeworth Expansion}}.
\newblock \bibinfo{publisher}{Springer New York}, \bibinfo{address}{New York,
  NY}, \bibinfo{pages}{39--81}.
\newblock
\showISBNx{978-1-4612-4384-7}
\urldef\tempurl%
\url{https://doi.org/10.1007/978-1-4612-4384-7_2}
\showDOI{\tempurl}


\bibitem[Hartigan(1969)]%
        {1969har}
\bibfield{author}{\bibinfo{person}{J.~A Hartigan}.}
  \bibinfo{year}{1969}\natexlab{}.
\newblock \showarticletitle{Using Subsample Values as Typical Values}.
\newblock \bibinfo{journal}{\emph{J. Amer. Statist. Assoc.}}
  \bibinfo{volume}{64}, \bibinfo{number}{328} (\bibinfo{year}{1969}),
  \bibinfo{pages}{1303--1317}.
\newblock
\showISSN{0162-1459}


\bibitem[Hartigan(1975)]%
        {1975har}
\bibfield{author}{\bibinfo{person}{J.~A. Hartigan}.}
  \bibinfo{year}{1975}\natexlab{}.
\newblock \showarticletitle{Necessary and Sufficient Conditions for Asymptotic
  Joint Normality of a Statistic and Its Subsample Values}.
\newblock \bibinfo{journal}{\emph{The Annals of statistics}}
  \bibinfo{volume}{3}, \bibinfo{number}{3} (\bibinfo{year}{1975}),
  \bibinfo{pages}{573--580}.
\newblock
\showISSN{0090-5364}


\bibitem[Iglehart(1978)]%
        {1978igl}
\bibfield{author}{\bibinfo{person}{D.~L. Iglehart}.}
  \bibinfo{year}{1978}\natexlab{}.
\newblock \showarticletitle{The regenerative method for simulation analysis.}
\newblock In \bibinfo{booktitle}{\emph{Handbook of Optimization in Medicine}},
  \bibfield{editor}{\bibinfo{person}{K.M. Chandy} {and} \bibinfo{person}{R.T.
  Yeh}} (Eds.). \bibinfo{publisher}{Prentice-Hall, Englewood Cliffs, N.J.}
\newblock


\bibitem[Lam(2022)]%
        {2022lam}
\bibfield{author}{\bibinfo{person}{Henry Lam}.}
  \bibinfo{year}{2022}\natexlab{}.
\newblock \bibinfo{title}{A Cheap Bootstrap Method for Fast Inference}.
\newblock
\newblock
\urldef\tempurl%
\url{https://doi.org/10.48550/ARXIV.2202.00090}
\showDOI{\tempurl}


\bibitem[Lehmann(1999)]%
        {1999leh}
\bibfield{author}{\bibinfo{person}{E.~L. (Erich~Leo) Lehmann}.}
  \bibinfo{year}{1999}\natexlab{}.
\newblock \bibinfo{booktitle}{\emph{Elements of large-sample theory}}.
\newblock \bibinfo{publisher}{Springer}, \bibinfo{address}{New York}.
\newblock
\showISBNx{0387985956}
\showLCCN{98034429}


\bibitem[Loh(1987)]%
        {1987loh}
\bibfield{author}{\bibinfo{person}{Wei-Yin Loh}.}
  \bibinfo{year}{1987}\natexlab{}.
\newblock \showarticletitle{Calibrating confidence coefficients}.
\newblock \bibinfo{journal}{\emph{J. Amer. Statist. Assoc.}}
  \bibinfo{volume}{82}, \bibinfo{number}{397} (\bibinfo{year}{1987}),
  \bibinfo{pages}{155--162}.
\newblock


\bibitem[Mahalanobis(1946)]%
        {1946mah}
\bibfield{author}{\bibinfo{person}{P.~C. Mahalanobis}.}
  \bibinfo{year}{1946}\natexlab{}.
\newblock \showarticletitle{Sample Surveys of Crop Yields in India}.
\newblock \bibinfo{journal}{\emph{Sankhyā the Indian journal of statistics}}
  \bibinfo{volume}{7}, \bibinfo{number}{3} (\bibinfo{year}{1946}),
  \bibinfo{pages}{269--280}.
\newblock
\showISSN{0036-4452}


\bibitem[McCarthy(1969)]%
        {1969mcc}
\bibfield{author}{\bibinfo{person}{P.~J McCarthy}.}
  \bibinfo{year}{1969}\natexlab{}.
\newblock \showarticletitle{Pseudo-Replication: Half Samples}.
\newblock \bibinfo{journal}{\emph{Revue de l'Institut international de
  statistique}} \bibinfo{volume}{37}, \bibinfo{number}{3}
  (\bibinfo{year}{1969}), \bibinfo{pages}{239}.
\newblock
\showISSN{0373-1138}


\bibitem[Mechanic and {McKay}(1966)]%
        {1966mecmck}
\bibfield{author}{\bibinfo{person}{H. Mechanic} {and} \bibinfo{person}{W.
  {McKay}}.} \bibinfo{year}{1966}\natexlab{}.
\newblock \bibinfo{booktitle}{\emph{Confidence intervals for averages of
  dependent data in simulations II}}.
\newblock \bibinfo{type}{{T}echnical {R}eport} 17-202.
\newblock


\bibitem[Mu\~{n}oz(1991)]%
        {1991mun}
\bibfield{author}{\bibinfo{person}{David~Fernando Mu\~{n}oz}.}
  \bibinfo{year}{1991}\natexlab{}.
\newblock \emph{\bibinfo{title}{Cancellation methods in the analysis of
  simulation output}}.
\newblock \bibinfo{thesistype}{Ph.\,D. Dissertation}. \bibinfo{school}{Stanford
  University}.
\newblock


\bibitem[Nakayama(2011)]%
        {2011nak}
\bibfield{author}{\bibinfo{person}{M.~K. Nakayama}.}
  \bibinfo{year}{2011}\natexlab{}.
\newblock \showarticletitle{Asymptotically valid confidence intervals for
  quantiles and values-at-risk when applying Latin hypercube sampling}.
\newblock \bibinfo{journal}{\emph{International Journal on Advances in Systems
  and Measurements}}  \bibinfo{volume}{4} (\bibinfo{year}{2011}).
\newblock


\bibitem[Nakayama(2014)]%
        {2014nak}
\bibfield{author}{\bibinfo{person}{Marvin~K Nakayama}.}
  \bibinfo{year}{2014}\natexlab{}.
\newblock \showarticletitle{Confidence intervals for quantiles using sectioning
  when applying variance-reduction techniques}.
\newblock \bibinfo{journal}{\emph{ACM Transactions on Modeling and Computer
  Simulation (TOMACS)}} \bibinfo{volume}{24}, \bibinfo{number}{4}
  (\bibinfo{year}{2014}), \bibinfo{pages}{1--21}.
\newblock


\bibitem[Pasupathy(2011a)]%
        {2011pasa}
\bibfield{author}{\bibinfo{person}{Raghu Pasupathy}.}
  \bibinfo{year}{2011}\natexlab{a}.
\newblock \showarticletitle{Generating homogeneous Poisson processes}.
\newblock \bibinfo{journal}{\emph{Wiley encyclopedia of operations research and
  management science}} (\bibinfo{year}{2011}).
\newblock


\bibitem[Pasupathy(2011b)]%
        {2011pasb}
\bibfield{author}{\bibinfo{person}{Raghu Pasupathy}.}
  \bibinfo{year}{2011}\natexlab{b}.
\newblock \showarticletitle{Generating nonhomogeneous Poisson processes}.
\newblock \bibinfo{journal}{\emph{Wiley encyclopedia of operations research and
  management science}} (\bibinfo{year}{2011}).
\newblock


\bibitem[Philipp and Stout(1975)]%
        {1975phisto}
\bibfield{author}{\bibinfo{person}{W. Philipp} {and} \bibinfo{person}{W.
  Stout}.} \bibinfo{year}{1975}\natexlab{}.
\newblock \showarticletitle{Almost Sure Invariance Principles for Partial Sums
  of Weakly Dependent Random Variables}.
\newblock \bibinfo{journal}{\emph{Mem. Amer. Math. Soc}}  \bibinfo{volume}{161}
  (\bibinfo{year}{1975}).
\newblock


\bibitem[Politis and Romano(1994)]%
        {1992polrom}
\bibfield{author}{\bibinfo{person}{D.~N. Politis} {and} \bibinfo{person}{J.~P.
  Romano}.} \bibinfo{year}{1994}\natexlab{}.
\newblock \showarticletitle{Large Sample Confidence Regions Based on Subsamples
  under Minimal Assumptions}.
\newblock \bibinfo{journal}{\emph{The Annals of statistics}}
  \bibinfo{volume}{22}, \bibinfo{number}{4} (\bibinfo{year}{1994}),
  \bibinfo{pages}{2031--2050}.
\newblock
\showISSN{0090-5364}


\bibitem[Politis et~al\mbox{.}(1999)]%
        {1999polromwol}
\bibfield{author}{\bibinfo{person}{D.~N. Politis}, \bibinfo{person}{J.~P.
  Romano}, {and} \bibinfo{person}{M. Wolf}.} \bibinfo{year}{1999}\natexlab{}.
\newblock \bibinfo{booktitle}{\emph{Subsampling} (\bibinfo{edition}{1st ed.
  1999.} ed.)}.
\newblock
\showISBNx{0-387-98854-8}


\bibitem[Quenouille(1949)]%
        {1949que}
\bibfield{author}{\bibinfo{person}{M.~H. Quenouille}.}
  \bibinfo{year}{1949}\natexlab{}.
\newblock \showarticletitle{Approximate Tests of Correlation in Time-Series}.
\newblock \bibinfo{journal}{\emph{Journal of the Royal Statistical Society.
  Series B, Methodological}} \bibinfo{volume}{11}, \bibinfo{number}{1}
  (\bibinfo{year}{1949}), \bibinfo{pages}{68--84}.
\newblock
\showISSN{0035-9246}


\bibitem[Sarykalin et~al\mbox{.}(2008)]%
        {2008sarserury}
\bibfield{author}{\bibinfo{person}{S. Sarykalin}, \bibinfo{person}{G.
  Serraino}, {and} \bibinfo{person}{S. Uryasev}.}
  \bibinfo{year}{2008}\natexlab{}.
\newblock \showarticletitle{Value-at-risk vs. conditional value-at-risk in risk
  management and optimization}.
\newblock In \bibinfo{booktitle}{\emph{State-of-the-art decision-making tools
  in the information-intensive age}}. \bibinfo{publisher}{Informs},
  \bibinfo{pages}{270--294}.
\newblock


\bibitem[Schruben(1983)]%
        {1983sch}
\bibfield{author}{\bibinfo{person}{L.~W. Schruben}.}
  \bibinfo{year}{1983}\natexlab{}.
\newblock \showarticletitle{Confidence Interval Estimation Using Standardized
  Time Series}.
\newblock \bibinfo{journal}{\emph{Operations Research}} \bibinfo{volume}{31},
  \bibinfo{number}{6} (\bibinfo{year}{1983}), \bibinfo{pages}{1090--1108}.
\newblock


\bibitem[Serfling(1980)]%
        {1980ser}
\bibfield{author}{\bibinfo{person}{R.~J. Serfling}.}
  \bibinfo{year}{1980}\natexlab{}.
\newblock \bibinfo{booktitle}{\emph{Approximation Theorems of Mathematical
  Statistics}}.
\newblock \bibinfo{publisher}{John Wiley \& Sons, Inc.}, \bibinfo{address}{New
  York, New York}.
\newblock


\bibitem[{van de Geer}(2006)]%
        {2006gee}
\bibfield{author}{\bibinfo{person}{S.~A. {van de Geer}}.}
  \bibinfo{year}{2006}\natexlab{}.
\newblock \bibinfo{booktitle}{\emph{Empirical Processes in M-Estimation}
  (\bibinfo{edition}{1-st} ed.)}.
\newblock \bibinfo{publisher}{Cambridge University Press},
  \bibinfo{address}{Cambridge}.
\newblock
\showISBNx{0521123259}


\bibitem[Welch(1967)]%
        {1967wel}
\bibfield{author}{\bibinfo{person}{P.D. Welch}.}
  \bibinfo{year}{1967}\natexlab{}.
\newblock \showarticletitle{The use of the {F}ast {F}ourier {T}ransform for the
  estimation of spectra; a method based on time averaging over short modified
  periodograms}.
\newblock \bibinfo{journal}{\emph{IEEE Transactions on Audio and
  Electroacoustics}} \bibinfo{number}{2} (\bibinfo{year}{1967}),
  \bibinfo{pages}{70--73}.
\newblock


\bibitem[Wellner(2022)]%
        {2013wel}
\bibfield{author}{\bibinfo{person}{Jon Wellner}.}
  \bibinfo{year}{2022}\natexlab{}.
\newblock \bibinfo{title}{Jon Wellner's Lecture Notes in Mathematical
  Statistics}.
\newblock
  \bibinfo{howpublished}{\url{https://sites.stat.washington.edu/peter/581/jaw/jaw.html}}.
\newblock
\newblock
\shownote{[Online; accessed 02-March-2022]}.


\end{thebibliography}

\appendix
\rp{\section{Subsampling and Bootstrapping}\label{sec:nonCLT}

In this section, we provide a concise overview of subsampling and bootstrapping. To maintain a clear connection with the topic of this paper, we focus the discussion on contexts that use a Studentized statistic.

\subsection{Subsampling}\label{sec:subsampling} 

\emph{Subsampling} and \emph{the bootstrap} are examples of methods that are not CLT-based methods in the sense that they do not assume knowledge of the normal weak limit in~\eqref{firstclt}, although they assume the existence of a weak limit. Subsampling is the culmination of decades of thought on using batches for confidence intervals, and was formalized in a 1992 paper by Politis and Romano~\cite{1992polrom}. See~\citep{1999polromwol} for a book-length treatment that includes situations where $\theta(P)$ resides in a separable Banach space.  

\begin{remark} There is a long history of using batches within the classical statistics literature in the context of constructing a confidence interval from time series data, e.g., interpenetration samples by Mahalanobis~\cite{1946mah}, the jacknife by Quenouille~\cite{1949que}, pseudoreplication by McCarthy~\cite{1969mcc}, and subsampling by Hartigan~\cite{1969har, 1975har}. There is also a corresponding history in the simulation literature dating back to Conway~\cite{1963con}, Mechanic and McKay~\cite{1966mecmck}, and Fishman~\cite{1979fis} --- precursors to the now mature methods to construct confidence intervals on the steady-state mean using batched simulation output.\end{remark} 

In the interest of easily conveying the essence of subsampling and the bootstrap, the ensuing discussion assumes use of the Studentized statistic $\hat{\sigma}_n^{-1}(\hat{\theta}_n - \theta(P))$, where $\hat{\theta}_n$ is the point estimator of $\theta(P)$, and $\hat{\sigma}^2_n$ is a point estimator of the variance parameter $\sigma^2$. Suppose \begin{align} \tau_n\frac{\left(\hat{\theta}_n - \theta(P)\right)}{\hat{\sigma}_n} \quad \inD \quad  J(P);  \label{subsamplingstat1}\end{align} and define \begin{align} \label{subsamplingstat2} L_n(x) &:= \frac{1}{n-m_n+1} \sum_{j=1}^{n - m_n + 1} \mathbb{I}\left\{\tau_{m_n}\frac{(\hat{\theta}_{j,m_n} - \hat{\theta}_n)}{\hat{\sigma}_{j,m_n}} \leq x\right\}, \quad x \in \mathbb{R}  \nonumber \\ &= \frac{1}{n-m_n+1}\sum_{j=1}^{n-m_n+1} \mathbb{I}\left\{\frac{\tau_{m_n}(\hat{\theta}_{j,m_n} - \theta(P))}{\hat{\sigma}_{j,m_n}} + \frac{\tau_{m_n}(\theta(P) - \hat{\theta}_n)}{\hat{\sigma}_{j,m_n}} \leq x\right\},\end{align} where $\hat{\theta}_n$ is a point estimator of $\theta(P)$ constructed from the entire data set $\left(X_1, X_2, \ldots, X_n\right)$,  $\theta_{j,m_n}$ is the estimator of $\theta(P)$ constructed from the $j$-th subsample $\left(X_{(j-1)m_n + 1}, X_{(j-1)m_n + 2}, \ldots, X_{(j-1)m_n + m_n}\right)$ (see Figure~\ref{batching} with offset $d_n=1$), $\{\tau_n, n \geq 1\}$ is a ``scaling" sequence, and $\hat{\sigma}_n$ is an estimate of what is called the \emph{scale} $\sigma$ in~\cite{1992polrom}, and what we call the \emph{variance constant} in this paper. Also define \begin{equation}\label{def:Un} U_n(x) : = \frac{1}{n-m_n+1}\sum_{j=1}^{n-m_n+1} \mathbb{I}\left\{\frac{\tau_{m_n}(\hat{\theta}_{j,m_n} - \theta(P))}{\hat{\sigma}_{j,m_n}} \leq x\right\}; \quad E_n: =  \left(\frac{\tau_{m_n}(\theta(P) - \hat{\theta}_n)}{\hat{\sigma}_{j,m_n}} \leq x\right).\end{equation} 
Suppose that in addition to the assumption of the weak limit existence in~\eqref{subsamplingstat1}, the following assumptions hold. \begin{enumerate} \item[(A.1)] the sequence $\{X_n, n \geq 1\}$ is stationary and strong-mixing (defined in Section~\ref{sec:mathprelim}); \item[(A.2)] $\tau_{m_n}/\tau_n \to 0$, $m_n/n \to 0$; and \item[(A.3)] the cdf $J_n(\cdot,P)$ of the Studentized statistic $\tau_n\left(\hat{\theta}_n - \theta(P)\right)/\hat{\sigma}_n$ is continuous (in its first argument).\end{enumerate} 

The crucial insight of subsampling is that the empirical cdf $L_n$ can be used to approximate the sampling distribution of the Studentized statistic by replacing $\hat{\theta}_n$ and $\hat{\sigma}_{n}$ in~\eqref{subsamplingstat1} by their subsample counterparts $\hat{\theta}_{j,m_n}, \hat{\sigma}_{j,m_n}$, and by replacing $\theta(P)$ in~\eqref{subsamplingstat1} by $\hat{\theta}_n$. In~\cite{1992polrom}, this is formalized by demonstrating the intuitive result \begin{equation}\label{subsamplingmainresult}  \sup_{x \in \mathbb{R}} |L_n(x) - J_n(x,P)| \inP 0.\end{equation} 

The assertion in~\eqref{subsamplingmainresult} motivates constructing the following (two-sided) subsampling confidence interval $I_{n,\alpha}$ on $\theta(P)$, that can be shown to be asymptotically valid: \begin{equation}\label{subsamplingconf} I_{n,\alpha} : = (\hat{\theta}_n - c_{n,\alpha/2}\frac{\hat{\sigma}_n}{\tau_n}, \hat{\theta}_n + c_{n,1-\alpha/2}\frac{\hat{\sigma}_n}{\tau_n}); \quad c_{n,q} := \inf\{x: L_n(x) \geq q\}.\end{equation}

\begin{remark} Notice from~\eqref{subsamplingconf} that subsampling assumes knowledge of the scaling $\tau_n$ but not the weak limit $J$. This will be true about the bootstrap as well.  \end{remark}

We emphasize that main instrument in~\cite{1992polrom} for establishing that $L_n(x) \inP J(x,P)$ is the ``small batch size" assumption in (A.2) above, which ensures that the probability of the event $E_n$ in~\eqref{def:Un} tends to one, and the variance of $U_n(x)$ in~\eqref{def:Un} tends to zero.

\subsection{The Bootstrap}
In the service of precisely explaining the bootstrap~\cite{1979efr,1992hall,1998efrtib,1997davhin}, let's enhance the notation introduced previously to view point estimators as functionals, that is, $\hat{\theta}: \mathcal{D}_n \to \mathbb{R}$ and $\hat{\sigma}^2: \mathcal{D}_n \to \mathbb{R}$, where $\mathcal{D}_n$ is the space of datasets of size $n$ with $S$-valued observations. So, the point estimator $\hat{\theta}_n$ of the statistical functional $\theta(P)$ and the point estimator $\hat{\sigma}^2_n$ of the variance constant $\sigma^2$ constructed using the given dataset $X_1, X_2, \ldots, X_n$ are $$\hat{\theta}_n \equiv \hat{\theta}(\{X_1,X_2, \ldots,X_n\}); \quad \hat{\sigma}^2_n \equiv \hat{\sigma}^2_n(\{X_1,X_2, \ldots,X_n\}).$$ 

The bootstrap's central idea is a method for approximating the sampling distribution of the Studentized statistic $\sqrt{n}\left(\hat{\theta}_n - \theta(P)\right)/\hat{\sigma}_n$. (The bootstrap assumes that the weak limit in~\eqref{subsamplingstat1} holds with $\tau_n = \sqrt{n}$, although the assumption of the existence of a weak limit is often not stated explicitly.) And, whereas subsampling constructs the empirical cdf $L_n$ in~\eqref{subsamplingstat2} using subsamples, the bootstrap accomplishes the objective of estimating the sampling distribution of $\sqrt{n}\left(\hat{\theta}_n - \theta(P)\right)/\hat{\sigma}_n$ through the following two steps: \begin{enumerate} \item \emph{resample}, that is, use a ``resampling measure'' $\hat{P}_n$ to generate $B$ datasets $\{X_{j,i}^*, 1 \leq j \leq n\}, i=1,2,\ldots,B$, e.g., iid draws with replacement from the original dataset; and \item \emph{compute}, that is, use the generated datasets $\{X_{j,i}^*, 1 \leq j \leq n\}, i=1,2,\ldots,B$ to compute ``bootstrap realizations'' $\hat{\theta}(\{X_{j,i}^*, 1 \leq j \leq n\}), i=1,2,\ldots,B$ of the point estimator and ``bootstrap realizations'' $$S_{n,i} := \sqrt{n}\frac{\left(\hat{\theta}(\{X_{j,i}^*, 1 \leq j \leq n\}) - \hat{\theta}_n\right)}{\hat{\sigma}_n}, i = 1,2,\ldots,B$$ of the Studentized statistic.\end{enumerate} (It is important that the above steps are for the context of bootstrapping with iid data; in the context of a time series, a modification such as the \emph{moving blocks bootstrap}~\cite[pp. 101]{1998efrtib} is needed.)

The bootstrap then uses observations $S_{n,i}, i =1,2,\ldots,B$ to compute the empirical cdf $\hat{J}_n(\cdot,\hat{P}_n)$ used to approximate the sampling distribution function $J_n(\cdot,P)$ of $\sqrt{n}\left(\hat{\theta}_n - \theta(P)\right)/\hat{\sigma}_n$, yielding the following two-sided $(1-\alpha)$ confidence interval on $\theta(P)$:
\begin{equation}\label{bootstrapconf} (\hat{\theta}_n - \hat{J}_n^{-1}(\alpha/2), \hat{\theta}_n + \hat{J}_n^{-1}(1-\alpha/2)).\end{equation} 

As can be observed in the statement and proof of the bootstrap's main theorem~\cite[Theorem 1.2.1]{1999polromwol}, consistency follows upon assuming that the resampling and compute steps above are such that (i) the resulting approximation $\hat{J}_n$ in a sense consistently approximates $J_n(\cdot,P)$, e.g., $\rho_L(J_n(\cdot,P),\hat{J}_n(\cdot,\hat{P}_n)) \as 0$ as $n \to \infty$, where $\rho_L$ is the L\'{e}vy metric~\cite[Section 15.1]{das2011}; and (ii) the distribution $J(\cdot,P)$ of the weak limit $J(P)$ in~\eqref{subsamplingstat1} is continuous and strictly increasing at $\inf\{x: J(x,P) \geq 1-\alpha\}$.

Since its original introduction in 1979~\cite{1992efr}, the bootstrap has received tremendous attention due to its simplicity and wide applicability, resulting in popular refinements~\cite{1981efr,1982efr,1985efr,1987efr}, the ability to handle time series~\cite{1998efrtib,1990bos}, extensions to the functional context~\cite{1988dicrom},  higher-order corrections\cite{1987loh,1992hall} to improve coverage accuracy, and most recently a computationally ``cheap'' version~\cite{2022lam}. Debates on whether subsampling or the bootstrap is better have continued, but it is now known that subsampling is more general in that the bootstrap requires the behavior of the bootstrap distribution $\hat{J}_n(\cdot,\hat{P}_n)$ to be smooth (around $P$) when seen as a function of its second argument. We go into no further detail on this point but see~\cite[Section 2.3]{1999polromwol}. Also see~\cite{1989ginzin} for an interesting theorem on the sense in which the bootstrap in its basic form is not valid if the variance parameter does not exist.}

\section{Some Useful Results}\label{sec:useful}
We will invoke the following useful result from~\cite{2010dur} that provides a weak law for triangular arrays of real-valued random variables that are not necessarily identically distributed. 

\begin{theorem}[Weak law for triangular arrays, Theorem 2.2.6,~\cite{2010dur}]\label{thm:weaklawtarrays} For each $n$, let $X_{n,k}, 1 \leq k \leq n$ be independent. Let $t_n > 0$ with $t_n \to \infty$ and let $\bar{X}_{n,k} = X_{n,k}\mathbbm{1}(|X_{n,k}| \leq t_n)$. Suppose that as $n \to \infty$, \begin{enumerate}\item $\sum_{k=1}^n P(|X_{n,k}| > t_n) \to 0$; and \item $t_n^{-2}\sum_{k=1}^n \mathbb{E}[\bar{X}_{n,k}^2] \to 0$.\end{enumerate} If we let $S_n = X_{n,1} + X_{n,2} + \cdots + X_{n,n}$ and put $\mu_n = \sum_{k=1}^n \mathbb{E}[\bar{X}_{n,k}]$, then $t_n^{-1} (S_n - \mu_n) \inP 0$.
\end{theorem}

\begin{theorem}[Slutsky's Theorem, see page 19 in~\cite{1980ser}]\label{thm:slutsky} Suppose $\{A_n, n \geq 1\}$, $\{B_n, n \geq 1\}$ and $\{X_n, n \geq 1\}$ are real-valued random sequences so that $$X_n \inD X; \quad A_n \inP A; \quad B_n \inP B.$$ Then $$A_nX_n +B_n \inD AX +B.$$ If $A \neq 0,$ then $$\frac{X_n}{A_n} \inD \frac{X}{A}.$$
\end{theorem}

\begin{theorem} [Covariance Bound, see Corollary 2.5, \citet{2009ethkur}]\label{thm:covbd} Let $1 \leq u,v,w \leq \infty, u^{-1} + v^{-1}+ w^{-1} =1.$ Then for real-valued $Y,Z$ with $Y \in L^w(\Omega,\mathcal{G},P), Z \in L^v(\Omega,\mathcal{H},P)$, $$\left |\mathbb{E}\left[ YZ ] - \mathbb{E}[Y]\mathbb{E}[Z]\right] \right | \leq 2^{v \wedge w \wedge 2 + 1} \alpha^{1/u}(\mathcal{G},\mathcal{H})\|Y\|_v \|Z\|_w.$$ \end{theorem}

\begin{theorem}[see Theorem 1.2.1,~\cite{1981csorev}]\label{thm:brownianinv} Let $\{W(t), 0 \leq t < \infty\}$ denote the Wiener process. If $\{a_n, n\geq 1\}$ is a monotonically non-decreasing sequence of $n$ such that $0 < a_n \leq n$ and the sequence $\{n/a_n, n \geq 1\}$ is monotonically non-decreasing, then,  $$\limsup_{n \to \infty} \sup_{0 \leq t \leq n-a_n} \sup_{0 \leq s \leq a_n} \beta_n \left|B(t+s) - B(t)\right| = 1 \emph{ a.s.}$$ where $$\beta_n = \left(2a_n\left( \log \frac{n}{a_n} + \log^2 n\right)\right)^{-\frac{1}{2}}.$$

\end{theorem}

\begin{theorem}[see Mapping Theorem 2.7,~\cite{1999bil}]\label{thm:cmt}
Suppose $h$ is an $\mathcal{S}/\mathcal{S}'$-measurable mapping from $S$ to $S'$ with discontinuity set $D_h \subset S$, where $(S,\mathcal{S})$ and $(S',\mathcal{S}')$ are metric spaces. If $\{Q_n, n \geq 1\}$ is a sequence of probability measures on $(S, \mathcal{S})$ with weak limit $Q$, that is, $Q_n \inD Q$, and $QD_h=0$, then $Q_nh^{-1} \inD Qh^{-1}.$
\end{theorem}

\section{Proof of Theorem~\ref{thm:ob2largebatch}}\label{app:ob2largebatch}
\begin{proof} Similar to the proof of Theorem~\ref{thm:ob1largebatch}, observe that \begin{align}\label{firstsplit-samplemean}  \hat{\sigma}^2_{\mbox{\tiny OB-II}}(m_n,b_n) &= \frac{1}{\kappa_2(\beta,b_{\infty})}\underbrace{\frac{m_n}{b_n}\sum_{j=1}^{b_n} \left[\left(\bar{\theta}_{j,m_n} - \bar{\theta}_{m_n} \right)^2 - \sigma^2\left(\tilde{B}_{j,m_n} - \frac{1}{b_n} \sum_{i=1}^{b_n} \tilde{B}_{i,m_n} \right)^2\right]}_{\bar{E_n}(m_n,b_n)} \nonumber \\ & \hspace{2.5in} + \frac{1}{\kappa_2(\beta,b_{\infty})} \underbrace{\frac{\sigma^2}{b_n}\sum_{j=1}^{b_n}\left(\sqrt{m_n} \tilde{B}_{j,m_n} - \frac{\sqrt{m_n}}{b_n} \sum_{i=1}^{b_n} \tilde{B}_{i,m_n} \right)^2}_{\bar{I}_n}  \nonumber \\ &= \frac{1}{\kappa_2(\beta,b_{\infty})}\left(\bar{E}_n(m_n,b_n) + \bar{I}_n\right).\end{align} 

We will now individually characterize the behavior of $\bar{E}_n(m_n,b_n)$ and $\bar{I}_n$ above.
 
Noticing that \begin{equation}\label{error-samplemean} \bar{\theta}_{j,m_n} - \bar{\theta}_{m_n} = \underbrace{\left(\bar{\theta}_{j,m_n} - \sigma\tilde{B}_{j,m_n}\right)}_{U_{j,m_n}} + \sigma\underbrace{\left(\tilde{B}_{j,m_n} - \frac{1}{b_n} \sum_{i=1}^{b_n} \tilde{B}_{i,m_n} \right)}_{\bar{H}_{j,m_n}} + \underbrace{\left(\frac{1}{b_n} \sum_{i=1}^{b_n} \tilde{B}_{i,m_n} - \bar{\theta}_{m_n,d_n} \right)}_{\bar{C}_{m_n}},\end{equation} we can write \begin{align}\label{split2-samplemean} \bar{E_n}(m_n,b_n) = m_n\left(\frac{1}{b_n}\sum_{j=1}^{b_n} U_{j,m_n}^2 + 2 \frac{\sigma}{b_n}\sum_{j=1}^{b_n}  U_{j,m_n} \bar{H}_{j,m_n} + 2 \frac{\sigma \bar{C}_{m_n}}{b_n}\sum_{j=1}^{b_n} U_{j,m_n} +  2 \frac{\sigma \bar{C}_{m_n}}{b_n}\sum_{j=1}^{b_n} \bar{H}_{j,m_n} + \bar{C}_{m_n}^2\right).\end{align} We already know that except for a set of measure zero in the probability space implied by Assumption~\ref{ass:stronginvar}, there exists $\Gamma(\omega)$ such that, uniformly in $j$,  \begin{align}\label{errbds1-samplemean} |U_{j,m_n}| & \leq \Gamma(\omega) m_n^{-1/2-\delta}\left(\log^2 m_n\right)^{1/2}, \end{align} and similarly, \begin{align}\label{errbds2-samplemean} |\bar{C}_{m_n}| & = \frac{\sigma}{b_n}  \left | \sum_{i=1}^{b_n} \sigma^{-1} \bar{\theta}_{i,m_n} - \tilde{B}_{i,m_n} \right |  \nonumber \\  & \leq \Gamma(\omega) m_n^{-1/2-\delta}\left(\log^2 m_n \right)^{1/2} \end{align}

Furthermore, similar to~\eqref{browniannuggetbds}, we use Theorem~\ref{thm:brownianinv} carefully again to see that there exists $n_0(\omega,\epsilon)$ such that for all $n \geq n_0(\omega,\epsilon)$, and uniformly in $j$, \begin{align}\label{browniannuggetbds-samplemean} |\bar{H}_{j,m_n}|  \leq 2 (1+\epsilon) m_n^{-1/2} \sqrt{2 \left(\log^2 n - \log \frac{m_n}{n} \right)} 
\end{align}

Plugging~\eqref{errbds1-samplemean}, ~\eqref{errbds2-samplemean}, and~\eqref{browniannuggetbds-samplemean} in~\eqref{split2-samplemean}, we get for all $n \geq n_0(\omega,\epsilon)$ that \begin{align}\label{errfinbd-samplemean} \bar{E}_n(m_n,b_n) &\leq \Gamma^2(\omega) m_n^{-2\delta}\log^2 m_n
\nonumber \\ & \hspace{0.5in} + 8 \sigma (1+\epsilon) \Gamma(\omega) m_n^{-\delta}\, \left(\log^2 m_n\right)^{1/2} (2(\log^2 n - \log \frac{m_n}{n}))^{1/2}  \nonumber \\ & \hspace{0.5in}  + 2\sigma \Gamma^2(\omega)m_n^{-2\delta}\log^2 m_n + \Gamma^2(\omega)m_n^{-2\delta}\log^2 m_n,\end{align} implying that $\bar{E}_n$ goes to zero almost surely. Now let's calculate the weak limit of $$\bar{I}_n :=  \frac{\sigma^2}{b_n}\sum_{j=1}^{b_n}\left(\sqrt{m_n} \tilde{B}_{j,m_n} - \frac{1}{b_n} \sum_{i=1}^{b_n} \sqrt{m_n} \tilde{B}_{i,m_n} \right)^2$$ appearing in \eqref{firstsplit-samplemean}. Similar to the proof of Theorem~\ref{thm:ob1largebatch}, define a lattice $\{0, \delta_n, 2 \delta_n, \ldots, \lfloor \frac{1}{\delta_n} \rfloor \delta_n \}$ having resolution $$\delta_n : = \frac{1-(m_n/n)}{b_n-1}$$ and a corresponding projection operation on the lattice $$\lfloor u \rfloor_{\scriptsize{\delta_n}} := \max\{k \delta_n: u \geq k \delta_n, k \in \mathbb{Z}\}, \quad u \in [0, 1-\frac{m_n}{n} + \delta_n].$$ 

Recalling that $b_n = 1 + d_n^{-1}(n - m_n)$, we  can write \begin{align}  \bar{I}_n &= \sigma^2\frac{1}{b_n} \frac{b_n-1}{1-(m_n/n)}  \int_0^{1- \frac{m_n}{n} + \delta_n} \left( {\frac{1}{\sqrt{m_n}} \left( W(n \lfloor  u \rfloor_{\delta_n} + m_n) - W(n \lfloor  u \rfloor_{\delta_n} + 1) \right) 
}\right. \nonumber \\ & \hspace{0.5in} \left.{ -\frac{1}{b_n} \frac{b_n-1}{1-(m_n/n)} \int_0^{1- \frac{m_n}{n} + \delta_n} \frac{1}{\sqrt{m_n}} \left( W(n \lfloor  s \rfloor_{\delta_n} + m_n) - W(n \lfloor  s \rfloor_{\delta_n} + 1) \right) \, ds }\right)^2  \, du  \\ & \overset{d}{=} \sigma^2\frac{b_n-1}{b_n} \frac{n}{n-m_n} \frac{n}{m_n} \int_0^{1- \frac{m_n}{n} + \delta_n}  \left( {W( \lfloor  u \rfloor_{\delta_n} + \frac{m_n}{n}) - W( \lfloor  u \rfloor_{\delta_n} + \frac{1}{n})}\right. \nonumber \\& \hspace{0.5in} \left.{- \frac{b_n-1}{b_n} \frac{n}{n-m_n} \int_0^{1- \frac{m_n}{n} + \delta_n} \frac{1}{\sqrt{m_n}} \left( W(n \lfloor  s \rfloor_{\delta_n} + m_n) - W(n \lfloor  s \rfloor_{\delta_n} + 1) \right) \, ds } \right)^2 \, du  \\ & \inD \sigma^2 \frac{\beta^{-1}}{1- \beta} \int_0^{1-\beta} \left( W(u + \beta) - W(u) - \frac{1}{1-\beta} \int_{0}^{1-\beta} \left(W(
s+ \beta) - W(s) \right)\, ds \right)^2 \, du \end{align} if $\delta_n \to 0$ as $n \to \infty$ which happens when $b_n \to b_{\infty}=\infty$. This proves the assertion in~\eqref{swag-samplemean} for $b_{\infty}=\infty$.

For the $b_{\infty} \in \{1,2,\ldots,\}$ case, we observe that
\begin{align}
    \bar{I}_n &:=  \frac{\sigma^2}{b_n}\sum_{j=1}^{b_n}\left(\sqrt{m_n} \tilde{B}_{j,m_n} - \frac{1}{b_n} \sum_{i=1}^{b_n} \sqrt{m_n} \tilde{B}_{i,m_n} \right)^2 \nonumber \\ 
    & \overset{d}{=} \sigma^2 \frac{n}{m_n} \frac{1}{b_n} \sum_{j=1}^{b_n} \left({W((j-1)\frac{1-(m_n/n)}{b_n-1} + \frac{m_n}{n}) -  W((j-1)\frac{1-(m_n/n)}{b_n-1} + \frac{1}{n}) - }\right. \nonumber \\ & \hspace{0.5in} \left.{  \frac{1}{b_n} \sum_{i=1}^{b_n} \left(W((i-1)\frac{1-(m_n/n)}{b_n-1} + \frac{m_n}{n}) -  W((i-1)\frac{1-(m_n/n)}{b_n-1} + \frac{1}{n}) \right)  } \right)^2  \\
    & \to \sigma^2 \frac{1}{\beta} \frac{1}{b_{\infty}} \sum_{j=1}^{b_{\infty}} \left( {W((j-1)\frac{1-\beta}{b_{\infty}-1} + \beta) -  W((j-1)\frac{1-\beta}{b_{\infty}-1}) -} \right. \nonumber \\ & \hspace{2in} \left.{ \frac{1}{b_{\infty}} \sum_{i=1}^{b_{\infty}} \left(W((i-1)\frac{1-\beta}{b_{\infty}-1} + \beta) -  W((i-1)\frac{1-\beta}{b_{\infty}-1} ) \right)  }\right)^2,
\end{align} thus proving the assertion for finite $b_{\infty}$. 

Next, observe that
\begin{align} \label{asympexp3.1}
     \mathbb{E}[ \bar{I}_n] &= \frac{\sigma^2}{b_n} \sum_{j=1}^{b_n} \, \mathbb{E} \left[\sqrt{m_n} \tilde{B}_{j, m_{n}} - \frac{1}{b_n} \sum_{i=1}^{b_n} \sqrt{m_n} \tilde{B}_{i, m_{n}} \right]^2 \nonumber\\
    &= \sigma^2 \left(\frac{m_n}{b_n} \sum_{j=1}^{b_n} \mathbb{E} [\tilde{B}_{j, m_{n}}^2] - \frac{m_n}{b_n^2} \sum_{i=1}^{b_n} \sum_{j=1}^{b_n} \mathbb{E} [\tilde{B}_{i, m_{n}} \tilde{B}_{j, m_{n}}] \right). 
\end{align}

Some algebra yields 
\begin{align}
\label{asympexp3.1.1}
\frac{m_{n}}{b_{n}} \sum_{j=1}^{b_{n}} \mathbb{E}\left[\tilde{B}_{j, m_{n}}^{2}\right]=\frac{m_{n}}{b_{n}} \sum_{j=1}^{b_{n}} \frac{m_{n}-1}{m_{n}^{2}}=1, \quad j \in \{1,2,\ldots,b_n\}
\end{align}

and

\begin{align}
\label{asympexp3.1.2}
m_n \, \mathbb{E}\left[\tilde{B}_{i, m_{n}} \tilde{B}_{j, m_{n}}\right]= \left(1 - \frac{|i-j|}{b_n-1} \frac{n-m_n}{m_n}\right)^+ \quad i,j \in \{1,2,\ldots,b_n\}.
\end{align}

Plug in \eqref{asympexp3.1.1} and \eqref{asympexp3.1.2} in \eqref{asympexp3.1}, and we have

\begin{align}
\mathbb{E} [\bar{I}_n] &= \sigma^2 \left(1- \frac{1}{b_n^2} \sum_{i=1}^{b_n} \sum_{j=1}^{b_n} \left(1 - \frac{|i-j|}{b_n-1} \frac{n-m_n}{m_n}\right)^+\right) \nonumber \\
&= \sigma^2 \left(1-\frac{1}{b_n} - \frac{2}{b_n}\sum_{h=1}^{b_n-1} \left(1 - \frac{h}{b_n-1} \frac{n-m_n}{m_n}\right)^+(1-h/b_n)\right) \label{expibarn1} \\ &= \sigma^2 \kappa_2(\beta,b_{\infty}) + o(n^{-\delta}),
\label{expibarn2} \end{align} where the last equality holds by the definition of $\kappa_2(\beta,b_{\infty})$ in~\eqref{biascorrect2again} and since we have assumed that $|m_n/n - \beta| = o(n^{-\delta})$. Also, since $\Gamma$ has been assumed to have finite second moment,~\eqref{errfinbd-samplemean} implies that \begin{equation} \label{errexpzero} \mathbb{E}[\bar{E}_n] = O\left( (\beta n)^{-\delta} \sqrt{2 \log^2 \beta n \log^2 n}\right ). \end{equation} Using~\eqref{expibarn2} and~\eqref{errexpzero} in~\eqref{firstsplit-samplemean}, we conclude that for finite $b_{\infty} \in \{1,2,\ldots\}$,
\begin{align}
    \lim_{n \to \infty} \mathbb{E}[\hat{\sigma}^2_{\mbox{\tiny OB-II}}(m_n,b_n)] = \sigma^2 + O(\epsilon_{2,n}).
\end{align}

Let's next consider the $b_{\infty} = \infty$ case. We write the summation appearing in~\eqref{expibarn1} as an integral on a lattice of size $\delta_n = 1/(b_n-1)$ as follows: \begin{align}\label{integform} \MoveEqLeft \frac{2}{b_n}\sum_{h=1}^{b_n-1} \left(1 - \frac{h}{b_n-1}(\frac{n}{m_n}-1)\right)^+(1-\frac{h}{b_n}) \nonumber & \\ & \hspace{1.5in} =  2\int_0^1 \left(1 - \lfloor u \rfloor_{\delta_n} (\frac{n}{m_n}-1)\right)^+(1-\lfloor u \rfloor_{\delta_n}\frac{b_n-1}{b_n}) \, du. \end{align} Plugging~\eqref{integform} in~\eqref{expibarn1}, and since $|m_n/n - \beta| = o(n^{-\delta})$ and $b_n^{-1} = o(n^{-\delta})$, we get \begin{align}\label{binfexp} \mathbb{E}[\bar{I}_n] &= \sigma^2\left(1 - 2\int_0^1 \left(1- u(\frac{1}{\beta} - 1)\right)^+(1-u) \, du \right) + O(\epsilon_{2,n}) \nonumber \\ &= \sigma^2\left(1- 2\int_0^{\frac{\beta}{1-\beta} \wedge 1} \left(1 - u\frac{1-\beta}{\beta}\right)(1-u)\, du\right) + O(\epsilon_{2,n}) \nonumber \\
&= \sigma^2\left(1 - 2\left(\frac{\beta}{1-\beta} \wedge 1\right) +  \frac{1}{\beta}\left(\frac{\beta}{1-\beta} \wedge 1\right)^2 - \frac{2}{3}\frac{1-\beta}{\beta}\left(\frac{\beta}{1-\beta} \wedge 1\right)^3\right) + O(\epsilon_{2,n}) \nonumber \\ &=: \sigma^2 \kappa_2(\beta,\infty) + O(\epsilon_{2,n}).\end{align}  Now use~\eqref{binfexp} and~\eqref{errexpzero} in~\eqref{firstsplit-samplemean} to see that the assertion corresponding to $b_{\infty}=\infty$ also holds.

Let's now prove that the statement in~\eqref{ob2t-stat} holds. Using Assumption~\ref{ass:stronginvar} and after some algebra, we have almost surely, \begin{align}\label{stronginvarapp2} \left| \sqrt{n}\frac{\left(\bar{\theta}_n - \theta(P)\right)}{\sigma} - \frac{1}{b_n}\sum_{j=1}^{b_n} \frac{\sqrt{n}}{m_n} \left(W(nc_j + m_n) - W(nc_j) \right)\right| \leq \Gamma m_n^{-\delta-1/2}\sqrt{\log^2 m_n}, \end{align} where $\Gamma$ is a well-defined random variable with finite mean, $\delta>0,$ and $$c_j = (j-1)\frac{n-m_n}{b_n-1}.$$ Also, since $m_n \to \infty$ and $m_n/n \to \beta >0$, we see that \begin{numcases}{\label{}\frac{1}{b_n}\sum_{j=1}^{b_n} \frac{\sqrt{n}}{m_n} \left(W(nc_j + m_n) - W(nc_j) \right) \inD} \frac{1}{1-\beta}\int_0^{1-\beta} W(s+\beta) - W(s) \, ds \quad & $b_{\infty}= \infty$;\nonumber \\
\frac{1}{b_{\infty}\beta} \sum_{j=1}^{b_{\infty}} W(nc_j + m_n) - W(nc_j)  \quad & $b_{\infty} \in \mathbb{N}\setminus \{1\}$, \nonumber \end{numcases} implying along with~\eqref{stronginvarapp2} and the \rp{Slutsky's theorem} (Theorem~\ref{thm:slutsky}) that the assertion in~\eqref{ob2t-stat} holds.

\end{proof}

\end{document}